\documentclass{amsart}

\usepackage{amsmath,amssymb,amsthm,amscd}
\usepackage{bm}
\usepackage{color}
\usepackage{mathrsfs}
\usepackage{upgreek}
\usepackage[all]{xy}

\newtheorem{theorem}{Theorem}

\newtheorem{prop}[theorem]{Proposition}
\newtheorem{lemma}[theorem]{Lemma}
\newtheorem{cor}[theorem]{Corollary}

\newtheorem{rem}[theorem]{Remark}
\newtheorem{conj}[theorem]{Conjecture}
\newtheorem{propdef}[theorem]{Proposition and Definition}

\newtheorem*{thmA}{Theorem A}
\newtheorem*{thmB}{Theorem B}
\newtheorem*{thmC}{Theorem C}
\newtheorem*{thmD}{Theorem D}

\numberwithin{theorem}{section}
\numberwithin{equation}{section}
\newcommand{\W}{\mathcal{W}}
\newcommand{\Z}{\mathbb{Z}}
\newcommand{\C}{\mathbb{C}}
\newcommand{\Zc}{\mathcal{Z}}
\newcommand{\Oc}{\mathcal{O}}
\newcommand{\bQ}{{\bf Q}}
\newcommand{\g}{\mathfrak{g}}
\newcommand{\h}{\mathfrak{h}}
\newcommand{\nil}{\mathfrak{n}}
\newcommand{\rf}{{\g_0}}

\newcommand{\bo}{\mathfrak{b}}
\newcommand{\z}{\mathfrak{z}}
\newcommand{\p}{\mathfrak{p}}
\newcommand{\lf}{\mathfrak{l}}
\newcommand{\uf}{\mathfrak{u}}
\newcommand{\slf}{\mathfrak{sl}}
\newcommand{\gl}{\mathfrak{gl}}
\newcommand{\so}{\mathfrak{so}}
\newcommand{\spf}{\mathfrak{sp}}
\newcommand{\ad}{\operatorname{ad}}
\newcommand{\Ad}{\operatorname{Ad}}
\newcommand{\rootht}{\operatorname{ht}}

\newcommand{\der}{\partial}
\newcommand{\e}{\mathrm{e}}
\newcommand{\Ker}{\operatorname{Ker}}
\newcommand{\Img}{\operatorname{Im}}
\newcommand{\End}{\operatorname{End}}
\newcommand{\Hom}{\operatorname{Hom}}
\newcommand{\inv}{(\hspace{0.5mm}\cdot\hspace{0.5mm}|\hspace{0.5mm}\cdot\hspace{0.5mm})}
\newcommand{\degG}{\deg_\Gamma}
\newcommand{\degQ}{\deg_\bQ}

\newcommand{\conf}{\Delta}
\newcommand{\charge}{\deg_\mathrm{ch}}
\newcommand{\Vtau}{V^{\tau_{k}}(\g_0)}
\newcommand{\Fch}{F_{\mathrm{ch}}(\g_{>0})}
\newcommand{\FchT}{F_{\mathrm{ch}}^T(\g_{>0})}

\newcommand{\Fchn}{F_{\mathrm{ch}}^{\hspace{0.5mm}n}}
\newcommand{\Fne}{\Phi(\g_{\frac{1}{2}})}
\newcommand{\FneT}{\Phi^T(\g_{\frac{1}{2}})}
\newcommand{\FneU}{\Phi^U(\g_{\frac{1}{2}})}
\newcommand{\FneF}{\Phi^F(\g_{\frac{1}{2}})}
\newcommand{\Hi}{\mathcal{H}}

\newcommand{\D}{\mathcal{D}}
\newcommand{\A}{\mathcal{A}}
\newcommand{\Azero}{\mathcal{A}_{\Delta_0^+}}
\newcommand{\Am}{\mathcal{A}_{\Delta_{>0}}}
\newcommand{\Anil}{\mathcal{A}_{\Delta_+}}

\newcommand{\Deltazero}{\Delta_{0}^{+}}
\newcommand{\Deltahalf}{\Delta_{\frac{1}{2}}}
\newcommand{\Deltaone}{\Delta_{1}}
\newcommand{\dst}{d_{\mathrm{st}}}
\newcommand{\dne}{d_{\mathrm{ne}}}
\newcommand{\dchi}{d_{\chi}}
\newcommand{\Py}{\pi}
\newcommand{\row}{\operatorname{row}}
\newcommand{\col}{\operatorname{col}}

\newcommand{\cop}{\operatorname{\Updelta}}
\newcommand{\para}{{\bf k}}

\newcommand{\ind}{\operatorname{Ind}}
\newcommand{\Ind}{\operatorname{\mathbb{I}nd}}
\newcommand{\VtauT}{V^{T}(\g_0)}

\newcommand{\killing}{\kappa^\circ}
\newcommand{\Id}{\operatorname{Id}}
\newcommand{\prin}{\mathrm{prin}}

\newcommand{\Zhu}{\operatorname{Zhu}}

\newcommand{\fne}{\bar{\Phi}(\g_{\frac{1}{2}})}

\newcommand{\GL}{\mathrm{GL}}

\newcommand{\Wak}{W}

\newcommand{\gr}{\operatorname{gr}}
\newcommand{\hQ}{\widetilde{Q}}

\newcommand{\nilcurrent}{\operatorname{\g_{>0}}[\![\hspace{0.2mm}t\hspace{0.2mm}]\!]}
\newcommand{\Ncurrent}{G_{>0}[\![\hspace{0.2mm}t\hspace{0.2mm}]\!]}
\newcommand{\Njet}{J_\infty G_{>0}}
\newcommand{\Weyl}{\mathbb{V}}

\title{Screening operators and Parabolic inductions for Affine $\W$-algebras}
\author{Naoki Genra}
\address[N.G.]{Department of Mathematical and Statistical Sciences, University of Alberta, 632 CAB, Edmonton, Alberta, Canada T6G 2G1}
\email{genra@ualberta.ca}

\begin{document}
\renewcommand{\thefootnote}{\fnsymbol{footnote}}
\footnote[0]{
This work was supported by Grant-in-Aid for JSPS Fellows (No.17J07495) and the Research Institute for Mathematical Sciences, a Joint Usage/Research Center located in Kyoto University.
}
\renewcommand{\thefootnote}{\arabic{footnote}}
\maketitle
\markboth{}{}

\begin{abstract}
(Affine) $\W$-algebras are a family of vertex algebras defined by the generalized Drinfeld-Sokolov reductions associated with a finite-dimensional reductive Lie algebra $\g$ over $\C$, a nilpotent element $f$ in $[\g,\g]$, a good grading $\Gamma$ and a symmetric invariant bilinear form $\kappa$ on $\g$. We introduce free field realizations of $\W$-algebras by using Wakimoto representations of affine Lie algebras, where $\W$-algebras are described as the intersections of kernels of screening operators. We call these Wakimoto free fields realizations of $\W$-algebras. As applications, under certain conditions that are valid in all cases of type $A$, we construct parabolic inductions for $\W$-algebras, which we expect to induce the parabolic inductions of finite $\W$-algebras defined by Premet and Losev. In type $A$, we show that our parabolic inductions are a chiralization of the coproducts for finite $\W$-algebras defined by Brundan-Kleshchev. In type $BCD$, we are able to obtain some generalizations of the coproducts in some special cases. This paper also contains an appendix by Shigenori Nakatsuka on the compatibility of screening operators with Miura maps.
\end{abstract}

\section{Introduction}\label{Introduction sec}
Let $\g$ be a reductive Lie algebra, $f$ a nilpotent element in $[\g,\g]$, $\kappa$ a symmetric invariant bilinear form on $\g$ and
\begin{align*}
\Gamma\colon\g=\bigoplus_{j\in\frac{1}{2}\Z}\g_j
\end{align*}
a good grading on $\g$ for $f$. We associate with the (affine) $\W$-algebra $\W^\kappa(\g,f;\Gamma)$ that is a $\frac{1}{2}\Z_{\geq0}$-graded conformal vertex algebra defined by means of the (generalized) Drinfeld-Sokolov reduction \cite{FF4, KRW}. The vertex algebra structure of $\W$-algebras doesn't depend on the choice of the good grading $\Gamma$ for fixed $\g,f,\kappa$, although the conformal grading does \cite{BG, AKM}.

In this paper, we construct inclusions
\begin{align*}
\Ind^\g_\lf\colon\W^\kappa(\g,f;\Gamma)\rightarrow\W^{\kappa_\lf}(\lf,f_\lf;\Gamma_\lf)
\end{align*}
for Levi subalgebras $\lf$ of $\g$, nilpotent elements $f_\lf$ in $[\lf,\lf]$ and good gradings $\Gamma_\lf$ on $\lf$ for $f_\lf$ that satisfy some conditions. We call the maps $\Ind^\g_\lf$ parabolic inductions of $\W$-algebras. We expect that our construction gives a chiralization of the parabolic induction for finite $\W$-algebras defined by Premet \cite{P} and Losev \cite{L}. In the case of $\g=\gl_N$, we show that these inclusions induce exactly the coproducts of the finite $\W$-algebras of Brundan-Kleshchev \cite{BK2}. In the case of $\g=\so_N,\spf_N$ with rectangular nilpotent elements, we obtain a generalization of the coproducts of the corresponding finite $\W$-algebras.

To state our results more precisely, let $\Pi$ be a set of simple roots of $\g$ compatible with $\Gamma$, $\Pi_j$ the subset of $\Pi$ consisting of simple roots whose root vectors belong to $\g_j$ for $j\in\frac{1}{2}\Z$. Then $\Pi=\Pi_0\sqcup\Pi_{\frac{1}{2}}\sqcup\Pi_1$ by \cite{EK}. According to Lusztig and Spaltenstein \cite{LS}, for any Levi subalgebra $\lf$ of $\g$, each nilpotent orbit $\Oc_\lf$ in $\lf$ defines a nilpotent orbit
\begin{align*}
\Oc_\g=\ind^\g_\lf\Oc_\lf\quad\mathrm{in}\ \g,
\end{align*}
which is called the induced nilpotent orbit of $\Oc_\lf$. Let $\Pi_\lf\subset\Pi$ be the set of simple roots of $\lf$, $G$ a connected Lie group corresponding to $\g$ and $L$ the Lie subgroup of $G$ such that $\mathrm{Lie}(L)=\lf$.

\begin{lemma}[Lemma \ref{ind good lem}]\label{intro lem}
Suppose that a good grading $\Gamma$ satisfies $\Pi\backslash\Pi_\lf\subset\Pi_1$. Then $\Oc_\g=G\cdot f$ is induced from $\Oc_\lf=L\cdot f_\lf$ for a nilpotent element $f_\lf$ in $[\lf,\lf]$. Moreover, the restriction $\Gamma_\lf$ of $\Gamma$ to $\lf$ is a good grading on $\lf$ for $f_\lf$.
\end{lemma}

We note that the existence of $\Gamma$ in Lemma \ref{intro lem} is valid in all cases of type $A$ (\cite{Kr, OW}), in the cases of rectangular nilpotent elements in type $BCD$ (\cite{Ke,Sp}), and in all cases of type $G$ (\cite{EK,GE}). However, there exist some induced nilpotent orbits $\Oc_\g=\ind^\g_\lf\Oc_\lf$ in type $E$ and $F$ such that no good grading on $\g$ satisfies that $\Pi\backslash\Pi_\lf\subset\Pi_1$, see \cite{EK, GE}.

\begin{thmA}[Theorem \ref{induced thm}, Proposition \ref{induced prop}]
Suppose that a good grading $\Gamma$ satisfies the condition that $\Pi\backslash\Pi_\lf\subset\Pi_1$.
\begin{enumerate}
\item For any symmetric invariant bilinear form $\kappa$ on $\g$, there exists an injective vertex algebra homomorphism
\begin{align*}
\Ind^\g_\lf\colon\W^\kappa(\g,f;\Gamma)\rightarrow\W^{\kappa_\lf}(\lf,f_\lf;\Gamma_\lf),
\end{align*}
where $f_\lf$, $\Gamma_\lf$ are given in Lemma \ref{intro lem}, $\kappa_\lf=\kappa+\frac{1}{2}\killing_\g-\frac{1}{2}\killing_\lf$, and $\killing_\g$, $\killing_\lf$ are the Killing forms on $\g$, $\lf$ respectively.
\item $\Ind^\g_\lf$ is a unique vertex algebra homomorphism that satisfies $\mu_\kappa=\mu_{\lf, \kappa_\lf} \circ \Ind^\g_\lf$, where $\mu_\kappa$, $\mu_{\lf, \kappa_\lf}$ are the Miura maps \cite{KW1} for $\W^\kappa(\g,f;\Gamma)$, $\W^{\kappa_\lf}(\lf,f_\lf;\Gamma_\lf)$ respectively.
\item Let $\lf'$ be any Levi subalgebra of $\g$ such that $\Pi\backslash\Pi_{\lf'}\subset\Pi_1$ and $\lf\subset\lf'\subset\g$. Then the maps $\Ind^\g_{\lf'}$, $\Ind^{\lf'}_{\lf}$ exist and $\Ind^\g_\lf=\Ind^{\lf'}_\lf\circ\Ind^\g_{\lf'}$.
\end{enumerate}
\end{thmA}

See Section \ref{sec:Miura maps} for the definition of the Miura map. In the case that $f$ is a principal nilpotent element, the map $\Ind^\g_\lf$ has been constructed in Theorem B 7.1 of \cite{BFN}.

For any $\frac{1}{2}\Z_{\geq0}$-graded conformal vertex algebra $V$, we can associate with an associative algebra $\Zhu(V)$, called the (twisted) Zhu algebra \cite{Z, FZ, DK}. It is proved in \cite{A1, DK} that $\Zhu(\W^\kappa(\g,f;\Gamma))$ is the finite $\W$-algebra associated with $\g,f,\Gamma$ \cite{P1, GG}, which we denote by $U(\g,f;\Gamma)$. It is easy to see that any vertex algebra homomorphism $\alpha\colon V\rightarrow W$ induces an algebra homomorphism between the Zhu algebras, which we denote by $\Zhu(\alpha)\colon\Zhu(V)\rightarrow\Zhu(W)$. For an algebra homomorphism $A$, we call a map $\alpha$ a chiralization of the map $A$ if $A=\Zhu(\alpha)$. In the case of $\alpha=\Ind^\g_\lf$, we obtain an algebra homomorphism
\begin{align*}
\Zhu(\Ind^\g_\lf)\colon U(\g,f;\Gamma)\rightarrow U(\lf,f_\lf;\Gamma_\lf),
\end{align*}
which is a unique injective algebra homomorphism that satisfies $\bar{\mu}=\bar{\mu}_\lf\circ\Zhu(\Ind^\g_\lf)$, where $\bar{\mu}$, $\bar{\mu}_\lf$ are the Miura maps for $U(\g,f;\Gamma)$, $U(\lf,f_\lf;\Gamma_\lf)$ respectively (Lemma \ref{Zhu ind lem}). See \cite{Ly} or Section \ref{chiral sec} for the definition of the Miura map for $U(\g,f;\Gamma)$.

Given an induced nilpotent orbit $G\cdot f=\ind^\g_\lf (L\cdot f_\lf)$ in $\g$ with a good grading $\Gamma$ on $\g$ for $f$ and a good grading $\Gamma_\lf$ on $\lf$ for $f_\lf$, Losev proved the existence of an injective algebra homomorphism
\begin{align}\label{intro Losev eq}
U(\g,f;\Gamma)\rightarrow\widetilde{U}(\lf,f_\lf;\Gamma_\lf)
\end{align}
in \cite{L}, where $\widetilde{U}(\lf,f_\lf;\Gamma_\lf)$ is a certain completion of $U(\lf,f_\lf;\Gamma_\lf)$. The map \eqref{intro Losev eq} induces a functor from the category of $U(\lf,f_\lf;\Gamma_\lf)$-modules to the category of $U(\g,f;\Gamma)$-modules, called the parabolic induction that was first introduced by Premet \cite{P}. We conjecture that $\Zhu(\Ind^\g_\lf)$ coincides with \eqref{intro Losev eq}, and this is the reason why we call the map $\Ind^\g_\lf$ the parabolic induction of $\W$-algebras.

In the case of $\gl_N$, any nilpotent element in $\slf_N=[\gl_N,\gl_N]$ admits a good $\Z$-grading. These good $\Z$-gradings on $\gl_N$ are classified by combinatoric objects called (even) pyramids $\Py$ introduced in \cite{EK}, which are sequences of the columns of $1\times1$ boxes such that each of rows in $\Py$ is a single connected strip (see Section \ref{coproduct sec} for details). For a pyramid $\Py$ consisting of $N$ boxes, we associate with a nilpotent element $f_\Py$ in $\gl_N$, a good $\Z$ grading $\Gamma_\Py$ on $\gl_N$ for $f_\Py$, and the finite $\W$-algebra $U(\gl_N,\Py)=U(\gl_N,f_\Py;\Gamma_\Py)$. It was shown by Brundan and Kleshchev in \cite{BK2} that $U(\gl_N,\Py)$ is isomorphic to a truncation of the Yangian $Y(\gl_n)$ for some $n\geq1$ and the coproduct of $Y(\gl_n)$ induces an injective algebra homomorphism between finite $\W$-algebras
\begin{align*}
\bar{\cop}=\bar{\cop}^\Py_{\Py_1, \Py_2}\colon U(\gl_N,\Py)\rightarrow U(\gl_{N_1},\Py_1)\otimes U(\gl_{N_2},\Py_2)
\end{align*}
for a pyramid $\Py$ that splits into sum of $\Py_1$ and $\Py_2$ along a column of $\Py$ (see e.g. Section \ref{cop thm sec}), which we denote by $\Py=\Py_1\oplus\Py_2$. This map $\bar{\cop}$ is called a coproduct of finite $\W$-algebras and satisfies the coassociativity, i.e.
\begin{align*}
(\Id\otimes\bar{\cop}^{\Py_2\oplus\Py_3}_{\Py_2,\Py_3})\circ\bar{\cop}^\Py_{\Py_1,\Py_2\oplus\Py_3}=(\bar{\cop}^{\Py_1\oplus\Py_2}_{\Py_1,\Py_2}\otimes\Id)\circ\bar{\cop}^\Py_{\Py_1\oplus\Py_2, \Py_3}
\end{align*}
for a pyramid $\Py=\Py_1\oplus\Py_2\oplus\Py_3$. The coproduct $\bar{\cop}$ plays a fundamental role to produce representations of finite $\W$-algebras of type $A$, see \cite{BK3}.

Consider a maximal Levi subalgebra $\lf$ in $\gl_N$, that is, $\lf=\gl_{N_1}\oplus\gl_{N_2}$ for some $N_1,N_2\in\Z_{\geq1}$ such that $N=N_1+N_2$. According to \cite{Kr, OW}, it follows that any induced nilpotent orbit in $\gl_N$ takes the form
\begin{align*}
\GL_N\cdot f_\Py=\ind^\g_\lf(\GL_{N_1}\cdot f_{\Py_1}+\GL_{N_2}\cdot f_{\Py_2})
\end{align*}
for some pyramid $\Py=\Py_1\oplus\Py_2$, where $f_{\Py_1}\in\gl_{N_1}$ and $f_{\Py_2}\in\gl_{N_2}$. Therefore, it is expected that $\bar{\cop}$ coincides with the special case of \eqref{intro Losev eq} for $\g=\gl_N$ and $\lf=\gl_{N_1}\oplus\gl_{N_2}$.
\smallskip

For $k\in\C$, let us denote by $\W^k(\gl_N,\Py)=\W^\kappa(\gl_N,f_\Py;\Gamma_\Py)$, where $\kappa$ is a symmetric invariant bilinear form on $\gl_N$ such that $\kappa(u|v)=k\ \mathrm{tr}(uv)$ for all $u,v\in\slf_N$. The following assertion is obtained from Theorem A.

\begin{thmB}[Theorem \ref{coproduct thm}, Proposition \ref{chiral prop}]
Let $\Py$ be a pyramid consisting of $N$ boxes such that $\Py=\Py_1\oplus\Py_2$.
\begin{enumerate}
\item For any $k\in\C$, there exists an injective vertex algebra homomorphism
\begin{align*}
\cop=\cop^\Py_{\Py_1,\Py_2}\colon\W^k(\gl_N,\Py)\rightarrow\W^{k_1}(\gl_{N_1},\Py_1)\otimes\W^{k_2}(\gl_{N_2},\Py_2),
\end{align*}
where $k+N=k_1+N_1=k_2+N_2$ and $N_i$ is a number of boxes in $\Py_i$ for $i=1,2$.
\item $\cop$ is a unique vertex algebra homomorphism that satisfies $\mu_k=(\mu_{1, k_1}\otimes\mu_{2, k_2})\circ\cop$, where $\mu_k$, $\mu_{1, k_1}$, $\mu_{2, k_2}$ are the Miura maps for $\W^k(\gl_N,\Py)$, $\W^{k_1}(\gl_{N_1},\Py_1)$, $\W^{k_2}(\gl_{N_2},\Py_2)$ respectively.
\item $\cop$ is coassociative, i.e. $(\Id\otimes\cop^{\Py_2\oplus\Py_3}_{\Py_2,\Py_3})\circ\cop^\Py_{\Py_1,\Py_2\oplus\Py_3}=(\cop^{\Py_1\oplus\Py_2}_{\Py_1,\Py_2}\otimes\Id)\circ\cop^\Py_{\Py_1\oplus\Py_2, \Py_3}$ for $\Py=\Py_1\oplus\Py_2\oplus\Py_3$.
\item $\cop$ is a chiralization of $\bar{\cop}$, that is, $\Zhu(\cop)=\bar{\cop}$.
\end{enumerate}
\end{thmB}

See Section \ref{examples sec} for some examples of $\cop$. In the case that $f_\Py$ is a principal nilpotent element, the coproduct $\cop$ is an injective map
\begin{align}\label{intro prin eq}
\W_N^k\rightarrow\W_{N_1}^{k_1}\otimes\W_{N_2}^{k_2}
\end{align}
for $N=N_1+N_2$ and $k+N=k_1+N_1=k_2+N_2$, where $\W_N^k$ is the $\W$-algebra of $\gl_N$ with a principal nilpotent element and level $k$ \cite{Za,FL}. It seems  that the existence of the map \eqref{intro prin eq} has been suggested in \cite{FiSt}.

In the case of $\g_N=\so_N$ or $\spf_{N}$, any maximal Levi subalgebra of $\g_N$ takes the form $\lf=\gl_{N_1}\oplus\g_{N_2}$ for some $N_1,N_2\in\Z_{\geq1}$ such that $N=2N_1+N_2$. Applying Theorem A to this setting with rectangular nilpotent elements, we obtain some generalizations of the coproducts for $\W$-algebras of $\g_N$. See Theorem \ref{cop BCD thm} for precise statements. We note that our results suggest the existence of certain coproducts for truncated twisted Yangians, which is obtained as Corollary \ref{cop BCD cor} in some special cases. We refer to \cite{R, Br} for connections between twisted Yangians and finite $\W$-algebras of type $BCD$.

The basic tool for the proof of Theorem A is Wakimoto representations of $\W$-algebras, which we introduce in Section \ref{W-alg Wak sec}. For simplicity, we assume that $\g$ is a simple Lie algebra. Denote by $\W^k(\g,f;\Gamma)=\W^\kappa(\g,f;\Gamma)$ if $k=\kappa(\theta|\theta)/2$ for the highest root $\theta$ of $\g$. Wakimoto representations of the affine Lie algebra $\widehat{\g}$ are introduced by Wakimoto \cite{Wak} in the case of $\widehat{\slf_2}$ and Feigin-Frenkel \cite{FF1, FF2, FF3, FF5, F} in general case, see also Section \ref{aff Wak sec}. The actions of $\widehat{\g}$ on Wakimoto representations are induced from an embedding of the affine vertex algebra $V^k(\g)$ into the tensor product of the Heisenberg vertex algebra $\Hi$ associated with a Cartan subalgebra $\h$ in $\g$ and $\dim\nil_+$ copies of the $\beta\gamma$-system, where $\nil_+=\mathrm{Lie}(N_+)$ and $N_+$ is the big cell of the flag manifold $G/B_-$. The image of this embedding is the intersection of kernels of screening operators $S_\alpha$ for all $\alpha\in\Pi$ if $k$ is a formal parameter (\cite{F}). As explained in detail in Section \ref{W-alg Wak sec}, applying Drinfeld-Sokolov reductions to Wakimoto representations of $\widehat{\g}$, we obtain free fields realizations of $\W$-algebras $\W^k(\g,f;\Gamma)$, which we call Wakimoto free fields realizations of $\W$-algebras $\W^k(\g,f;\Gamma)$.

Let $U = \C[\para]$ be the polynomial ring with a formal parameter $\para$ and $F = \C(\para)$ be the quotient field. Set $T = U$ or $F$. When the base ring (or field) is $T$, we replace everywhere the complex number $k$ by a formal parameter $\para$, and denote the corresponding $\W$-algebra and Heisenberg vertex algebra by $\W^T(\g,f;\Gamma)$, $\Hi^T$ instead of $\W^k(\g,f;\Gamma)$, $\Hi$ respectively. Let $\mathcal{N}$ be the nilpotent cone of $\g$, $\mathcal{S}_f$ the Slodowy slice of $\g$ through $f$.

\begin{thmC}[Corollary \ref{main cor}]
The $\W$-algebras $\W^T(\g,f;\Gamma)$ over $T$ may be embedded into the tensor products of $\Hi^T$ and $\frac{1}{2}\dim(\mathcal{N}\cap\mathcal{S}_f)$ copies of the $\beta\gamma$-system. These image can be identified with the intersections of kernels of screening operators $Q_\alpha$ induced by $S_\alpha$.
\end{thmC}

See Theorem \ref{main thm} for the precise formulae of $Q_\alpha$. In the case that $f$ is a principal nilpotent element, screening operators $Q_\alpha$ coincide with the ones constructed in \cite{FL, FF3}. In the case that $\g_0$ is a Cartan subalgebra $\h$, screening operators $Q_\alpha$ coincide with the ones constructed in \cite{G}.

Our strategy to prove Theorem A is simple. Under the assumption in Theorem A, we consider the specialization of inclusion maps
\begin{align}\label{intro screening eq}
\bigcap_{\alpha\in\Pi}\Ker Q_\alpha\hookrightarrow\bigcap_{\alpha\in\Pi_\lf}\Ker Q_\alpha\hookrightarrow\bigcap_{\alpha\in\Pi_0}\Ker Q_\alpha.
\end{align}
We show by using Theorem C that the first map is nothing but $\Ind^\g_\lf$, i.e.
\begin{align*}
\W^{\kappa_\lf}(\lf,f_\lf;\Gamma_\lf)\simeq\bigcap_{\alpha\in\Pi_\lf}\Ker Q_\alpha.
\end{align*}
Our assumption ($\Pi\backslash\Pi_\lf\subset\Pi_1$) is used here. Since the Miura map is injective by \cite{F,A,G}, Theorem A therefore follows if we show that the map $\Ind^\g_\lf$ satisfies the formula $\mu_\kappa=\mu_{\lf, \kappa_\lf} \circ\Ind^\g_\lf$, which in fact follows from \eqref{intro screening eq} and Theorem D.

\begin{thmD}[Corollary \ref{Miura cor}]
The specialization
\begin{align*}
\mu_k\colon\W^k(\g,f;\Gamma)\rightarrow\Vtau\otimes\Fne
\end{align*}
of an inclusion map
\begin{align*}
\bigcap_{\alpha\in\Pi}\Ker Q_\alpha\hookrightarrow\bigcap_{\alpha\in\Pi_0}\Ker Q_\alpha.
\end{align*}
coincides with the Miura map.
\end{thmD}

We include an appendix due to Shigenori Nakatsuka concerning the relationship between $Q_\alpha$ and screening operators $\widetilde{Q}_\alpha$ given in \cite{G}. The main tool is a commutative diagram between the Wakimoto resolution and a parabolic Wakimoto resolution of $V^F(\g)$. Using Drinfeld-Sokolov functors, the diagram implies a commutative one between the Wakimoto resolutions of $\W^F(\g, f;\Gamma)$ containing $Q_\alpha$ and a parabolic Wakimoto resolution of $\W^F(\g,f;\Gamma)$ containing $\widetilde{Q}_\alpha$. Since $\widetilde{Q}_\alpha$ is compatible with the Miura map in the sense of Lemma \ref{lem:Miura F}, we obtain a lemma (Lemma \ref{Miura lemma}) on the compatibility of $Q_\alpha$ with the Miura map for $\W^F(\g, f;\Gamma)$, which derives Theorem D.

Let us make some comment on the relationship between $\W^k(\gl_N,\Py)$ and the affine Yangian $Y(\widehat{\gl}_n)$. In the case that $f$ is a principal nilpotent element, an action of $Y(\widehat{\gl}_1)$ on $\W^k_N$ was first suggested by Aldey-Gaiotto-Tachikawa \cite{AGT} and was studied by Maulik-Okounkov \cite{MO} and Schiffmann-Vasserot \cite{SV}, see also \cite{BFN} for the generalizations. The coproduct \eqref{intro prin eq} is expected to be induced by the coproduct of $Y(\widehat{\gl}_1)$ as an analogue of the finite cases. We hope to study the relationship between the coproduct $\cop$ in Theorem B and that of affine Yangian $Y(\widehat{\gl}_n)$ in our future works.

The paper is organized as follows. In Section \ref{sec:W-alg-def}, we review the definitions of $\W$-algebras. In Section \ref{W-alg T sec}, we introduce the $\W$-algebras over $T$. In Section \ref{sec:Miura maps}, we recall Miura maps. In Section \ref{parabolic screening operator}, we recall results in \cite{G}. In Section \ref{aff Wak sec}, we recall Wakimoto representations, Wakimoto resolutions and screening operators $S_\alpha$ of $V^T(\g)$. In Section \ref{local sec}, we define special coordinates on $N_+$ that give some restrictions on $S_\alpha$ and are always chosen in the rest of paper. In Section \ref{semi-inf sec}, we recall the semi-regular bimodule theory developed in \cite{ACL} and derive the vanishing of semi-infinite cohomologies of Wakimoto representations of $V^T(\g)$. In Section \ref{Wak sec}, we introduce Wakimoto representations, Wakimoto resolutions and screening operators $Q_\alpha$ of $\W$-algebras over $T$ by applying to those of $V^T(\g)$ the Drinfeld-Sokolov functors and results on vanishing of cohomologies. In Section \ref{Explicit forms of screenings}, we prove Theorem \ref{main thm}. In Section \ref{sec:compati Miura}, we state Lemma \ref{Miura lemma} and prove Theorem C and Theorem D by using Lemma \ref{Miura lemma} and Theorem \ref{main thm}. In Section \ref{reductive W-alg sec}, we define $\W$-algebras associated with reductive Lie algebras and conclude some results from Theorem C and Theorem D. In Section \ref{ind nil sec}, we recall the definitions and properties of induced nilpotent orbits. In Section \ref{pre sec}, we prepare some preliminary results in order to prove Theorem A. In Section \ref{induced sec}, we prove Theorem A. In Section \ref{chiral sec}, we derive some results for finite $\W$-algebras from Theorem A. In Section \ref{pyramid sec}, we recall the definitions of pyramids. In Section \ref{cop thm sec}, we prove Theorem B. In Section \ref{cop BCD sec}, we derive some generalizations of the coproducts for the $\W$-algebras of type $BCD$ from Theorem A. In Section \ref{examples sec}, we give examples of Theorem B in the case that $f$ is a principal, rectangular and subregular nilpotent element. In Appendix \ref{appendix}, S. Nakatsuka presents a relationship between the Wakimoto resolution and a parabolic Wakimoto resolution of $V^F(\g)$ by using Fiebig's equivalence and results in \cite{Ku}, and proves Proposition \ref{coincidence with Miura map}, which implies Lemma \ref{Miura lemma}.

\vspace{3mm}

{\it Acknowledgments}\quad N.G. is grateful to his supervisor Tomoyuki Arakawa for valuable discussions and lots of advice to improve this paper. He wishes to express his gratitude to Boris Feigin for sharing his ideas and useful discussions for Theorem B. He is grateful to Hiraku Nakajima for valuable comments, which triggered Theorem A. Some part of this work was done while N.G. was visiting M.I.T. and National Research University Higher School of Economics in 2016, and N.G. and S.N. were visiting National Cheng Kung University in 2019. They are grateful to those institutes for their hospitality.

\section{Affine $\W$-algebras}\label{W-alg sec}

\subsection{Definitions of $\W$-algebras}\label{sec:W-alg-def}
We recall the definitions of the (affine) $\W$-algebras, following \cite{KRW}. Let $\g$ be a finite-dimensional simple Lie algebra over $\C$, $f$ a nilpotent element of $\g$ and $\Gamma$ a good grading of $\g$ for $f$ denoted by
\begin{align*}
\Gamma\colon\g=\bigoplus_{j\in\frac{1}{2}\Z}\g_{j},
\end{align*}
where the $\frac{1}{2}\Z$-grading $\Gamma$ is called good for $f$ if $[\g_i, \g_j]\subset\g_{i+j}$ for all $i,j\in\frac{1}{2}\Z$, $f\in\g_{-1}$ and $\ad f\colon\g_{j}\rightarrow\g_{j-1}$ is injective for $j\geq\frac{1}{2}$, surjective for $j\leq\frac{1}{2}$. Then there exists a semisimple element $h\in\g$ such that the grading $\Gamma$ of $\g$ is the eigenspace decomposition of $\ad(\frac{1}{2}h)$. By Jacobson-Morozov Theorem, there exists an $\slf_2$-triple $(e,h,f)$ in $\g$, and $\ad(\frac{1}{2}h)$ defines a $\frac{1}{2}\Z$-grading on $\g$, which is good for $f$ called the {\em Dynkin grading}. Choose the Cartan subalgebra $\h$ containing $h$ so that $\h\subset\g_0$. Let $\Delta$ be the set of roots, $\Delta_{+}$ the set of positive roots such that $\bigoplus_{\alpha\in\Delta_{+}}\g_{\alpha}\subset\g_{\geq0}$, where $\g_{\alpha}$ is the root space of $\alpha\in\Delta$. Let $\Pi$ be the set of simple roots, $\Delta_{j}=\{\alpha\in\Delta\mid\g_{\alpha}\subset\g_{j}\}$ and $\Pi_{j}=\Pi\cap\Delta_{j}$ for all $j\in\frac{1}{2}\Z$. Set $\Deltazero=\Delta_0\cap\Delta_+$. Then
\begin{align*}
\Delta=\bigsqcup_{j\in\frac{1}{2}\Z}\Delta_{j},\quad
\Delta_{+}=\Deltazero\sqcup\bigsqcup_{j>0}\Delta_j,\quad
\Pi=\Pi_{0}\sqcup\Pi_{\frac{1}{2}}\sqcup\Pi_{1},
\end{align*}
see \cite{EK}. Denote by $\degG\alpha=j$ if $\alpha\in\Delta_j$. Fix a root vector $e_{\alpha}\in\g$ for each $\alpha\in\Delta$ and a non-degenerate symmetric invariant bilinear form $\inv$ on $\g$ such that $(\theta|\theta)=2$ for the highest root $\theta$ of $\g$. Then $\killing_\g(u|v)=2h^{\vee}(u|v)$ for all $u,v\in\g$, where $\killing_{\g}$ is the Killing form on $\g$ and $h^{\vee}$ is the dual Coxeter number of $\g$. Let $\chi\colon\g\rightarrow\C$ be a linear map defined by $\chi(u)=(f|u)$ for $u\in\g$. Denote by $\nil_{\pm}=\bigoplus_{\alpha\in\Delta_{\pm}}\g_{\alpha}$ and $\bo_{\pm}=\h\oplus\nil_{\pm}$.

We follow \cite{FBZ, Ka} for the definitions of vertex algebras. We use the following notations:
\begin{align*}
A(z)=\sum_{n\in\Z}A_{(n)}z^{-n-1},\quad\int A(z)\ dz=A_{(0)}
\end{align*}
for any field $A(z)$, and $\delta(z-w)=\sum_{n\in\Z}z^{-n-1}w^n$. Denote by $|0\rangle$ the vacuum vector, by $:A(z)B(z):$ the normally ordered products and by $A(z)B(w)\sim\sum_{n\geq0}\frac{C_n(w)}{(z-w)^{n+1}}$ the operator product expansion for local fields $A(z), B(z)$, where $[A(z),B(w)]=\sum_{n\geq0}\frac{1}{n!}C_n(w)\der_w^n\delta(z-w)$. If $A(z),B(z)$ are fields on a vertex (super)algebra, $C_n(z)=(A_{(n)}B)(z)$ for $n\geq0$. For $k\in\C$, let $V^k(\g)$ be the affine vertex algebra associated with $\g$ of level $k$, whose generating fields $u(z)$ for $u\in\g$ satisfy
\begin{align*}
u(z)v(w)\sim\frac{[u,v](w)}{z-w}+\frac{k(u|v)}{(z-w)^2}
\end{align*}
for all $u,v\in\g$. Let $\Fch$ be the charged fermion vertex superalgebra associated with $\g_{>0}$, whose generating odd fields $\varphi_{\alpha}(z),\varphi^{\alpha}(z)$ for $\alpha\in\Delta_{>0}$ satisfy
\begin{align*}
\varphi_{\alpha}(z)\varphi^{\beta}(w)\sim\frac{\delta_{\alpha,\beta}}{z-w},\quad
\varphi_{\alpha}(z)\varphi_{\beta}(w)\sim0\sim\varphi^{\alpha}(z)\varphi^{\beta}(w)
\end{align*}
for all $\alpha,\beta\in\Delta_{>0}$. The charged decomposition $\Fch=\bigoplus_{n\in\Z}\Fchn$ is defined by the {\em charged} degree $\charge(\varphi_{\alpha}(z))=-1$ and $\charge(\varphi^{\alpha}(z))=1$ for all $\alpha\in\Delta_{>0}$, where $\Fchn=\{A\in\Fch\mid\charge(A)=n\}$. Let $\Fne$ be the neutral vertex algebra associated with $\g_{\frac{1}{2}}$, whose generating (even) fields $\Phi_{\alpha}(z)$ for $\alpha\in\Deltahalf$ satisfy
\begin{align*}
\Phi_{\alpha}(z)\Phi_{\beta}(w)\sim\frac{\chi([e_{\alpha},e_{\beta}])}{z-w}
\end{align*}
for all $\alpha,\beta\in\Deltahalf$. Set
\begin{align*}
C_{k}=V^k(\g)\otimes\Fch\otimes\Fne
\end{align*}
and $d=\int d(z)\ dz$, where $d(z)=\dst(z)+\dne(z)+\dchi(z)$ is an odd field on $C_k$ defined by
\begin{align*}
\dst(z)&=\sum_{\alpha\in\Delta_{>0}}:e_{\alpha}(z)\varphi^{\alpha}(z):-\frac{1}{2}\sum_{\alpha,\beta,\gamma\in\Delta_{>0}}c_{\alpha,\beta}^{\gamma}:\varphi_{\gamma}(z)\varphi^{\alpha}(z)\varphi^{\beta}(z):,\\
\dne(z)&=\sum_{\alpha\in\Deltahalf}:\varphi^{\alpha}(z)\Phi_{\alpha}(z):,\quad
\dchi(z)=\sum_{\alpha\in\Deltaone}\chi(e_{\alpha})\varphi^{\alpha}(z),
\end{align*}
where $c_{\alpha,\beta}^{\gamma}\in\C$ is the structure constant for $\alpha,\beta,\gamma\in\Delta_{>0}$. The charged decomposition $C_k^\bullet=V^k(\g)\otimes\Fch^\bullet\otimes\Fne$ is induced from that of $\Fch$. Since $d^2=0$ and $d\cdot C_k^p\subset C_k^{p+1}$, an odd vertex operator $d$ defines a differential of a cochain complex on $C_k$. The $\W$-algebra $\W^k(\g,f;\Gamma)$ associated with $\g,f,k,\Gamma$ is defined as the BRST cohomology of the complex $(C_k,d)$:
\begin{align*}
\W^k(\g,f;\Gamma)=H(C_k,d),
\end{align*}
called the (generalized) Drinfeld-Sokolov reduction. There exists a decomposition of the complex $C_k=C_{-}\otimes C_+$ such that $H(C_{-},\C)=\C$ and $C_+$ has only non-negative charged degree. Moreover $\W^k(\g,f;\Gamma)=H^{0}(C_+,d)$ (\cite{KW1}). A vertex algebra structure on $\W^k(\g,f;\Gamma)$ is induced from that of $C_k$ and does not depend on the choice of $\Gamma$ \cite{BG, AKM}. A conformal $\frac{1}{2}\Z$-grading on $C_k$ is defined by $\conf(u)=1-j$ ($u\in\g_j$), $\conf(\varphi_\alpha)=1-\degG\alpha$, $\conf(\varphi^\alpha)=\degG\alpha$ and $\conf(\Phi_\alpha)=\frac{1}{2}$, where $\Delta(A)$ is the conformal weight of $A\in C_k$. This conformal grading is preserved by the differential $d$ and induces a $\frac{1}{2}\Z_{\geq0}$-grading on $\W^k(\g,f;\Gamma)$, which depends on the choice of $\Gamma$.

\subsection{$\W$-algebras over $T$}\label{W-alg T sec}

Set the polynomial ring $U$ with a formal parameter $\para$ and the quotient field $F$, that is,
\begin{align*}
U=\C[\para],\quad
F=\C(\para).
\end{align*}
We denote by $T = U$ or $F$. Let $V^T(\g)$ the affine vertex algebra over $T$, where we replace $k$ by a formal parameter $\para$. Set $\FchT=\Fch\otimes T$ and $\FneT=\Fne\otimes T$. Then $d$ defines a differential on
\begin{align*}
C_T=V^T(\g)\otimes\FchT\otimes\FneT,
\end{align*}
where $\otimes=\otimes_T$. Instead of $\otimes_T$, we use the notation $\otimes$ whenever the base ring (or field) is $T$. The $\W$-algebra $\W^T(\g,f;\Gamma)$ over $T$ is defined by the BRST cohomology of the complex $(C_T,d)$. See e.g. \cite{ACL}. Let $V^{\tau_k}(\g_{\leq0})$ be the affine vertex algebra associated with $\g_{\leq0}$ and its invariant bilinear form $\tau_k$, see \eqref{eq:tauk-def} for the definition of $\tau_k$. By replacing $k$ with the indeterminate $\para$, we have the vertex algebra $V^{\tau_\para}(\g_{\leq0})$ over $T$, which we simply denote by $V^T(\g_{\leq0})$. Since the statements in \cite{KW1, KW2} are proved for the $\W$-algebras over $\C$ but the same proof applies for the $\W$-algebras over $T$, we have an embedding
\begin{align}\label{eq:W-alg emb T}
\W^T(\g,f;\Gamma) \hookrightarrow V^T(\g_{\leq0}) \otimes \FneT
\end{align}
and filtrations on $\W^T(\g,f;\Gamma)$, $V^T(\g_{\leq0})$ and $\FneT$ such that the image of the induced map
\begin{align*}
\gr\W^T(\g,f;\Gamma) \hookrightarrow \gr V^T(\g_{\leq0}) \otimes \gr\FneT
\end{align*}
coincides with $\gr V^T(\g^f)$, where $\g^f \subset \g_{\leq0}$ is the centralizer of $f$ in $\g$. This implies that there exists a finite set $\{W^i(z)\}_{i=1}^{\dim\g^f}$ of fields on $V^T(\g_{\leq0}) \otimes \FneT$ such that $\{W^i(z)\}_{i=1}^{\dim\g^f}$ freely generates $\W^T(\g,f;\Gamma)$ over $T$ and each $W^i(z)$ has a leading term $u_i(z)$ on $V^T(\g_{\leq0})$ for some basis $\{u_i\}_{i=1}^{\dim\g^f}$ of $\g^f$. By using the PBW basis of $\W^T(\g,f;\Gamma)$ that forms $W^{i_1}_{(-n_1)} \cdots W^{i_s}_{(-n_s)}|0\rangle$ with some orders on $i_j$ and $n_j \in \Z_{\geq1}$, it follows that $\W^U(\g,f;\Gamma)$ is a subalgebra of $\W^F(\g,f;\Gamma)$ and
\begin{align*}
\W^U(\g,f;\Gamma) \otimes_U F=\W^F(\g,f;\Gamma),\quad
\W^U(\g,f;\Gamma)\otimes_U \C_{k}=\W^{k}(\g,f;\Gamma),
\end{align*}
where $\C_{k}$ is a $1$-dimensional $U$-module defined by $\para \mapsto k\in\C$. For $k\in\C$, we call the functor $?\otimes_U\C_{k}$ the {\em specialization}.

\subsection{Miura maps}\label{sec:Miura maps}
By \cite{KW1, KW2}, the $\W$-algebra $\W^k(\g,f;\Gamma)$ is embedded into a vertex algebra $V^{\tau_k}(\g_{\leq0}) \otimes \Fne$. Composing the embedding with a surjective homomorphism $V^{\tau_k}(\g_{\leq0}) \twoheadrightarrow \Vtau$ induced from the projection $\g_{\geq0} \twoheadrightarrow \g_0$, we have a vertex algebra homomorphism
\begin{align*}
\mu_k \colon \W^k(\g, f;\Gamma) \rightarrow \Vtau \otimes \Fne,
\end{align*}
which is called the Miura map of $\W^k(\g, f ;\Gamma)$ and injective for all $k\in\C$ by \cite{F, A}, see also \cite{G}. As in the same way, using the embedding \eqref{eq:W-alg emb T}, we also have a homomorphism
\begin{align*}
\mu_T \colon \W^T(\g,f;\Gamma) \hookrightarrow V^T(\g_0) \otimes \FneT
\end{align*}
of vertex algebras over $T$, which we call the Miura map of $\W^T(\g,f;\Gamma)$. Using the facts that $\W^U(\g,f;\Gamma)$ is a subalgebra of $\W^F(\g,f;\Gamma)$ and that $\W^U(\g,f;\Gamma)\otimes_U \C_{k}=\W^{k}(\g,f;\Gamma)$, it follows that
\begin{align}\label{eq:Miura U k property}
\mu_U = \mu_F|_{\W^U(\g,f;\Gamma)},\quad
\mu_k = \mu_U \otimes_U \C_k.
\end{align}
Since $\mu_F$ is injective by Lemma \ref{lem:Miura F}, $\mu_U$ is injective as well.

\subsection{Screening Operators $\hQ_\alpha$}\label{parabolic screening operator}
We recall the screening operators of $\W$-algebras introduced in \cite{G}. Let $\bQ_0=\bigoplus_{\gamma\in\Pi_0}\Z\gamma$ be the root lattice of $\rf$, $\Pi^\Gamma=\{\alpha\in\Delta_{>0}\mid{}^\nexists\beta,\gamma\in\Delta_{>0}\ \mathrm{s.t.}\ \alpha=\beta+\gamma\}$ a set of indecomposable roots in $\Delta_{>0}$. Define an equivalence relation on $\Delta_{>0}$ by $\alpha\sim\beta\iff\alpha-\beta\in\bQ_0$, which may restrict to $\Pi^\Gamma$. Let $[\Pi^\Gamma]=\Pi^\Gamma/\sim$ be the quotient set and
\begin{align}\label{[a] def eq}
[\alpha]=\{\beta\in\Delta_+\mid\beta-\alpha\in\bQ_0\}
\end{align}
the equivalence class of $\alpha\in\Pi^\Gamma$ in $[\Pi^\Gamma]$. Consider a map $\flat\colon\Pi_{>0}\ni\alpha\mapsto[\alpha]\in[\Pi^\Gamma]$.
\begin{lemma}\label{flat lemma}
The map $\flat$ is bijective.
\end{lemma}
\begin{proof}
Since it is clear that $\flat$ is injective, we will show that $\flat$ is surjective. Let $\beta\in\Pi^\Gamma$ and $n=\rootht\beta$ the height of $\beta$. In the case that $n=1$, we have $\beta\in\Pi_{>0}$ and $\flat(\beta)=[\beta]$. Next, we assume that $n>1$. Then there exist $\beta_1,\beta_2\in\Delta_+$ such that $\beta=\beta_1+\beta_2$. Since $\beta\in\Pi^\Gamma$, we may assume that $\beta_2\in\Deltazero$. Then $\beta-\beta_1=\beta_2\in\bQ_0$ and $[\beta]=[\beta_1]$. We claim that $\beta_1\in\Pi^\Gamma$. The reason is below. If there exist $\gamma_1,\gamma_2\in\Delta_{>0}$ such that $\beta_1=\gamma_1+\gamma_2$, we have
\begin{align*}
\g_\beta=[\g_{\beta_1},\g_{\beta_2}]=[[\g_{\gamma_1},\g_{\gamma_2}],\g_{\beta_2}]=[[\g_{\gamma_1},\g_{\beta_2}],\g_{\gamma_2}]+[\g_{\gamma_1},[\g_{\gamma_2},\g_{\beta_2}]].
\end{align*}
Then $\gamma_1+\beta_2\in\Delta_{>0}$ or $\gamma_2+\beta_2\in\Delta_{>0}$. Hence, it turns out that $\beta=\gamma_1+\gamma_2+\beta_2$ can be decomposed to the sum of two roots in $\Delta_{>0}$, which is contrary to our assumption that $\beta\in\Pi^\Gamma$. Therefore there exists $\beta_1\in\Pi^\Gamma$ such that $[\beta]=[\beta_1]$ and $\rootht(\beta_1)<\rootht(\beta)$. By induction on $n$, it follows that there exists $\alpha\in\Pi_{>0}$ such that $\flat(\alpha)=[\beta]$, that is, $\flat$ is surjective. The proof of the lemma is now complete.
\end{proof}
By Lemma \ref{flat lemma}, we may identify $[\Pi^\Gamma]$ with $\Pi_{>0}$ through $\flat$. Set a vector space $\C^{[\alpha]}=\bigoplus_{\beta\in[\alpha]}\C v_\beta$ for each $\alpha\in\Pi_{>0}$ and define a $\rf$-action on $\C^{[\alpha]}$ by
\begin{align}\label{eq:L0(-alpha)}
u\cdot v_\beta=\sum_{\gamma\in[\alpha]}c_{\gamma,u}^\beta v_\gamma,\quad
u \in \g_0,\ 
\beta \in [\alpha],
\end{align}
where $c_{\gamma,u}^\beta$ is a structure constant defined by the following formula: $[e_\gamma,u]=\sum_{\beta\in[\alpha]}c_{\gamma,u}^\beta e_\beta$. Then, by Remark 3.3 in \cite{G}, $\C^{[\alpha]}$ is the simple highest weight $\g_0$-module with a highest weight vector $v_\alpha$ of the highest weight $-\alpha$, which we denote by $L_{0}(-\alpha)$ (see Appendix \ref{appendix}). Let $\g_0^F = \g_0 \otimes F$ be a Lie algebra over $F$ and
\begin{align}\label{g0 affine Lie algebra}
\hat{\g}_0^F= \g_0^F \otimes F[t,t^{-1}]\oplus F K
\end{align}
be the affine Lie algebra of $\g_0^F$ over $F$ that is the central extension of $\g_0 ^F \otimes F[t,t^{-1}]$ by $\tau_\para$. Then $L_{0}(-\alpha)^F = L_{0}(-\alpha) \otimes F$ is a $\g_0^F \otimes F[t] \oplus F K$-module by $\g_0^F \otimes F[t]t=0$ and $K=1$. Let $\Weyl^F_0(-\alpha)$ be the Weyl module of $\hat{\g}_0^F$ with the highest weight $-\alpha$, which is the induced $V^F(\g_0)$-module from $L_{\Pi_0}(-\alpha)^F$ (see Appendix \ref{appendix}) defined by
\begin{align*}
\Weyl^F_0(-\alpha) = U(\hat{\g}^F_0)\underset{U(\g_0^F \otimes F[t]\oplus F K)}{\otimes} L_{0}(-\alpha)^F \simeq V^F(\g_0) \otimes \bigoplus_{\beta\in[\alpha]}F v_\beta.
\end{align*}
We will introduce the screening operators $\hQ_\alpha$ as intertwining operators. For a vertex algebra $V$ and $V$-modules $L,M$ and $N$, a linear map
\begin{align*}
Y_{L,M}^N(\cdot,z)\colon L\rightarrow\Hom(M,N)\{z\}=\sum_{n\in\mathbb{Q}}\Hom(M,N)\ z^n
\end{align*}
is called an intertwining operator of type $\binom{N}{L\ M}$ if it satisfies the Borcherds identity (see \cite{FHL} for the details). For a $V$-module $M$ and a vector $m\in M$,  we shall call $Y_{M,V}^M(m, z)$ an {\em intertwining operator corresponding to $m$}. In the present paper, we only consider intertwining operators of type $\binom{M}{M\ V}$. Let $\widetilde{V}_\beta(z)=\sum_{n\in\Z}\widetilde{V}_{\beta, n}z^{-n}$ be an intertwining operator corresponding to $v_\beta$ defined by $\widetilde{V}_{\beta,n}\cdot |0\rangle=\delta_{n,0}v_\beta$ $(n\geq0)$ and
\begin{align*}
[u(z),\widetilde{V}^\beta(w)]=\sum_{\gamma\in[\alpha]}c_{\gamma,u}^\beta \widetilde{V}^\gamma(w)\delta(z-w)
\end{align*}
for all $u\in\rf$. Then $\widetilde{V}^\beta(z)$ is well-defined, see Proposition 3.7 in \cite{G}. We define screening operators $\hQ_\alpha \colon V^F(\g_0) \otimes \FneF \rightarrow \Weyl^F_0(-\alpha) \otimes \FneF$ for $\alpha \in \Pi_{>0}$ by
\begin{align*}
\hQ_\alpha = 
\begin{cases}
\displaystyle \sum_{\beta\in[\alpha]}\int\widetilde{V}^\beta(z)\Phi_\beta(z)\ dz & (\alpha\in\Pi_{\frac{1}{2}}),\\
\displaystyle \sum_{\beta\in[\alpha]}\chi(e_\beta)\int \widetilde{V}^\beta(z)\ dz & (\alpha\in\Pi_1).
\end{cases}
\end{align*}
By the proof of Lemma 5.4 in \cite{G} (where the statement was proved for $\W$-algebras over $\C[(\para+h^\vee)^{-1}]$, but the same proof applies for $\W^F(\g,f;\Gamma)$), we have a vertex algebra isomorphism
\begin{align}\label{old main eq}
\W^F(\g,f;\Gamma) \simeq \bigcap_{\alpha\in\Pi_{>0}}\Ker\left( \hQ_\alpha \colon V^F(\g_0) \otimes \FneF \rightarrow \Weyl^F_0(-\alpha) \otimes \FneF \right).
\end{align}
\begin{lemma}[{\cite[Lemma 5.1]{G}}]\label{lem:Miura F}
The embedding $\W^F(\g,f;\Gamma) \hookrightarrow V^F(\g_0) \otimes \FneF$ induced from \eqref{old main eq} coincides with the Miura map $\mu_F$. In particular, $\mu_F$ is injective.
\begin{proof}
Though the assertion was proved in \cite{G} for $\W^k(\g,f;\Gamma)$ with generic $k\in\C$, the same proof applies.
\end{proof}
\end{lemma}

\section{Wakimoto representations for Affine Vertex Algebras}\label{aff Wak sec}

We introduce Wakimoto representations of $V^T(\g)$ and the screening operators $S_\alpha$ of $V^T(\g)$. We follow the construction given in \cite{F}.

\subsection{Differential representations of $\g$}
Let $G$ be a connected simply-connected Lie group corresponding to $\g$, $B_{+}$ the Borel subgroup corresponding to $\bo_{+}$, $B_{-}$ the opposite Borel subgroup and $N_{+}$ the unipotent subgroup corresponding to $\nil_{+}$. The left $G$-action on a flag variety $G/B_{-}$ induces a Lie algebra homomorphism $\rho_{G/B_{-}}\colon\g\rightarrow\D_{G/B_{-}}$, where $\D_{G/B_{-}}$ is the ring of differential operators of regular functions on $G/B_{-}$. Let $p\colon G\rightarrow G/B_{-}$ be the canonical projection and $U_0=N_{+}\cdot p(1)$ an $N_+$-orbit in $G/B_{-}$, where $1$ denotes the unit in $G$. An orbit $U_0$ is a unique open dense orbit in $G/B_{-}$ called the big cell. Since $N_{+}$ is unipotent, the exponential map $c(\nil_+)\colon\nil_{+}\rightarrow N_{+}$ is an isomorphism. The big cell $U_0 \simeq N_{+}$ is then the affine space of the dimension $|\Delta_{+}|$ and the ring $\C[N_{+}]$ of regular functions on $N_+$ is a polynomial ring. A Lie algebra homomorphism $\rho\colon\g\rightarrow\D_{N_{+}}$ is defined by the restriction of $\rho_{G/B_{-}}$ on $U$. Fix a coordinate system $\{x_{\alpha}\}_{\alpha\in\Delta_{+}}$ on $N_{+}$ by using $c(\nil_+)$ such that $h\cdot x_{\alpha}=-\alpha(h)x_{\alpha}$ for all $h\in\h$ and $\alpha\in\Delta_{+}$. This coordinate is called homogeneous.

To describe the image of $\rho$, we introduce the frameworks in \cite{F}. Fix a root vector $e_\alpha\in\g_\alpha$ for $\alpha\in\Delta$. Denote by $f_\alpha=e_{-\alpha}$ and $h_\alpha=[e_\alpha,f_\alpha]$ for $\alpha\in\Delta_+$. Let $G^{\circ}=p^{-1}(U_0)=N_{+}\cdot B_{-}$ be a dense open submanifold in $G$. For $a\in\g$, set a smooth curve $\gamma(t)=\exp(-t a)$ on $G$. Given $X\in G^{\circ}$,
\begin{align*}
\gamma(t)X=Z_{+}(t)Z_{-}(t)
\end{align*}
for $|t|\ll1$, where $Z_{+}(t)\in N_{+}$ and $Z_{-}(t)\in B_{-}$. A vector field $\zeta_a$ is then given by the following formula :
\begin{align*}
(\zeta_a f)(p(X))=\frac{d}{d t}f(Z_{+}(t))|_{t=0}
\end{align*}
for any smooth function $f$ defined in an open subset in $U_0$ around $p(X)$. Choose a faithful representation $V_0$ of $\g$ and consider $X\in N_{+}$ as a matrix in $\GL(V_0)$ whose entries are polynomials in $\C[N_{+}]=\C[x_{\alpha}]_{\alpha\in\Delta_{+}}$. We have
\begin{align*}
(1-t a)X=Z_{+}(t) Z_{-}(t)\quad\mod.\ (t^2).
\end{align*}
Hence $Z_{+}(t)=X+t Z$, $Z_{-}=1+t Z'$ $\mod.(t^2)$, where $Z\in\nil_{+}$ and $Z'\in\bo_{-}$. We have
\begin{align*}
\zeta_a\cdot X=-X(X^{-1} a X)_{+},
\end{align*}
where $(\cdot)_{+}:\g=\nil_{+}\oplus\bo_{-}\rightarrow\nil_{+}$ is the first projection. For $a\in\g$, $\rho(a)$ is a derivation in $\C[N_{+}]$ such that
\begin{align*}
\rho(e_{\alpha})&=\sum_{\beta\in\Delta_{+}}P_{\alpha}^{\beta}(x)\der_\beta=\der_\alpha+\sum_{\beta\in\Delta_{+}\backslash\{\alpha\}}P_{\alpha}^{\beta}(x)\der_\beta,\\
\rho(h_{\alpha})&=-\sum_{\beta\in\Delta_{+}}\beta(h_{\alpha})x_{\beta}\der_\beta,\\
\rho(f_{\alpha})&=\sum_{\beta\in\Delta_{+}}Q_{\alpha}^{\beta}(x)\der_\beta
\end{align*}
for all $\alpha\in\Delta_+$, where $\der_\alpha=\der/\der x_\alpha$, $x=(x_{\alpha})_{\alpha\in\Delta_{+}}$ and $P_{\alpha}^{\beta}(x), Q_{\alpha}^{\beta}(x)\in\C[N_{+}]$. For $\lambda\in\h^*$, we have a twisted Lie algebra homomorphism $\rho_{\lambda}\colon\g\rightarrow\D_{N_{+}}$ by
\begin{align*}
\rho_{\lambda}(e_{\alpha})&=\sum_{\beta\in\Delta_{+}}P_{\alpha}^{\beta}(x)\der_\beta,\\
\rho_{\lambda}(h_{\alpha})&=-\sum_{\beta\in\Delta_{+}}\beta(h_{\alpha})x_{\beta}\der_\beta+\lambda(h_{\alpha}),\\
\rho_{\lambda}(f_{\alpha})&=\sum_{\beta\in\Delta_{+}}Q_{\alpha}^{\beta}(x)\der_\beta+\lambda(h_{\alpha})x_{\alpha}
\end{align*}
for all $\alpha\in\Pi$.

\subsection{Wakimoto representations of $V^T(\g)$}\label{sec:Wakimoto subsec}
For any finite set $S$, let $\A_S$ be the (infinite-dimensional) Weyl vertex algebra associated with $S$, whose generating fields $a_{\alpha}(z),a^*_{\alpha}(z)$ for $\alpha\in S$ satisfy
\begin{align*}
a_{\alpha}(z)a^{*}_{\beta}(w)\sim\frac{\delta_{\alpha,\beta}}{z-w},\quad
a_{\alpha}(z)a_{\beta}(w)\sim0\sim a^*_{\alpha}(z)a^{*}_{\beta}(w)
\end{align*}
for all $\alpha,\beta\in S$. For a polynomial $P(x)\in\C[N_+]$, we define a field $P(a^*)(z)$ on $\Anil$ by
\begin{align}\label{P(a) pol eq}
P(a^*)(z):=P(x)|_{x_\alpha=a^*_\alpha(z)\ (\alpha\in\Delta_+)}.
\end{align}
Since $a^*_\alpha(z)$ and $a^*_\beta(z)$ commute for all $\alpha,\beta\in\Delta_+$, $P(a^*)(z)$ is well-defined. Denote by $P(a^*)$ the vector in $\Anil$ corresponding to a field $P(a^*)(z)$. We have
\begin{align}\label{a der eq}
a_\alpha(z)P(a^*)(w)\sim\frac{\der_\alpha P(a^*)(w)}{z-w}.
\end{align}
Let $\Hi^k = \Hi^k(\mathcal{\h}) =V^{k+h^{\vee}}(\h)$ be the Heisenberg vertex algebra associated with the Cartan subalgebra $\h$ of $\g$, whose generating fields $b_{\alpha}(z)$ for $\alpha\in\Pi$ satisfy
\begin{align*}
b_{\alpha}(z)b_{\beta}(w)\sim\frac{(k+h^{\vee})(\alpha|\beta)}{(z-w)^2}
\end{align*}
for all $\alpha,\beta\in\Pi$. For $\lambda\in\h^*$, denote by $\Hi_\lambda^k = \Hi_\lambda^k(\h)$ the highest weight $\Hi^k$-module with highest weight $\lambda$ whose highest weight vector $|\lambda\rangle$ satisfies that $b_{\alpha (n)}|\lambda\rangle = \delta_{n, 0}(\alpha|\lambda)$ for all $n \in \Z_{\geq0}$. Let $\Hi^T$ be the Heisenberg vertex algebra over $T$ and $\Hi_\lambda^T$ the highest weight $\Hi^T$-module with highest weight $\lambda\in\h^*$, where we replace $k$ by a formal parameter $\para$ in $T$. Set $\A_S^T=\A_S\otimes_\C T$.

\begin{lemma}[\cite{Wak, FF1, F}]\label{aff def lemma} There exists an injective vertex algebra homomorphism $\hat{\rho}_T\colon V^T(\g)\rightarrow\Anil^T\otimes\Hi^T$ over $T$ with some $c_\alpha\in\C$ for each $\alpha\in\Pi$ such that
\begin{align*}
\hat{\rho}_T(e_{\alpha}(z))&=\sum_{\beta\in\Delta_{+}}:P_{\alpha}^{\beta}(a^*)(z)a_\beta(z):\ =a_\alpha(z)+\sum_{\beta\in\Delta_{+}\backslash\{\alpha\}}:P_{\alpha}^{\beta}(a^*)(z)a_\beta(z):,\\
\hat{\rho}_T(h_{\alpha}(z))&=-\sum_{\beta\in\Delta_{+}}\beta(h_{\alpha}):a^*_{\beta}(z)a_\beta(z):+b_\alpha(z),\\
\hat{\rho}_T(f_{\alpha}(z))&=\sum_{\beta\in\Delta_{+}}:Q_{\alpha}^{\beta}(a^*)(z)a_\beta(z):+:b_\alpha(z)a^*_\alpha(z):+((e_\alpha|f_\alpha)\para+c_\alpha)\der a^*_\alpha(z)
\end{align*}
for all $\alpha\in\Pi$. For any $\alpha\in\Delta_+$, $\hat{\rho}_T(e_{\alpha}(z))$, $\hat{\rho}_T(h_{\alpha}(z))$ also take the same forms.
\end{lemma}

The injective vertex algebra homomorphism $\hat{\rho}_T$ provides a $V^T(\g)$-module structure on any $\Anil^T\otimes\Hi^T$-module, called a Wakimoto representation of $V^T(\g)$. By definition, it is clear that the restriction of $\hat{\rho}_F$ on $V^U(\g)$ coincides with $\hat{\rho}_U$:
\begin{align*}
\hat{\rho}_U = \hat{\rho}_F|_{V^U(\g)}.
\end{align*}
The specialization of $\hat{\rho}_U$ induces a vertex algebra homomorphism
\begin{align*}
\hat{\rho}_k=\hat{\rho}_U\otimes_U\C_k\colon V^k(\g)\rightarrow\Anil\otimes\Hi,
\end{align*}
which is also injective by \cite{F}.

\subsection{Screening Operators for $V^T(\g)$}\label{sec:affine Wak sc}
Let $\rho^R\colon\nil_{+}\rightarrow\D_{N_{+}}$ be the Lie algebra anti-homomorphism induced by the right action of $N_{+}$ on itself. Denote by
\begin{align}\label{right rep eq}
\rho^R(e_{\alpha})=\sum_{\beta\in\Delta_{+}}P_{\alpha}^{\beta,R}(x)\der_\beta
\end{align}
for $\alpha\in\Delta_{+}$, where $P_{\alpha}^{\beta,R}(x)$ is a polynomial in $\C[N_{+}]$. Since the left and right actions of $N_+$ on itself commute, we have
\begin{align*}
[\rho(e_\alpha),\rho^R(e_{\beta})]=0
\end{align*}
for all $\alpha,\beta\in\Delta_+$. Let
\begin{align*}
\Wak^T(\lambda)=\Anil^T\otimes\Hi^T_\lambda
\end{align*}
be a Wakimoto representation of $V^T(\g)$ for $\lambda\in\h^*$. Denote by
\begin{align*}
\Wak(\lambda)=\Wak^U(\lambda)\otimes\C_k,\quad
\Wak^T_\g=\Wak^T(0),\quad
\Wak_\g=\Wak(0).
\end{align*}
Let $W$ be the Weyl group of $\g$, $\ell(w)$ the length of $w \in W$ and $w_\circ$ the longest element in $W$. For $w \in W$ and $\lambda \in \h^*$, $w \circ \lambda = w(\lambda+\rho_\circ)-\rho_\circ$ is called the dot action of $w$, where $\rho_\circ$ is the Weyl vector of $\g$, that is, the half sum of positive roots of $\g$. Let
\begin{align*}
D_i^T = \bigoplus_{\begin{subarray}{c} w \in W \\ \ell(w) = i \end{subarray}}\Wak^T(w^{-1} \circ 0).
\end{align*}
Then $D_0^T = \Wak^T_\g$ and $D_1^T = \bigoplus_{\alpha\in\Pi}\Wak^T(-\alpha)$. By \cite{FF5}, there exists an exact sequence
\begin{align}\label{affine Wakimoto resolution}
0\rightarrow V^F(\g)
\xrightarrow{\hat{\rho}_F} D_0^F
\xrightarrow{\bigoplus S_\alpha} D_1^F
\rightarrow \cdots
\rightarrow D_{\ell(w_\circ)}^F
\rightarrow 0,
\end{align}
where and $S_\alpha \colon \Wak^F_\g \rightarrow \Wak^F(-\alpha)$ is an intertwining operator defined by
\begin{align}\label{scaffine eq}
S_\alpha = \int S_\alpha(z)\ dz = \int:\hat{\rho}^{R}(e_\alpha(z))\ \e^{-\frac{1}{\para+h^{\vee}}\int b_\alpha(z)}:dz
\end{align}
for $\alpha\in\Pi$, where
\begin{align}\label{right affine eq}
\hat{\rho}^{R}(e_\alpha(z))=\sum_{\beta\in\Delta_{+}}:P_{\alpha}^{\beta,R}(a^*)(z)a_\beta(z):.
\end{align}
In particular, we have
\begin{align}\label{affine Wakimoto}
V^F(\g) \simeq \Img\hat{\rho}_F = \bigcap_{\alpha\in\Pi}\Ker\left(S_\alpha \colon \Wak^F_\g \rightarrow \Wak^F(-\alpha) \right).
\end{align}
The long exact sequence \eqref{affine Wakimoto resolution} is called the Wakimoto resolution of $V^F(\g)$ and the intertwining operators $S_\alpha$ are called the screening operators for $V^T(\g)$. We note that these screening operators are only considered for generic $\para=k$ in \cite{FF5} but the same proof also holds when the base field is $F$. See also \cite{ACL}. Since $\hat{\rho}_U = \hat{\rho}_F|_{V^U(\g)}$, we have
\begin{align*}
V^U(\g) \simeq \Img\hat{\rho}_U \subset \bigcap_{\alpha\in\Pi}\Ker\left(S_\alpha \colon \Wak^U_\g \rightarrow \Wak^F(-\alpha) \right).
\end{align*}
\begin{prop}\label{affine Wak U}
\begin{align*}
\Img\hat{\rho}_U = \bigcap_{\alpha\in\Pi}\Ker\left(S_\alpha \colon \Wak^U_\g \rightarrow \Wak^F(-\alpha) \right).
\end{align*}
\end{prop}
\begin{proof}
Let $\{u^i\}_{i=1}^{\dim\g}$ be a basis of $\g$ and $\{v^i\}_{i\in\Lambda}$ be a PBW basis of $V^T(\g)$ over $T$  that forms $u^{i_1}_{(-n_1)}\cdots u^{i_s}_{(-n_s)}|0\rangle$ with some order on $i_j$ and $n_j \in \Z_{\geq1}$. Then we can take the index set $\Lambda$ independent of the choice of $T$. Using the PBW basis of $V^U(\g)$, we have the specialization map
\begin{align*}
V^U(\g) \ni A \mapsto A|_{\para=k} \in V^k(\g),\quad
k \in \C
\end{align*}
that maps the PBW basis of $V^U(\g)$ to that of $V^k(\g)$ and each coefficients are specialized by $\para \mapsto k$. Let $w \in \bigcap_{\alpha\in\Pi}\Ker\left(S_\alpha \colon \Wak^U_\g \rightarrow \Wak^F(-\alpha) \right)$ with $w\neq0$. By \eqref{affine Wakimoto}, there exists $v \in V^F(\g)$ such that $\hat{\rho}_F(v) = w$. Then
\begin{align*}
v = \sum_{i \in \Lambda_v} \left( \frac{f_i}{g_i} \right) v^i,\quad
f_i, g_i \in U\backslash\{0\},\ 
(f_i, g_i)=1,
\end{align*}
where $\Lambda_v$ is a finite subset of $\Lambda$. Let $g$ be the least common multiple of $\{g_i\}_{i \in \Lambda_v}$ in $U$ (unique up to $\C\backslash\{0\}$). Then $g \cdot v \in V^U(\g)$. Suppose that $g \notin \C\backslash\{0\}$. Let $k_0$ be any root of $g$. The specialization $\para \mapsto k_0$ yields the equation
\begin{align*}
\hat{\rho}_{k_0}\left((g \cdot v)|_{\para = k_0}\right) = g|_{\para = k_0}  \cdot w|_{\para = k_0} =0.
\end{align*}
Since $g$ is the least common multiple of $\{g_i\}_{i \in \Lambda_v}$, we have $(g \cdot v)|_{\para = k_0} \neq 0$, contrary to the injectivity of $\hat{\rho}_{k_0}$. Therefore $g \in \C\backslash\{0\}$ and $v \in V^U(\g)$. Hence $\hat{\rho}_U(v) = w$.
\end{proof}

\section{Wakimoto representations for Affine $\W$-algebras}\label{W-alg Wak sec}

\subsection{Coordinates on $N_{+}$}\label{local sec}
Let $G_{>0}$, $G_0^+$ be the unipotent Lie subgroups in $N_{+}$ corresponding to $\g_{>0}$, $\g_0^+=\g_0\cap\nil_+$ respectively. Since $\g_{>0}$ is an ideal in $\nil_{+}$, a subgroup $G_{>0}$ is normal in $N_{+}$. Hence a set $G_{>0}\times G_0^+$ has a group structure and is isomorphic to $N_{+}$ by $G_{>0}\times G_0^+\ni(a,b)\mapsto a\cdot b\in G_{>0}\cdot G_0^+=N_{+}$. Let $c(\g_{>0})$, $c(\g_0^+)$ be homogeneous coordinates on $G_{>0}$, $G_0^+$ respectively. A coordinate $c(\nil_+)$ on $N_+$ is then defined by $c(\nil_+)=c(\g_{>0})\cdot c(\g_0^+)$, which induces a ring isomorphism $\C[N_{+}]\simeq\C[G_{>0}]\otimes\C[G_0^+]$. We call $c(\nil_+)=c(\g_{>0})\cdot c(\g_0^+)$ a coordinate on $N_+$ {\em compatible with the decomposition $N_{+}=G_{>0}\times G_0^+$}. By construction, we have
\begin{align*}
\rho|_{\g_{>0}}=\rho_{\g_{>0}},\quad
\rho^R|_{\g_0^{+}}=\rho^R_{\g_0^{+}},
\end{align*}
where $\rho_{\g_{>0}}$ is the Lie algebra homomorphism derived from the left action of $G_{>0}$ on $G_{>0}$ and $\rho^R_{\g_0^{+}}$ is the Lie algebra anti-homomorphism derived from the right action of $G_0^{+}$ on $G_0^{+}$. Thus, we obtain:

\begin{lemma}\label{pol0 lemma} Suppose that $c(\nil_+)$ is compatible with the decomposition $N_{+}=G_{>0}\times G_0^+$. Then
\begin{enumerate}
\item $\rho(u)$ belongs to $\D_{G_{>0}}$ for all $u\in\g_{>0}$.
\item $\rho^R(u)$ belongs to $\D_{G_0^{+}}$ for all $u\in\g_0^+$. In particular,
\begin{align*}
\rho^R(e_\alpha)=\sum_{\beta\in\Deltazero}P_\alpha^{\beta,R}(x)\der_\beta
\end{align*}
for all $\alpha\in\Deltazero$.
\end{enumerate}
\end{lemma}

Let $\bQ$ be the root lattice of $\g$ and $\bQ_+\subset\bQ$ the set of linear combination with coefficients in $\Z_{\geq0}$ of elements of $\Pi$. Define a $\bQ$-valued grading on $\D_{N_+}$ by
\begin{align*}
\degQ(\der_\alpha)=-\alpha,\quad
\degQ x_\alpha=\alpha
\end{align*}
for $\alpha\in\Delta_+$, which induces a $\bQ_+$-grading on $\C[N_+]$. We define a $\bQ$-valued grading on $\g$ by $\degQ(\g_\alpha)=\alpha$ and $\degQ(\h)=0$. Then $\rho$ and $\rho^R$ reverse the $\bQ$-gradings, i.e. $\degQ\rho(u)=-\degQ(u)$ for $u\in\g$, and $\degQ\rho^R(u)=-\degQ(u)$ for $u\in\nil_+$. Therefore
\begin{align}
\label{polQ eq}\degQ P_\alpha^\beta(x)&=\degQ P_\alpha^{\beta,R}(x)=\beta-\alpha\in\bQ_+,\\
\degQ Q_\alpha^\beta(x)&=\beta+\alpha\in\bQ_+
\end{align}
unless $P_\alpha^\beta(x)=P_\alpha^{\beta,R}(x)=Q_\alpha^\beta(x)=0$. A $\frac{1}{2}\Z$-grading $\degG$ on $\Delta$ may be extended to $\bQ$ linearly. Then the composition map $\degG\circ\degQ$ defines a $\frac{1}{2}\Z$-grading on $\D_{N_{+}}$, which we denote by $\degG$ by abuse of notations. We have
\begin{align}\label{grD eq}
\degG(\der_\alpha)=-\degG\alpha,\quad
\degG x_\alpha=\degG\alpha
\end{align}
for $\alpha\in\Delta_{+}$, which induces a $\frac{1}{2}\Z_{\geq0}$-grading on $\C[N_+]$. Then $\rho$ and $\rho^R$ reverse the gradings, i.e. $\degG\rho(u)=-\degG(u)$ for $u\in\g$, and $\degG\rho^R(u)=-\degG(u)$ for $u\in\nil_+$. We have
\begin{align}
\label{polG eq}\degG P^\beta_\alpha(x)&=\degG P^{\beta,R}_\alpha(x)=\degG\beta-\degG\alpha\geq0,\\
\degG Q_\alpha^\beta(x)&=\degG\beta+\degG\alpha\geq0
\end{align}
unless $P_\alpha^\beta(x)=P_\alpha^{\beta,R}(x)=Q_\alpha^\beta(x)=0$.

\begin{lemma}\label{pol1 lemma}
Suppose that $c(\nil_+)$ is compatible with the decomposition $N_{+}=G_{>0}\times G_0^+$. If $\degG\alpha=\degG\beta$, polynomials $P^\beta_\alpha(x)$ and $P^{\beta,R}_\alpha(x)$ belong to $\C[G_0^+]$.
\end{lemma}
\begin{proof}
Since $\degG P^\beta_\alpha(x)=\degG P^{\beta,R}_\alpha(x)=0$, they are concentrated in the homogeneous component of $\C[N_+]$ with degree $0$, which coincides with $\C[G_0^+]$. This completes the proof.
\end{proof}

\begin{lemma}\label{pol2 lemma}
Suppose that $c(\nil_+)$ is compatible with the decomposition $N_{+} =G_{>0}\times G_0^+$. For $\alpha\in\Delta_{>0}$,
\begin{align*}
\rho(e_\alpha)=\der_\alpha+\sum_{\begin{subarray}{c} \beta\in\Delta_{>0}\\ \degG\beta>\degG\alpha \end{subarray}}P_\alpha^\beta(x)\der_\beta.
\end{align*}
\end{lemma}
\begin{proof}
Let $\alpha\in\Delta_{>0}$. Then $\rho(e_\alpha)$ belongs to $\D_{G_{>0}}$ by Lemma \ref{pol0 lemma}, so polynomials $P_\alpha^\beta(x)$ are in $\C[G_{>0}]$ for all $\beta\in\Delta_{>0}$. First, we assume that $\degG\beta<\degG\alpha$. If $P_\alpha^\beta(x)\neq0$, we have $\degG P^\beta_\alpha(x)=\degG\beta-\degG\alpha\geq0$, which is contrary to our assumption. Therefore $P_\alpha^\beta(x)=0$.

Next, we assume that $\degG\beta=\degG\alpha$. Then $P_\alpha^\beta(x)\in\C[G_{>0}]\cap\C[G_0^+]=\C$ by Lemma \ref{pol1 lemma}. Therefore, $P_\alpha^\beta(x)$ is a scalar. Hence, $\degQ P_\alpha^\beta(x)=\beta-\alpha$ should be zero unless $P_\alpha^\beta(x)=0$. Now, we have $P_\alpha^\beta(x)=0$ if $\degG\beta\leq\degG\alpha$ except for $\beta=\alpha$. Since $P_\alpha^\alpha(x)=1$ by construction, the lemma follows.
\end{proof}

\subsection{Semi-infinite cohomology}\label{semi-inf sec}

Let $V^T(\g_{>0})$ be the vertex subalgebra of $V^T(\g)$ generated by $u(z)$ for all $u \in \g_{>0}$. Notice that $\Am^T$ is a $V^T(\g_{>0})$-module through the Wakimoto representation of $V^T(\g)$ with the coordinate $c(\nil_+)$ on $N_+$ compatible with $N_+ = G_{>0} \times G_0^+$. It follows from \cite{A2} that $\Am^T$ is isomorphic to the \textit{semi-regular bimodule} \cite{V93, V} of the loop algebra $\g_{>0}((t))\otimes T$ over $T$.
\begin{lemma}[{\cite[Proposition 3.4]{ACL}}]\label{ACL lemma} For any left $V^T(\g_{>0})$-module $M$ that is an integrable $\nilcurrent\otimes T$-module and free over $T$, there exists a $V^T(\g_{>0})$-bimodule isomorphism
\begin{align*}
\Psi \colon \Am^T\otimes M\xrightarrow{\sim}\Am^T\otimes M
\end{align*}
such that
\begin{align*}
\Psi \circ (x(z)\otimes 1+1\otimes x(z)) &= (x(z)\otimes 1) \circ \Psi,\\
\Psi \circ (x^R(z)\otimes 1) &= (x^R(z)\otimes 1-1\otimes x(z)) \circ \Psi.
\end{align*}
for $x\in\g_{>0}$, where $x^R(z)$ denotes the right action on $\Am^T$. Moreover, if $M$ is a vertex algebra equipped with a vertex algebra homomorphism $V^T(\g_{>0}) \rightarrow M$ that induces the $V^T(\g_{>0})$-module structure on $M$, $\Psi$ is a vertex algebra isomorphism.
\end{lemma}
The construction of $\Psi$ is as follows: Let $\Njet$ be the arc space of $G_{>0}$, which coincides with a pro-unipotent group $\Ncurrent$ whose Lie algebra is $\nilcurrent$. Then the ring $T[\Njet]$ of regular functions on $\Njet$ over $T$ may be identified with a vertex subalgebra of $\Am^T$ generated by $a^*_\alpha(z)$ for all $\alpha\in\Delta_{>0}$, where $a^*_{\alpha (-n-1)}$ corresponds to a coordinate for $e_\alpha t^n\in\nilcurrent$ for $n\geq0$. We have an isomorphism
\begin{align*}
V^T(\g_{>0})\otimes T[\Njet]\xrightarrow{\sim}\Am^T,\quad
u\otimes v\mapsto\hat{\rho}(u)_{(-1)}v
\end{align*}
of left $\g_{>0}[t^{-1}]t^{-1}$-right $\nilcurrent$-modules. By our assumptions, $M$ is an $\nilcurrent$-integrable module, and so is an $\Ncurrent$-module. This implies that the associated comodule map
\begin{align*}
\psi\colon M\rightarrow T[\Njet]\otimes M
\end{align*}
is defined. Given a basis $\{m_i\}_{i\in I}$ over $T$ of $M$, set regular functions $v_{i,j}\in T[\Njet]$ for $i,j\in I$ by
\begin{align*}
g\cdot m_i=\sum_{j\in I}v_{i,j}(g)m_j
\end{align*}
for all $g\in\Ncurrent=\Njet$. Then we have
\begin{align*}
\psi(m_i)=\sum_{j\in I}v_{i,j}\otimes m_j.
\end{align*}
We extend $\psi$ to the isomorphism $\widetilde{\psi}$ by
\begin{align*}
\widetilde{\psi}\colon T[\Njet]\otimes M\xrightarrow{\sim} T[\Njet]\otimes M,\quad
v\otimes m\mapsto (v\otimes1)\psi(m).
\end{align*}
It follows that
\begin{align*}
\widetilde{\psi}\circ(g^L\otimes g^L)=(g^L\otimes1)\circ\widetilde{\psi},\quad
\widetilde{\psi}\circ(g^R\otimes1)=(g^R\otimes g^{-1})\circ\widetilde{\psi}
\end{align*}
for $g\in\Ncurrent$, where $g^R\cdot v(g')=v(g g')$ for all $v\in T[\Njet]$ and $g'\in\Ncurrent$. Then the isomorphism
\begin{align*}
\Psi\colon\Am^T\otimes M=V^T(\g_{>0})\otimes(T[\Njet]\otimes M)\xrightarrow{\sim}\Am^T\otimes M
\end{align*}
is defined by
\begin{align*}
u \otimes w \mapsto \cop(u)(\widetilde{\psi}^{-1}(w))
\end{align*}
for $u\in V^T(\g_{>0})$, $w\in T[\Njet]\otimes M$, where $\cop\colon V^T(\g_{>0})\rightarrow \Am^T\otimes M$ is the coproduct action of $V^T(\g_{>0})\simeq U(\g_{>0}[t^{-1}]t^{-1})$ on $\Am^T\otimes M$, and we embed $T[\Njet]\otimes M$ into $\Am^T\otimes M$.

The \textit{semi-infinite cohomology} \cite{Fei} of a $V^T(\g_{>0})$-module $M$ is the cohomology
\begin{align*}
H^{\frac{\infty}{2}+\bullet}(V^T(\g_{>0}),M):=H^\bullet(M\otimes\FchT,\dst)
\end{align*}
of the cochain complex $M\otimes\FchT$ with the differential $\dst=\int \dst(z)dz$ whose degree inherits from the charged degree on $\FchT$. Define a $V^T(\g_{>0})$-module structure on $\FneT$ by
\begin{align*}
e_\alpha(z) \mapsto \Phi_\alpha(z)+\chi(e_\alpha),\quad
\alpha\in\Delta_{>0}.
\end{align*}
If $\g_{\frac{1}{2}}=0$, $\FneT$ is just a $V^T(\g_{>0})$-module $T_\chi$ free over $T$ of rank $1$ defined by $e_\alpha(z) \mapsto \chi(e_\alpha)$. By definition,
\begin{align*}
\W^T(\g,f;\Gamma)=H^{\frac{\infty}{2}+0}(V^T(\g_{>0}),V^T(\g)\otimes\FneT)
\end{align*}
with the diagonal $V^T(\g_{>0})$-action on $V^T(\g)\otimes\FneT$. Now, we consider the cohomology
\begin{align*}
H^{\frac{\infty}{2}+0}(V^T(\g_{>0}),\Wak^T(\lambda)\otimes\FneT)
\end{align*}
with the diagonal $V^T(\g_{>0})$-action on $\Wak^T(\lambda)\otimes\FneT$. If $c(\nil_+)$ is compatible with $N_+ = G_{>0} \times G_0^+$, we have an isomorphism
\begin{align*}
\Wak^T(\lambda) \simeq \Am^T \otimes \Azero^T\otimes\Hi^T_\lambda \otimes \FneT
\end{align*}
of $V^T(\g_{>0})$-modules, where the $V^T(\g_{>0})$-action on the right hand side is defined by the diagonal action with a trivial $V^T(\g_{>0})$-module $\Azero^T\otimes\Hi^T_\lambda$. By using Lemma \ref{ACL lemma} for $M = \Azero^T\otimes\Hi^T_\lambda \otimes \FneT$, we have a $V^T(\g_{>0})$-bimodule isomorphism
\begin{align*}
\Psi_\lambda\colon\Am^T\otimes(\Azero^T\otimes\Hi^T_\lambda\otimes\FneT)\xrightarrow{\sim}\Am^T\otimes(\Azero^T\otimes\Hi^T_\lambda\otimes\FneT).
\end{align*}
\begin{prop}\label{semi-inf Wak prop}
Suppose that $c(\nil_+)$ is compatible with the decomposition $N_+ = G_{>0} \times G_0^+$. The map $\Psi_\lambda$ induces a $V^T(\g_{>0})$-bimodule isomorphism
\begin{align*}
[\Psi_\lambda]\colon H^{\frac{\infty}{2}+0}(V^T(\g_{>0}),\Wak^T(\lambda)\otimes\FneT) \xrightarrow{\sim} \Azero^T\otimes\Hi^T_\lambda\otimes\FneT
\end{align*}
that is defined by the identity on $\Azero^T\otimes\Hi^T_\lambda$ and
\begin{align*}
\hat{\Phi}_\alpha:=\Phi_\alpha - \sum_{\beta \in \Deltahalf} \chi([e_\beta, e_\alpha])a^*_{\beta} \mapsto \Phi_\alpha,\quad
\alpha \in \Deltahalf.
\end{align*}
\end{prop}
\begin{proof}
First of all, by using the fact that $V^T(\g_{>0})$ acts trivially on $\Azero^T\otimes\Hi^T_\lambda$, it follows that
\begin{align*}
H^{\frac{\infty}{2}+0}(V^T(\g_{>0}),\Wak^T(\lambda)\otimes\FneT) = \Azero^T\otimes\Hi^T_\lambda \otimes H^{\frac{\infty}{2}+0}(V^T(\g_{>0}),\Am^T\otimes\FneT).
\end{align*}
On the other hand, since
\begin{align}\label{cohomology vanishing for semi-regular bimodule} 
H^{\frac{\infty}{2}+0}(V^T(\g_{>0}),\Am^T)=T
\end{align}
by \cite[Theorem 2.1]{V93}, $\Psi_\lambda$ induces a $V^T(\g_{>0})$-bimodule isomorphism
\begin{multline}
H^{\frac{\infty}{2}+0}(V^T(\g_{>0}),\Anil^T\otimes\Hi_\lambda^T\otimes\FneT)\\
\rightarrow H^{\frac{\infty}{2}+0}(V^T(\g_{>0}),\Am^T)\otimes(\Azero^T\otimes\Hi_\lambda^T\otimes\FneT)
=\Azero^T\otimes\Hi_\lambda^T\otimes\FneT,
\end{multline}
which we denote by $[\Psi_\lambda]$. Thus, we have an isomorphism
\begin{align*}
[\Psi_\lambda] \colon \Azero^T\otimes\Hi^T_\lambda \otimes H^{\frac{\infty}{2}+0}(V^T(\g_{>0}),\Am^T\otimes\FneT) \xrightarrow{\sim} \Azero^T\otimes\Hi_\lambda^T\otimes\FneT
\end{align*}
that is the identity on $\Azero^T\otimes\Hi^T_\lambda$.

Now, consider $\Psi_\lambda$ on $\Am^T\otimes\FneT$. Let $x_{\beta (-n-1)}$ be a coordinate on $\Njet$ corresponding to $e_\beta t^n\in\nilcurrent$ for $\beta \in \Delta_{>0}$ and $n\geq0$. Notice that one can define the coordinate $J_\infty c(\g_{>0})\colon \nilcurrent \xrightarrow{\sim} \Ncurrent = \Njet$ by replacing $x_\beta e_\beta$ in $c(\g_{>0})$ with $\sum_{n \geq 0}x_{\beta (-n-1)} e_{\beta (n)}$. Then $x_{\beta (-n-1)}$ corresponds to $a^*_{\beta (-n-1)}$ in the isomorphism
\begin{align*}
V^T(\g_{>0}) \otimes T[\Njet] \xrightarrow{\sim} \Am^T,
\end{align*}
so we may identify $x_{\beta (-n-1)}$ with $a^*_{\beta (-n-1)}$. For $\displaystyle g = J_\infty c(\g_{>0})$ with indeterminates $x_{\beta (-n-1)}$, using the equations
\begin{align*}
e_{\beta (n)} \cdot \Phi_\alpha = \delta_{n, 0}\ \chi([e_\beta, e_\alpha]),\quad
\alpha \in \Deltahalf,\ 
\beta \in \Delta_{>0},\ 
n \in \Z_{\geq 0}
\end{align*}
and the fact that the linear term of $g$ forms $1 - \sum_{\beta \in \Delta_{>0}}x_{\beta (-n-1)}e_{\beta (n)}$, we have
\begin{align*}
g \cdot \Phi_\alpha = \Phi_\alpha - \sum_{\beta \in \Deltahalf}\chi([e_\beta, e_\alpha]) x_{\beta (-1)}.
\end{align*}
This implies that the comodule map $\psi$ satisfies the equations
\begin{align*}
\psi(\Phi_\alpha) = \Phi_\alpha - \sum_{\beta \in \Deltahalf} \chi([e_\beta, e_\alpha])a^*_{\beta} = \hat{\Phi}_\alpha,\quad
\alpha \in \Deltahalf.
\end{align*}
Then
\begin{align*}
\widetilde{\psi}(u(z)v(z))=\ :u(z)\psi(v)(z):
\end{align*}
for $u \in T[\Njet]$, $v \in \FneT$. Therefore $\widetilde{\psi}^{-1}(\hat{\Phi}_\alpha) = \Phi_\alpha$, and so $\Psi_\lambda(\hat{\Phi}_\alpha) = \Phi_\alpha$ by construction of $\Psi_\lambda$. By direct computations, we see that $\hat{\Phi}_\alpha$ is in the kernel of the differential of the complex defining $H^{\frac{\infty}{2}+0}(V^T(\g_{>0}),\Am^T\otimes\FneT)$. Thus $[\Psi_\lambda]$ maps $\hat{\Phi}_\alpha$ to $\Phi_\alpha$. This completes the proof.
\end{proof}
\begin{cor}\label{semi-inf Wak cor}
Suppose that $c(\nil_+)$ is compatible with the decomposition $N_+ = G_{>0} \times G_0^+$. Then
\begin{align*}
[\Psi_0] \colon H^{\frac{\infty}{2}+0}(V^T(\g_{>0}),\Wak^T_\g\otimes\FneT) \xrightarrow{\sim} \Azero^T \otimes \Hi^T \otimes \FneT
\end{align*}
is an isomorphism of vertex algebras and $[\Psi_\lambda]$ is an isomorphism of $H^{\frac{\infty}{2}+0}(V^T(\g_{>0}),\Wak^T_\g\otimes\FneT) = \Azero^T \otimes \Hi^T \otimes \FneT$-modules.
\begin{proof}
Since
\begin{align*}
\hat{\Phi}_\alpha(z)\hat{\Phi}_\beta(w) \sim \frac{\chi([e_\alpha, e_\beta])}{z-w},\quad
\alpha, \beta \in \Deltahalf,
\end{align*}
$\hat{\Phi}_\alpha \mapsto \Phi_\alpha$ defines an isomorphism of vertex algebras from $H^{\frac{\infty}{2}+0}(V^T(\g_{>0}),\Am^T\otimes\FneT)$ to $\FneT$. Thus, the assertion is immediate from Proposition \ref{semi-inf Wak prop}.
\end{proof}
\end{cor}

\subsection{Wakimoto free fields realizations of $\W^T(\g,f;\Gamma)$}\label{Wak sec}
Given a $V^T(\g)$-module $M$, define a $C_T$-module $C_T(M)=M\otimes\FchT\otimes\FneT$. Then $(C_T(M),d)$ is a cochain complex whose cohomology
\begin{align*}
H^\bullet_\chi(M)=H^\bullet(C_T(M),d)=H^{\frac{\infty}{2}+\bullet}(V^T(\g_{>0}),M\otimes\FneT)
\end{align*}
has a structure of a $\W^T(\g,f;\Gamma)$-module by construction. We have (\cite{KW1})
\begin{align}\label{eq:W-alg vanish}
H_\chi^\bullet(V^T(\g)) = H^0_\chi(V^T(\g)) = \W^T(\g,f;\Gamma).
\end{align}
Consider the case that $M = \Wak^T(\lambda)$. Then $H_\chi(\Wak^T(\lambda))$ is a $\W^T(\g,f;\Gamma)$-module for all $\lambda \in \h^*$. If $c(\nil_+)$ is compatible with the decomposition $N_+ = G_{>0} \times G_0^+$, Proposition \ref{semi-inf Wak prop} implies that
\begin{align}\label{eq:W-Wak vanish}
H_\chi^\bullet(\Wak^T(\lambda)) = H^0_\chi(\Wak^T(\lambda)) \simeq \Azero^T\otimes\Hi^T_\lambda\otimes\FneT,
\end{align}
which is an isomorphism of $H_\chi(\Wak^T_\g)$-module by Corollary \ref{semi-inf Wak cor}. Now, set
\begin{align*}
\Wak^T_\chi(\lambda) = \Azero^T\otimes\Hi^T_\lambda\otimes\FneT.
\end{align*}
Denote by
\begin{align*}
\Wak_\chi(\lambda) = \Wak^U_\chi(\lambda)\otimes\C_k,\quad
\Wak^T_{\g, \chi} = \Wak^T_\chi(0)\quad
\Wak_{\g, \chi} = \Wak_\chi(0).
\end{align*}
Applying the cohomology functor $H_\chi(?)$ with \eqref{eq:W-alg vanish} and \eqref{eq:W-Wak vanish} to \eqref{affine Wakimoto resolution}, we have an exact sequence
\begin{align}\label{W-alg Wakimoto resolution}
0\rightarrow \W^F(\g,f;\Gamma)
\xrightarrow{\omega_F} E_0^F
\xrightarrow{\bigoplus Q_\alpha} E_1^F
\rightarrow \cdots
\rightarrow E_{\ell(w_\circ)}^F
\rightarrow 0
\end{align}
of $\W^T(\g,f;\Gamma)$-modules with
\begin{align*}
E_i^T = \bigoplus_{\begin{subarray}{c} w \in W \\ \ell(w) = i \end{subarray}} H^0_\chi(\Wak^T(w^{-1}\circ0)) = \bigoplus_{\begin{subarray}{c} w \in W \\ \ell(w) = i \end{subarray}}\Wak^T_\chi(w^{-1}\circ0),
\end{align*}
where
\begin{align*}
\omega_F \colon \W^F(\g,f;\Gamma) \rightarrow \Wak^F_{\g, \chi}
\end{align*}
is an injective homomorphism of vertex algebras over $F$ induced from $\hat{\rho}_F$, and
\begin{align*}
Q_\alpha = \int Q_\alpha(z)\ dz \colon \Wak^F_{\g, \chi} \rightarrow \Wak^F_\chi(-\alpha)
\end{align*}
is the screening operator induced from $S_\alpha$. In particular, we have
\begin{align}\label{eq:W-sc F}
\W^F(\g,f;\Gamma) \simeq \Img\omega_F = \bigcap_{\alpha \in \Pi}\Ker\left( Q_\alpha \colon \Wak^F_{\g, \chi} \rightarrow \Wak^F_\chi(-\alpha) \right).
\end{align}
We call \eqref{W-alg Wakimoto resolution} the Wakimoto resolution of $\W^F(\g,f;\Gamma)$. Since the restriction of $\omega_F$ to $\W^U(\g,f;\Gamma)$ is induced from $\hat{\rho}_U$ and $H_\chi(\Wak^U_\g) = \Wak^U_{\g, \chi}$ by Proposition \ref{semi-inf Wak prop}, we also have an injective homomorphism
\begin{align*}
\omega_U =\omega_F|_{\W^U(\g,f;\Gamma)} \colon \W^U(\g,f;\Gamma) \rightarrow \Wak^U_{\g, \chi}
\end{align*}
of vertex algebras over $U$. Then
\begin{align}\label{eq:W-sc U}
\W^U(\g,f;\Gamma) \simeq \Img\omega_U \subset  \bigcap_{\alpha \in \Pi}\Ker\left( Q_\alpha \colon \Wak^U_{\g, \chi} \rightarrow \Wak^F_\chi(-\alpha) \right).
\end{align}
The map $\omega_T$ provides a $\W^T(\g,f;\Gamma)$-module structure on any $\Wak^T_{\g, \chi}$-module, which we call a {\em Wakimoto representation for a $\W$-algebra $\W^T(\g,f;\Gamma)$ over $T$}.

\subsection{Explicit forms of $Q_\alpha$}\label{Explicit forms of screenings}

Recall that $Q_\alpha$ is the screening operator induced from $S_\alpha$ that is defined in \eqref{scaffine eq}. Also recall that $P_\alpha^{\beta,R}(a^*)(z)$ is the field on $\Anil$ defined by \eqref{P(a) pol eq} for the polynomial $P_\alpha^{\beta,R}(x)$ in \eqref{right rep eq}, which depends on the choice of coordinates $c(\nil_+)$ on $N_+$. According to Lemma \ref{pol1 lemma}, it follows that $P_\alpha^{\beta,R}(a^*)(z)$ is a field on $\Azero$ if $\degG\alpha=\degG\beta$. In addition, recall that $[\alpha]$ is defined in \eqref{[a] def eq}.
\begin{lemma}\label{lem:PR zero}
Suppose that $c(\nil_+)$ is compatible with the decomposition $N_+=G_{>0}\times G_0^+$. Let $\alpha\in\Pi_{>0}$ and $\beta\in\Delta_+$ such that $\degG\alpha=\degG\beta$. Then $P_\alpha^{\beta,R}(x)=0$ unless $\beta\in[\alpha]$.
\begin{proof}
Assume that $P_\alpha^{\beta,R}(x)\neq0$ for $\alpha\in\Pi_{>0}$ and $\beta\in\Delta_+$ such that $\degG\alpha=\degG\beta$. We will show that $\beta\in[\alpha]$. Since $P_\alpha^{\beta,R}(x)$ is a polynomial in $\C[G_0^+]$ by Lemma \ref{pol1 lemma} and $\degQ P_\alpha^{\beta,R}(x)=\beta-\alpha$ by \eqref{polQ eq}, we have $\beta-\alpha\in\bQ_0$, that is, $\beta\in[\alpha]$. Thus, $P_\alpha^{\beta,R}(x)=0$ unless $\beta\in[\alpha]$.
\end{proof}
\end{lemma}
\begin{theorem}\label{main thm}
Suppose that the coordinate $c(\nil_+)$ on $N_{+}$ is compatible with the decomposition $N_{+}=G_{>0}\times G_0^+$. Then, we have
\begin{align*}
Q_\alpha =
\begin{cases}
S_\alpha & (\alpha\in\Pi_0),\\
\displaystyle \sum_{\beta\in[\alpha]}\int V_\beta(z)\Phi_\beta(z)\ dz & (\alpha\in\Pi_{\frac{1}{2}}),\\
\displaystyle \sum_{\beta\in[\alpha]}\chi(e_{\beta})\int V_\beta(z)\ dz & (\alpha\in\Pi_1),
\end{cases}
\end{align*}
where $V_\beta(z) =\ :P_\alpha^{\beta, R}(a^*)(z)\ \e^{-\frac{1}{\para+h^{\vee}}\int b_{\alpha}(z)}:$\ .
\begin{proof}
First, consider the action of $\rho^R(x)$ on $\C[N_+]$ for $x\in\g_{>0}$ under the assumption that $c(\nil_+)$ is compatible with $N_+ = G_{>0} \times G_0^+$. To compute the action of $\rho^R(x)$, we set a smooth curve $X=\e^{-t x}$ on $N_+$ for $|t|\ll1$. Given $(g_1,g_2)\in G_{>0}\times G_0^+\simeq N_+$, the right action $X^R$ of $X$ on $N_+$ can be computed as follows:
\begin{align*}
X^R\cdot(g_1,g_2)=g_1 g_2 X=g_1 g_2 X g_2^{-1} g_2=(\Ad_{g_2}(X)^R\cdot g_1, g_2)=(\e^{-t\Ad_{g_2}(x)}\cdot g_1,g_2).
\end{align*}
Applying for $g_1$ (resp. $g_2$) whose entry in the coefficient corresponding to $e_\alpha$ is the indeterminate $x_\alpha$ for each $\alpha\in\Delta_{>0}$ (resp. $\alpha\in\Deltazero$), we have the differential $\rho^R(x)$ of the action $X^R$ as follows:
\begin{align*}
\rho^R(e_\alpha)=\sum_{\beta\in\Delta_{>0}}T_\alpha^\beta(x)e_\beta^R
\end{align*}
for $\alpha\in\Delta_{>0}$, where $T_\alpha^\beta(x)\in\C[G_0^+]$ with $\degQ T_\alpha^\beta(x)=\beta-\alpha$, and $e_\beta^R$ denotes the right action of $e_\beta$ on $\C[G_{>0}]$ induced from a right $G_{>0}$-action on itself. Recall that $\Am^T$ is a $V^T(\g_{>0})$-module through the Wakimoto representation of $V^T(\g)$ and is isomorphic to the semi-regular bimodule of the loop algebra $\g_{>0}((t))\otimes T$ over $T$. See Section \ref{semi-inf sec}. Then we have
\begin{align}\label{eq:right-action diff}
\hat{\rho}^R(e_\alpha(z))=\sum_{\beta\in\Delta_{>0}}:T_\alpha^\beta(a^*)(z)e_\beta^R(z):,\quad
\alpha\in\Delta_{>0},
\end{align}
where $e_\beta^R(z)$ is the right action on $\Am^T$.

Next, recall that the exact sequence \eqref{W-alg Wakimoto resolution} is obtained by \eqref{eq:W-alg vanish} and \eqref{eq:W-Wak vanish}, and the isomorphism \eqref{eq:W-Wak vanish} is exactly $[\Psi_\lambda]$ in Proposition \ref{semi-inf Wak prop}. The map $\Psi_\lambda$ is the special case of Lemma \ref{ACL lemma} for $M = \Wak^T_\chi(\lambda)$ and satisfies the properties in Lemma \ref{ACL lemma}. Then, by Lemma \ref{ACL lemma} and \eqref{eq:right-action diff},
\begin{align*}
&\Psi_{-\alpha}\circ S_\alpha(z)=\\
&\left\{S_\alpha(z)-:\left(\sum_{\beta\in\Delta_{\frac{1}{2}}}T_\alpha^\beta(a^*)(z)\Phi_\beta(z)+\sum_{\beta\in\Delta_1}\chi(e_\beta)T_\alpha^\beta(a^*)(z)\right)\e^{-\frac{1}{\para+h^{\vee}}\int b_{\alpha}(z)}:\right\}\circ\Psi_0.
\end{align*}
Thus, the cohomology functor $H_\chi(?)$ induces
\begin{align*}
&[\Psi_{-\alpha}]\circ [S_\alpha(z)]=\\
&\left[S_\alpha(z)-:\left(\sum_{\beta\in\Delta_{\frac{1}{2}}}T_\alpha^\beta(a^*)(z)\Phi_\beta(z)+\sum_{\beta\in\Delta_1}\chi(e_\beta)T_\alpha^\beta(a^*)(z)\right)\e^{-\frac{1}{\para+h^{\vee}}\int b_{\alpha}(z)}:\right]\circ [\Psi_0],
\end{align*}
where $[X(z)]$ denotes the intertwining operator induced from $X(z)$ by $H_\chi(?)$. Therefore, by weight considerations,
\begin{align*}
Q_\alpha(z)&=\left[S_\alpha(z)-:\left(\sum_{\beta\in\Delta_{\frac{1}{2}}}T_\alpha^\beta(a^*)(z)\Phi_\beta(z)+\sum_{\beta\in\Delta_1}\chi(e_\beta)T_\alpha^\beta(a^*)(z)\right)\e^{-\frac{1}{\para+h^{\vee}}\int b_{\alpha}(z)}:\right]\\
&=\begin{cases}
\ \displaystyle{S_\alpha(z)} & \mathrm{if}\ \alpha\in\Pi_0,\\
\displaystyle{-\sum_{\beta\in\Delta_{\frac{1}{2}}}:T_\alpha^\beta(a^*)(z)\Phi_\beta(z)\e^{-\frac{1}{\para+h^{\vee}}\int b_{\alpha}(z)}:} & \mathrm{if}\ \alpha\in\Pi_{\frac{1}{2}}, \\
\displaystyle{-\sum_{\beta\in\Delta_1}\chi(e_\beta):T_\alpha^\beta(a^*)(z)\e^{-\frac{1}{\para+h^{\vee}}\int b_{\alpha}(z)}:} & \mathrm{if}\ \alpha\in\Pi_1.
\end{cases}
\end{align*}

Furthermore, notice that $\rho^R(x) \in \D_{G_0^+}$ for all $x \in \g_0^+$ by Lemma \ref{pol0 lemma} and so $\hat{\rho}^R(x(z)) \in \Azero^T$ for all $x \in \g_0^+$, thus $S_\alpha$ for $\alpha \in \Pi_0$ is well-defined on $\Wak^T_\g = \Azero^T \otimes \Hi^T \otimes \FneT$.
By definition,
\begin{align*}
\sum_{\beta\in\Delta_{>0}}T_\alpha^\beta(x)e^R_\beta=\sum_{\beta\in\Delta_{>0}}P_\alpha^{\beta,R}(x)\der_\beta
\end{align*}
for all $\alpha\in\Delta_{>0}$, and
\begin{align*}
e^R_\beta=\der_\beta(z)+\sum_{\gamma\in\Delta_{>\degG\beta}}\tilde{T}_\beta^\gamma(x)\der_\gamma,\quad
\tilde{T}_\beta^\gamma(x)\in\C[G_{>0}]
\end{align*}
with $\degQ\tilde{T}_\beta^\gamma(x)=\gamma-\beta$. This implies that
\begin{align*}
T_\alpha^\beta(x)=P_\alpha^{\beta,R}(x)
\end{align*}
for all $\alpha\in\Pi_{>0}$ and $\beta\in\Delta_{\degG\alpha}$. Thus, $Q_\alpha=\int Q_\alpha(z)dz$ may be described as follows:
\begin{align*}
Q_\alpha =
\begin{cases}
S_\alpha & (\alpha\in\Pi_0),\\
\displaystyle -\sum_{\beta\in\Deltahalf}\int:P_{\alpha}^{\beta,R}(a^*)(z)\Phi_\beta(z)\ \e^{-\frac{1}{\para+h^{\vee}}\int b_{\alpha}(z)}:dz & (\alpha\in\Pi_{\frac{1}{2}}),\\
\displaystyle -\sum_{\beta\in\Deltaone}\chi(e_{\beta})\int:P_{\alpha}^{\beta,R}(a^*)(z)\ \e^{-\frac{1}{\para+h^{\vee}}\int b_{\alpha}(z)}:dz & (\alpha\in\Pi_1).
\end{cases}
\end{align*}

Finally, by Lemma \ref{lem:PR zero}, $P_\alpha^{\beta,R}(a^*)(z)=0$ for all $\alpha\in\Pi_{i}$ and $\beta\in\Delta_i\backslash[\alpha]$ with $i=\frac{1}{2},1$, and thus we may restrict the summation range in $Q_\alpha$ to $\{\beta\in[\alpha]\}$ for all $\alpha\in\Pi_{>0}$. Since $\Ker Q_\alpha=\Ker(-Q_\alpha)$, we may replace $Q_\alpha$ by $-Q_\alpha$ for all $\alpha\in\Pi_{>0}$. This completes the proof.
\end{proof}
\end{theorem}
Let $\Vtau$ be the affine vertex algebra associated with $\g_0$ and its invariant bilinear form $\tau_k$ defined by
\begin{align}\label{eq:tauk-def}
\tau_k(u|v)=k(u|v)+\frac{1}{2}\killing_\g(u|v)-\frac{1}{2}\killing_{\g_0}(u|v),\quad
u, v \in\g_0,
\end{align}
where $\killing_\g$, $\killing_{\g_0}$ are the Killing forms on $\g$, $\g_0$ respectively. Denote by $\VtauT$ instead of $V^{\tau_\para}(\g_0)$ when the base ring (or field) is $T$. By \cite{F}, there exists an exact sequence
\begin{align*}
0\rightarrow V^F(\g_0) \xrightarrow{\hat{\rho}_{\g_0, F}}\Azero^F\otimes\Hi^F\xrightarrow{\bigoplus S_\alpha}\bigoplus_{\alpha\in\Pi_0}\Azero^F\otimes\Hi^F_{-\alpha},
\end{align*}
where
\begin{align}\label{para wak eq}
\hat{\rho}_{\g_0, F}\colon V^F(\g_0) \rightarrow \Azero^F \otimes \Hi^F
\end{align}
is an injective vertex algebra homomorphism over $F$, called a Wakimoto representation of $V^F(\g_0)$, defined by
\begin{align*}
\hat{\rho}_{\rf, F}(e_{\alpha}(z))&=\sum_{\beta\in\Deltazero}:P_{\alpha}^{\beta}(a^*)(z)a_\beta(z):,\\
\hat{\rho}_{\rf, F}(h_{\alpha'}(z))&=-\sum_{\beta\in\Deltazero}\beta(h_{\alpha'}):a^*_{\beta}(z)a_\beta(z):+b_{\alpha'}(z),\\
\hat{\rho}_{\rf, F}(f_{\alpha}(z))&=\sum_{\beta\in\Deltazero}:Q_{\alpha}^{\beta}(a^*)(z)a_\beta(z):+:b_\alpha(z)a^*_\alpha(z):+((e_\alpha|f_\alpha)\para+c'_\alpha)\der a^*_\alpha(z)
\end{align*}
for all $\alpha\in\Pi_0$ and $\alpha'\in\Pi$. Here $P_{\alpha}^{\beta}(a^*)(z)$ and $Q_{\alpha}^{\beta}(a^*)(z)$ are determined by the coordinate $c(\g_0^+)$ on $G_0^+$. Thus, we have
\begin{align}\label{para wak eq1}
V^F(\g_0) \simeq \Img\hat{\rho}_{\g_0, F} = \bigcap_{\alpha\in\Pi_0}\Ker\left(S_\alpha \colon \Azero^F\otimes\Hi^F \rightarrow \Azero^F\otimes\Hi^F_{-\alpha}\right).
\end{align}
Let $\hat{\rho}_{\g_0, U} = \hat{\rho}_{\g_0, F}|_{V^U(\g_0)}$ be the restriction of $\hat{\rho}_{\g_0, F}$ to $V^U(\g_0)$. By construction, it is clear that the image of $\hat{\rho}_{\g_0, U}$ is in $\Azero^U\otimes\Hi^U$. Using the same proof as that of Proposition \ref{affine Wak U}, it follows that
\begin{align}\label{para wak eq2}
V^U(\g_0) \simeq \Img\hat{\rho}_{\g_0, U} = \bigcap_{\alpha\in\Pi_0}\Ker\left(S_\alpha \colon \Azero^U\otimes\Hi^U \rightarrow \Azero^F\otimes\Hi^F_{-\alpha}\right).
\end{align}
\begin{cor}\label{Miura const lem}
Suppose that the coordinate $c(\nil_+)$ on $N_{+}$ is compatible with the decomposition $N_{+}=G_{>0}\times G_0^+$. Then
\begin{align*}
\VtauT\otimes\FneT \simeq \bigcap_{\alpha\in\Pi_0}\Ker\left(Q_\alpha\colon\Wak^T_{\g, \chi}\rightarrow\Wak^F_\chi(-\alpha)\right).
\end{align*}
\begin{proof}
Since $Q_\alpha=S_\alpha$ by Theorem \ref{main thm} and acts as $0$ on $\FneT$ for all $\alpha\in\Pi_0$, the assertion is immediate from \eqref{para wak eq1} for $T = F$ and \eqref{para wak eq2} for $T = U$.
\end{proof}
\end{cor}
Remark that the specialization $\hat{\rho}_{\g_0, k} = \hat{\rho}_{\g_0, U} \otimes_U \C_k$ of $\hat{\rho}_{\g_0, U}$ is also known to be injective by \cite{F}.

\subsection{Compatibility with Miura maps}\label{sec:compati Miura}

Since $\Img\omega_F = \bigcap_{\alpha\in\Pi}\Ker Q_\alpha|_{\Wak^F_{\g, \chi}}$ by \eqref{eq:W-sc F}, we have an injective homomorphism
\begin{align*}
\widetilde{\mu}_F \colon \W^F(\g,f;\Gamma) \hookrightarrow \bigcap_{\alpha\in\Pi_0}\Ker Q_\alpha|_{\Wak^F_{\g, \chi}} \simeq V^F(\g_0)\otimes\FneF
\end{align*}
of vertex algebras over $F$. Recall that we introduce the Miura maps in Section \ref{sec:Miura maps}. We will show the following lemma in Appendix \ref{appendix}:
\begin{lemma}\label{Miura lemma}
The map $\widetilde{\mu}^F$ is the Miura map $\mu_F$.
\end{lemma}
By Lemma \ref{Miura lemma} and \eqref{eq:Miura U k property}, the restriction $\widetilde{\mu}_F|_{\W^U(\g, f ;\Gamma)}$ of $\widetilde{\mu}_F$ to $\W^U(\g, f ;\Gamma)$ is the Miura map $\mu_U$ of $\W^U(\g, f ;\Gamma)$. Note that
\begin{align*}
\omega_T = (\hat{\rho}_{\g_0, T}\otimes\Id) \circ \mu_T.
\end{align*}
Then by \eqref{eq:W-sc F} and \eqref{para wak eq1}, we have
\begin{align*}
\W^F(\g, f;\Gamma) \simeq \Img\mu_F = \bigcap_{\alpha \in \Pi_{>0}}\Ker\left(Q_\alpha \colon V^F(\g_0)\otimes\FneF \rightarrow \Wak^F_\chi(-\alpha)\right).
\end{align*}
Moreover, by \eqref{eq:W-sc U} and \eqref{para wak eq2}, we have
\begin{align*}
\W^U(\g, f;\Gamma) \simeq \Img\mu_U \subset \bigcap_{\alpha \in \Pi_{>0}}\Ker\left(Q_\alpha \colon V^U(\g_0)\otimes\FneU \rightarrow \Wak^F_\chi(-\alpha)\right).
\end{align*}
\begin{prop}
\begin{align*}
\Img\mu_U = \bigcap_{\alpha \in \Pi_{>0}}\Ker\left(Q_\alpha \colon V^U(\g_0)\otimes\FneU \rightarrow \Wak^F_\chi(-\alpha)\right).
\end{align*}
\begin{proof}
As mentioned in Section \ref{W-alg T sec}, there exists a set $\{W^i(z)\}_{i=1}^{\dim\g^f}$ of fields on $V^T(\g_{\geq0})\otimes\FneT$ that generates $\W^T(\g,f;\Gamma)$ and gives rise to a PBW basis $W^{i_1}_{(-n_1)} \cdots W^{i_s}_{(-n_s)}|0\rangle$ over $T$ of $\W^T(\g,f;\Gamma)$ with some orders on $i_j$ and $n_j \in \Z_{\geq1}$. We denote by $\{ v^i\}_{i\in \Lambda}$ the PBW basis of $\W^T(\g,f;\Gamma)$ whose index set $\Lambda$ is independent of the choice of $T$. Also recall that the Miura map $\mu_k = \mu_U\otimes_U\C_k$ of $\W^k(\g,f;\Gamma)$ is injective for all $k\in\C$, see Section \ref{sec:Miura maps}. Then, though we only consider $V^U(\g)$ in Proposition \ref{affine Wak U}, the same proof applies by using the PBW basis $\{ v^i\}_{i\in \Lambda}$ of $\W^T(\g,f;\Gamma)$ and the injectivity of $\mu_k$. This completes the proof.
\end{proof}
\end{prop}
\begin{cor}\label{main cor}
\begin{align*}
\W^T(\g, f;\Gamma) \simeq \Img\omega_T = \bigcap_{\alpha \in \Pi}\Ker\left(Q_\alpha \colon \Wak^T_{\g, \chi} \rightarrow \Wak^F_\chi(-\alpha)\right).
\end{align*}
\end{cor}
\begin{cor}\label{Miura cor}
The specialization of a natural embedding
\begin{align*}
\bigcap_{\alpha\in\Pi}\Ker Q_\alpha|_{\Wak^U_{\g, \chi}} \hookrightarrow \bigcap_{\alpha \in \Pi_0}\Ker Q_\alpha|_{\Wak^U_{\g, \chi}}
\end{align*}
coincides with the Miura map $\mu_k$.
\end{cor}
Let
\begin{align*}
\omega_k = \omega_U\otimes_U\C_k \colon \W^k(\g,f;\Gamma) \rightarrow \Wak_{\g,\chi}
\end{align*}
be the specialization of $\omega_U$. Then $\omega_k = \hat{\rho}_{\g_0, k} \circ \mu_k$.
\begin{cor}
$\omega_k$ is injective for all $k \in \C$.
\begin{proof}
As both $\hat{\rho}_{\g_0, k}$ and $\mu_k$ are injective, so is $\omega_k$.
\end{proof}
\end{cor}
The map $\omega_k$ provides a $\W^k(\g, f;\Gamma)$-module structure on any $\Wak_{\g, \chi}$-module, which we call a Wakimoto representation for $\W^k(\g,f;\Gamma)$.
\smallskip

Let $f=f_{\mathrm{prin}}=\sum_{\alpha\in\Pi}e_{-\alpha}$ be a principal nilpotent element in $\g$ and $\Gamma$ the Dynkin grading on $\g$ for $f$. Then $\Pi=\Pi_1=\Delta_1$ and $\chi(e_\alpha)=1$ for all $\alpha\in\Pi$. By Lemma \ref{pol2 lemma}, $P_\alpha^{\beta,R}(z)=\delta_{\alpha,\beta}$ for all $\alpha,\beta\in\Pi$. By Theorem \ref{main thm} and Corollary \ref{main cor},
\begin{align*}
\W^T(\g,f_{\mathrm{prin}};\Gamma)\simeq\bigcap_{\alpha\in\Pi}\Ker \int:\e^{-\frac{1}{\para+h^{\vee}}\int b_{\alpha}(z)}:dz,
\end{align*}
which is a well-known result given in \cite{FF4}, see also \cite{FBZ}.

In the case that $\g_0=\h$, we have $\Pi=\Pi_{\frac{1}{2}}\sqcup\Pi_1$ and, by Lemma \ref{pol2 lemma}, $P_\alpha^{\beta,R}(z)=\delta_{\alpha,\beta}$ for all $\alpha\in\Pi_i$, $\beta\in\Delta_i$ and $i=\frac{1}{2},1$. By Theorem \ref{main thm} and Corollary \ref{main cor}, $\W^T(\g,f;\Gamma)$ is isomorphic to
\begin{align*}
\bigcap_{\alpha\in\Pi_{\frac{1}{2}}}\Ker \int:\Phi_\alpha(z)\ \e^{-\frac{1}{\para+h^{\vee}}\int b_{\alpha}(z)}:dz
\cap
\bigcap_{\begin{subarray}{c} \alpha\in\Pi_1 \\ \chi(e_\alpha)\neq0 \end{subarray}}\Ker \int:\e^{-\frac{1}{\para+h^{\vee}}\int b_{\alpha}(z)}:dz,
\end{align*}
which is a result previously obtained in \cite{G} in case that $T = F$.
\smallskip

We give several remarks on the topics in Section \ref{W-alg Wak sec}. First, note that $\Azero\otimes\Fne$ coincides with $\frac{1}{2}\dim(\mathcal{N}\cap\mathcal{S}_f)$ copies of the $\beta\gamma$-system, where $\mathcal{N}$ is the nilpotent cone of $\g$, $\mathcal{S}_f$ is the Slodowy slice of $\g$ through $f$, since
\begin{align*}
\frac{1}{2}\dim(\mathcal{N}\cap\mathcal{S}_f)=\frac{1}{2}(\dim\g_0+\dim\g_{\frac{1}{2}}-\dim\h)=\dim\g_0^++\frac{1}{2}\dim\g_{\frac{1}{2}}.
\end{align*}
Next, we have the following commutative diagram that summarizes the correspondence between the screening operators $Q_\alpha$ and $\hQ_\alpha$, see Appendix \ref{appendix}:
\begin{align*}
\SelectTips{cm}{}
\xymatrix@W15pt@H15pt@R15pt@C12pt{
\W^T(\g,f;\Gamma) \ar[rr]^(0.4){\omega_T} \ar[d]^{\mu_T} & & \Wak^T_{\g,\chi} \ar@{}[d]|{\|} \\
\VtauT\otimes\FneT \ar[rr]^(0.45){\hat{\rho}_{\rf, T}\otimes\Id} \ar[d]^{\bigoplus\hQ_\alpha} & & \Wak^T_{\g, \chi} \ar[d]^{\bigoplus Q_\alpha} \\
\displaystyle{\bigoplus_{\alpha\in\Pi_{>0}}\Weyl^F_0(-\alpha)\otimes\FneF} \ar@{}[r]|(0.50){\displaystyle{\hookrightarrow}} & \displaystyle{\bigoplus_{\alpha\in\Pi_{>0}}\Wak^F_0(-\alpha)\otimes\FneF} \ar@{}[r]|(0.60){\displaystyle{=}} & \displaystyle{\bigoplus_{\alpha\in\Pi_{>0}}\Wak^F_\chi(-\alpha)}
}
\end{align*}
Finally, remark that all statements for $\W^F(\g,f;\Gamma)$ (and $V^F(\g)$) hold if we replace $\para$ with a generic level $k\in\C$. Then we have the Wakimoto resolution and free fields realizations of $\W^k(\g,f;\Gamma)$ as common kernels of screening operators $Q_\alpha$ for generic $k\in\C$ from \eqref{W-alg Wakimoto resolution} and \eqref{eq:W-sc F}.

\section{Parabolic inductions}

In this section, we state and prove our main theorem (Theorem \ref{induced thm}). From now on, we assume that $\g$ is a reductive Lie algebra.

\subsection{$\W$-algebras for reductive Lie algebras}\label{reductive W-alg sec}
Let $\g$ be a finite-dimensional reductive Lie algebra, $f$ a nilpotent element in $[\g,\g]$, $\kappa$ a symmetric invariant bilinear form on $\g$ and $\Gamma$ a good grading for $f$ on $\g$ satisfying that the center $\z_\g$ of $\g$ lies in $\g_0$. The definition of $\W$-algebras $\W^\kappa(\g,f;\Gamma)$ naturally extends for $\g,f,\Gamma$ and $\kappa$. We use the same notations: $\g_i$, $\Delta_i$, $\Pi_i$, $\Deltazero$, $\h$, $\nil_+$, $\chi$ as in Section \ref{W-alg sec} and \ref{local sec}. Set
\begin{align}\label{reductive dec eq}
\g=\z_\g\oplus\bigoplus_{i=1}^m\g^{i},
\end{align}
where $\g^{i}$ is a simple Lie algebra. Let $\g^i_j=\g^i\cap\g_j$ and $f_i\in\g^i$ such that $f=\sum_{i=1}^m f_i$ corresponding to \eqref{reductive dec eq}. Then
\begin{align*}
\Gamma_i:\g^i=\bigoplus_{j\in\frac{1}{2}\Z}\g^i_j
\end{align*}
is good for $f_i$. We have an isomorphism of vertex algebras
\begin{align}\label{reductive W-alg eq}
\W^\kappa(\g,f;\Gamma)\simeq V^\kappa(\z_\g)\otimes\bigotimes_{i=1}^m\W^{k_i}(\g^i,f_i;\Gamma_i),
\end{align}
where $k_i=\kappa(\theta_i|\theta_i)/2\in\C$ for the highest root $\theta_i$ in $\g^i$. Let $\h^i$ be a Cartan subalgebra of $\g^i$ contained in $\g^i_0$ and $h^\vee_i$ the dual Coxter number of $\g^i$. We have $\h=\z_\g\oplus\bigoplus_{i=1}^m\h^i$. Denote by $\Delta^i$, $\Delta_+^i$, $\Pi^i$ the sets of roots, positive roots and simple roots in $\g^i$ respectively. Let $\nil_+^i=\bigoplus_{\alpha\in\Delta_+^i}\g_\alpha$ and $\g_0^{+,i}=\g_0^i\cap\nil_+^i$. We also have $\nil_+=\bigoplus_{i=1}^m\nil_+^i$ and $\g_0^+=\bigoplus_{i=1}^m\g_0^{+,i}$. Set $\chi^i(u) = (f_i|u)$, $\Delta_j^i=\Delta_j\cap\Delta^i$, $\Pi_j^i=\Pi_j\cap\Delta^i$ and $(\Delta_0^i)^+=\Delta_0^i\cap\Delta_+^i$. Define
\begin{align*}
&V^{\tau_\kappa}(\rf)=V^\kappa(\z_\g)\otimes\bigotimes_{i=1}^m V^{\tau_{k_i}}(\g_0^i),\quad
\Wak_\chi(\lambda) = \Azero \otimes \Hi^\kappa_\lambda \otimes \Fne,\\
&\Azero = \bigotimes_{i=1}^m\mathcal{A}_{(\Delta_0^i)^+},\quad
\Hi^\kappa_\lambda = V^\kappa(\z_\g)\otimes\bigotimes_{i=1}^m\Hi^{k_i+h^\vee_i}_{\lambda^i}(\h^i),\quad
\Fne = \bigotimes_{i=1}^m \Phi(\g^i_{\frac{1}{2}})
\end{align*}
for all $\lambda \in \h^:$, where $\lambda^i = \lambda|_{\h^i}$. Set $\Wak_{\g, \lambda} = \Wak_\chi(0)$. Let $U_i=\C[\para_i]$ be the polynomial ring with a formal parameter $\para_i$. By abuse of notation, we denote by $U=\bigotimes_{i=1}^m U_i$. By Corollary \ref{main cor},
\begin{align*}
\W^{U_i}(\g^i,f_i;\Gamma_i)\simeq\bigcap_{\alpha\in\Pi^i}\Ker\left(Q_\alpha \colon \Wak^{U_i}_{\g^i,\chi^i} \rightarrow \Wak^{U_i}_{\chi^i}(-\alpha)\right).
\end{align*}
Now, we replace $k_i$ with the formal parameter $\para_i$ and extend the base field $U_i$ to $U$. Then we define the $\W$-algebra $\W^U(\g,f;\Gamma)$ over $U$, the affine vertex algebra $V^U(\g_0)$ over $U$, the Wakimoto modules $\Wak^U_\chi(\lambda)$ of $\W^U(\g,f;\Gamma)$. Set $\Wak^U_{\g, \chi} = \Wak^U_\chi(0)$. Since $V^U(\z_\g)$ commutes with all $Q_\alpha$ and $\Wak^{U_i}_{\chi^i}(-\alpha) \subset \Wak^U_\chi(-\alpha)$, we have an isomorphism
\begin{align}\label{reductive screening eq}
\W^U(\g,f;\Gamma) \simeq \bigcap_{\alpha\in\Pi} \Ker\left(Q_\alpha \colon \Wak^U_{\g, \chi} \rightarrow \Wak^U_\chi(-\alpha)\right)
\end{align}
of vertex algebras over $U$. Define the Miura map $\mu_\kappa$ of $\W^\kappa(\g,f;\Gamma)$ by
\begin{align*}
\mu_\kappa = \Id \otimes \bigotimes_{i=1}^m \mu_{k_i} \colon \W^\kappa(\g,f;\Gamma) \rightarrow V^{\tau_\kappa}(\rf)\otimes\Fne.
\end{align*}
Then, by Corollary \ref{Miura cor}, $\mu_\kappa$ coincides with the map induced by the specialization of a natural embedding
\begin{align*}
\bigcap_{\alpha\in\Pi}\Ker Q_\alpha |_{\Wak^U_{\g,\chi}}\hookrightarrow\bigcap_{\alpha\in\Pi_0}\Ker Q_\alpha|_{\Wak^U_{\g,\chi}}.
\end{align*}

\subsection{Induced nilpotent orbits}\label{ind nil sec}
Let $\mathcal{N}$ be the set of all nilpotent elements in $[\g,\g]$. A Lie group $G$ acts on $\mathcal{N}$ by the adjoint action, which decompose $\mathcal{N}$ into finitely many orbits, called nilpotent orbits in $\g$. See e.g. \cite{CM}. Let $\p$ be a parabolic subalgebra, i.e. $\bo\subset\p$. There exists the Levi decomposition $\p=\lf\oplus\uf$ such that $\lf$ is a reductive Lie subalgebra and $\uf$ is a nilpotent subalgebra. We have a root subsystem $\Delta_\lf\subset\Delta$ such that
\begin{align}\label{Levi root eq}
\lf=\h\oplus\bigoplus_{\alpha\in\Delta_\lf}\g_\alpha.
\end{align}
The reductive Lie subalgebra $\lf$ is called a Levi subalgebra of $\g$ and uniquely determined by simple roots $\Pi_\lf$ of $\Delta_\lf$ up to conjugation. Denote by $P$, $L$ the Lie subgroups of $G$ corresponding to $\p$, $\lf$ respectively. The following results are due to Lusztig and Spaltenstein \cite{LS}.
\begin{propdef}[\cite{LS}]
Let $\Oc_\lf$ be a nilpotent orbit in $\lf$. Then there exists a unique nilpotent orbit $\Oc_\g$ in $\g$ such that $(\Oc_\lf+\uf)\cap\Oc_\g$ is Zariski dense in $\Oc_\lf+\uf$, and $\Oc_\g$ doesn't depend on the choice of $\p$. The orbit $\Oc_\g$ is called the induced nilpotent orbit from $\Oc_\lf$ and denoted by $\ind^\g_\lf\Oc_\lf$.
\end{propdef}

\begin{prop}[\cite{LS}]\label{LS prop}
Let $\Oc_\lf$ be a nilpotent orbit in $\lf$ and $\Oc_\g=\ind^\g_\lf\Oc_\lf$ the induced nilpotent orbit from $\Oc_\lf$.
\begin{enumerate}
\item $\Oc_\g$ is a unique nilpotent orbit that has the dimension $\dim\Oc_\g=\dim\Oc_\lf+2\dim\uf$ and $(\Oc_\lf+\uf)\cap\Oc_\g\neq\phi$.
\item Induced nilpotent orbits are transitive, i.e.
\begin{align*}
\ind^\g_\lf\Oc_\lf=\ind^\g_{\lf'}\ind^{\lf'}_\lf\Oc_\lf
\end{align*}
for any Levi subalgebra $\lf'$ such that $\lf\subset\lf'\subset\g$.
\end{enumerate}
\end{prop}

To prove Lemma \ref{ind good lem}, we recall the properties of (good) gradings in \cite{EK}.
\begin{lemma}[\cite{EK}]\label{good pro lem}
Let $\g$ be a reductive Lie algebra, $f$ a nilpotent element of $[\g,\g]$. Let $\Gamma$ be a $\frac{1}{2}\Z$-grading on $\g$ such that $f\in\g_{-1}$, the center of $\g$ lies in $\g_0$ and $[\g_i,\g_j]\subset\g_{i+j}$ for all $i,j\in\frac{1}{2}\Z$. 
\begin{enumerate}
\item The following are equivalent.
\begin{enumerate}
\item $\ad(f)\colon\g_j\rightarrow\g_{j-1}$ is injective for $j\geq\frac{1}{2}$.
\item $\ad(f)\colon\g_j\rightarrow\g_{j-1}$ is surjective for $j\leq\frac{1}{2}$.
\item $\Gamma$ is good for $f$.
\end{enumerate}
\item Suppose that $\Gamma$ is good for $f$. Then $\dim\g^f=\dim\g_0+\dim\g_{\frac{1}{2}}$.
\end{enumerate}
\end{lemma}

\begin{lemma}\label{ind good lem}
Let $\Gamma$ be a good grading for $f$ on $\g$, $G\cdot f$ the nilpotent orbit in $\g$ through $f$, $\lf$ a Levi subalgebra of $\g$ with simple roots $\Pi_\lf$. Suppose that $\degG\alpha=1$ for all $\alpha\in\Pi\backslash\Pi_\lf$. Then there exists a nilpotent element $f_\lf$ in $[\lf,\lf]$ such that $\Gamma_{\lf}$ is a good grading for $f_\lf$ and $G\cdot f=\ind^\g_\lf L\cdot{f_\lf}$, where $\Gamma_{\lf}$ is the restriction of $\Gamma$ to $\lf$ and $L\cdot f_\lf$ is the nilpotent orbit in $\lf$ through $f_\lf$.
\end{lemma}
\begin{proof}
As in Section \ref{reductive W-alg sec}, there exists a root decomposition $\g=\h\oplus\bigoplus_{\alpha\in\Delta}\g_\alpha$ compatible with $\Gamma$. We may choose $\Delta_+$ such that $\nil_+=\bigoplus_{\alpha\in\Delta_+}\g_\alpha\subset\g_{\geq0}$. We have a root subsystem $\Delta_\lf$ of $\Delta$ satisfying \eqref{Levi root eq}. Let $\uf=\bigoplus_{\alpha\in\Delta_{-}\backslash(\Delta_{-}\cap\Delta_\lf)}\g_\alpha$. Then $\p=\lf\oplus\uf$ is a parabolic subalgebra including the opposite Borel subalgebra $\bo_-$ and gives the Levi decomposition of $\p$ whose Levi subalgebra is $\lf$. Denote by $\lf_j=\lf\cap\g_j$ and by $\uf_j=\uf\cap\g_j$. Since $\Pi\backslash\Pi_\lf\subset\Pi_1$, we have $\g_j=\lf_j\oplus\uf_j$ for all $j\leq\frac{1}{2}$. Choose $f_\lf\in\lf_{-1}$ and $f_\uf\in\uf_{-1}$ such that $f=f_\lf+f_\uf$ corresponding to $\g_{-1}=\lf_{-1}\oplus\uf_{-1}$. Since $[f,\g_j]=\g_{j-1}$ for all $j\leq\frac{1}{2}$ and $[\lf,\uf]\subset\uf$, we have $[f_\lf,\lf_j]=\lf_{j-1}$ for all $j\leq\frac{1}{2}$. Note that the center of $\lf$ lies in $\h\subset\g_0$ and the formula $[\lf_i,\lf_j]\subset\lf_{i+j}$ is deduced from $[\g_i,\g_j]\subset\g_{i+j}$. These imply that $\Gamma_\lf$ is good for $f_\lf$ by Lemma \ref{good pro lem} (1). Therefore $\dim\lf^{f_\lf}=\dim\lf_0+\dim\lf_{\frac{1}{2}}$ by Lemma \ref{good pro lem} (2). Since $\g_j=\lf_j$ for $j=0,\frac{1}{2}$, we have $\dim\lf^{f_\lf}=\dim\g^f$. Hence,
\begin{align*}
\dim G\cdot f=\dim\g-\dim\g^f=\dim\lf+2\dim\uf-\dim\lf^{f_\lf}=\dim L\cdot{f_\lf}+2\dim\uf.
\end{align*}
By construction, $f\in G\cdot f\cap(L\cdot{f_\lf}+\uf)\neq\phi$. Therefore $G\cdot f=\ind^\g_\lf L\cdot{f_{\lf}}$ by Proposition \ref{LS prop}.
\end{proof}

\begin{cor}\label{ind chi cor}
Under the conditions in Lemma \ref{ind good lem}, $\chi(u)=(f_\lf|u)$ for all $u\in\lf$.
\end{cor}
\begin{proof}
We use the notations in the proof of Lemma \ref{ind good lem}. Since $(\lf\mid\uf)=0$, we have $\chi(u)=(f_\lf|u)+(f_\uf|u)=(f_\lf|u)$ for all $u\in\lf$.
\end{proof}

\begin{rem}\label{ind rem}
The condition $\Pi\backslash\Pi_\lf\subset\Pi_1$ in Lemma \ref{ind good lem} is valid for all cases of type $A$ by \cite{Kr, OW}, rectangular nilpotent cases of type $BCD$ by \cite{Ke, Sp}, all cases of type $G$ and many cases of other exceptional types by \cite{GE}.
\end{rem}

\subsection{Preliminary results}\label{pre sec}
Continue to use the notations in Section \ref{reductive W-alg sec} and \ref{ind nil sec}. Under the condition $\Pi\backslash\Pi_\lf\subset\Pi_1$, we have a nilpotent element $f_\lf$ in $[\lf,\lf]$ such that $\Gamma_{\lf}$ is good for $f_\lf$ by Lemma \ref{ind good lem}. Set $\nil_+^\lf=\lf\cap\nil_+$ and $\lf_0^+=\lf_0\cap\nil_+$. Denote by $(\Delta_\lf)_+=\Delta_\lf\cap\Delta_+$ and by $(\Delta_\lf)_j=\Delta_\lf\cap\Delta_j$. Let $N_+^\lf$, $L_0^+$, $L_{>0}$ be the Lie subgroup of $L$ corresponding to $\nil_+^\lf$, $\lf_0^+$, $\lf_{>0}$ respectively. Since $\g_0=\lf_0$, we have $L_0^+=G_0^+$. Let $c(\nil_+)=c(\g_{>0})\cdot c(\g_0^+)$ be a coordinate on $N_+$ compatible with the decomposition $N_+=G_{>0}\times G_0^+$. Then $c(\nil_+^\lf)=c(\nil_+)|_{N_+^\lf}=c(\g_{>0})|_{L_{>0}}\cdot c(\g_0^+)$ is a coordinate on $N_+^\lf$ compatible with the decomposition $N_+^\lf=L_{>0}\times L_0^+$. Let $\rho^R_\lf\colon\nil_+^\lf\rightarrow\D_{N_+^\lf}$ be the anti-homomorphism induced by the right action of $N_+^\lf$ on itself. Then $\rho^R_\lf(u)=\rho^R(u)$ for all $u\in\nil_+^\lf$ as differentials on $\C[N_+^\lf]$ by construction. Set
\begin{align*}
\rho^R_\lf(e_\alpha)=\sum_{\beta\in(\Delta_\lf)_+}P_{\alpha,\lf}^{\beta,R}(x)\der_\beta
\end{align*}
for all $\alpha\in(\Delta_\lf)_+$. We have
\begin{align}\label{Ps=P eq}
P_{\alpha,\lf}^{\beta,R}(x)=P_\alpha^{\beta,R}(x)|_{x_\gamma=0\ \mathrm{for}\ \mathrm{all}\ \gamma\in\Delta_+\backslash(\Delta_\lf)_+}.
\end{align}

\begin{lemma}\label{Ps=P lemma}
Suppose that $\Pi\backslash\Pi_\lf\subset\Pi_1$. For $\alpha,\beta\in(\Delta_\lf)_+$, $P_{\alpha,\lf}^{\beta,R}(x)=P_\alpha^{\beta,R}(x)$ if $\degG\alpha=\degG\beta$.
\end{lemma}
\begin{proof}
Since $\Pi\backslash\Pi_\lf\subset\Pi_1$, we have $\Delta_+\backslash(\Delta_\lf)_+\subset\Delta_{\geq1}$. The assertion follows by \eqref{Ps=P eq} and Lemma \ref{pol1 lemma}.
\end{proof}

Let $\Phi(\lf_{\frac{1}{2}})$ be the neutral vertex algebra associated with $\lf_{\frac{1}{2}}$, which is defined by
\begin{align*}
\Phi^\lf_\alpha(z)\Phi^\lf_\beta(w)\sim\frac{(f_\lf|[e_\alpha,e_\beta])}{z-w}
\end{align*}
for generating fields $\Phi_\alpha^\lf(z)$, $\Phi_\beta^\lf(z)$ with $\alpha,\beta\in(\Delta_\lf)_{\frac{1}{2}}$.
\begin{lemma}\label{ind 1/2 lemma}
Suppose that $\Pi\backslash\Pi_\lf\subset\Pi_1$. Then $\Phi(\lf_{\frac{1}{2}})=\Fne$.
\end{lemma}
\begin{proof}
The assertion of the lemma immediately follows from Corollary \ref{ind chi cor}.
\end{proof}

Let $(\Pi_\lf)_j=\Pi_\lf\cap\Delta_j$. Recall that $[\alpha]$ is the subset of $\Delta_{>0}$ defined by \eqref{[a] def eq} for $\alpha\in\Pi_{>0}$.

\begin{lemma}\label{[a]s=[a] lemma}
Suppose that $\Pi\backslash\Pi_\lf\subset\Pi_1$. Then all roots in $[\alpha]$ lie in $\Delta_+^\lf$ for all $\alpha\in(\Pi_\lf)_{>0}$.
\end{lemma}
\begin{proof}
All roots in $[\alpha]$ are spanned by $\alpha$ and simple roots in $\Pi_0=(\Pi_\lf)_0$. Hence, the assertion of the lemma follows.
\end{proof}

\subsection{Parabolic inductions}\label{induced sec}
Set the Killing forms $\killing_\g$, $\killing_\lf$ on $\g$, $\lf$ respectively. 

\begin{theorem}\label{induced thm}
Let $\Gamma$ be a good grading for $f$ on $\g$ and $\lf$ a Levi subalgebra of $\g$ with simple roots $\Pi_\lf$. Suppose that $\Pi\backslash\Pi_\lf\subset\Pi_1$. Let $\Gamma_\lf$ be a $\frac{1}{2}\Z$-grading on $\lf$ defined by restriction of $\Gamma$ and $f_\lf$ the nilpotent element of $[\lf,\lf]$ chosen by Lemma \ref{ind good lem}. Then there exists an injective vertex algebra homomorphism
\begin{align*}
\Ind^\g_\lf\colon\W^\kappa(\g,f;\Gamma)\rightarrow\W^{\kappa_\lf}(\lf,f_\lf;\Gamma_\lf),
\end{align*}
where
\begin{align*}
\kappa_\lf=\kappa+\frac{1}{2}\killing_\g-\frac{1}{2}\killing_\lf.
\end{align*}
Moreover, the map $\Ind^\g_\lf$ is a unique vertex algebra homomorphism that satisfies
\begin{align*}
\mu_\kappa=\mu_{\lf, \kappa_\lf} \circ \Ind^\g_\lf,
\end{align*}
where $\mu_\kappa$, $\mu_{\lf, \kappa_\lf}$ are the Miura maps for $\W^\kappa(\g,f;\Gamma)$, $\W^{\kappa_\lf}(\lf,f_\lf;\Gamma_\lf)$ respectively.
\end{theorem}
\begin{proof}
First, we consider the case that $\g$ is a simple Lie algebra. By Corollary \ref{main cor},
\begin{align*}
\W^U(\g,f;\Gamma)\simeq\bigcap_{\alpha\in\Pi}\Ker\left(Q_\alpha \colon \Wak^U_{\g,\chi} \rightarrow \Wak^U_\chi(-\alpha)\right).
\end{align*}
Set
\begin{align}\label{s dec eq}
\lf=\z_\lf\oplus\bigoplus_{i=1}^{m_\lf}\lf^i,
\end{align}
where $\z_\lf$ is the center of $\lf$ and $\lf^i$ is a simple Lie algebra. Note that $\h$ is also a Cartan subalgebra of $\lf$. Let $\theta_i$ be the highest root of $\lf^i$, $h^\vee_i$ the dual Coxeter number of $\lf^i$ and $\bm{\kappa}_\lf=\bm{\kappa}+\frac{1}{2}\killing_\g-\frac{1}{2}\killing_\lf$ a $U$-valued invariant bilinear form on $\g$, where $\bm{\kappa}(u|v)=\para(u|v)$. Then
\begin{align*}
\bm{\kappa}_\lf(u|v)=
\begin{cases}
(\para+h^\vee)(u|v) & (u,v\in\z_\lf)\\
\para_i(u|v) & (u,v\in\lf^i),
\end{cases}
\end{align*}
where $\para_i=\frac{2}{(\theta_i|\theta_i)}(\para+h^\vee)-h^\vee_i$. We shall denote by $\W^U(\lf,f_\lf;\Gamma_\lf)$ instead of $\W^{\bm{\kappa}_\lf}(\lf,f_\lf;\Gamma_\lf)$. Set $(\Pi_\lf)^i=\{\alpha\in\Pi_\lf\mid\lf_\alpha\subset\lf^i\}$ and $(\Pi_\lf)_j^i=(\Pi_\lf)^i\cap(\Pi_\lf)_j$. As in Section \ref{reductive W-alg sec}, we have
\begin{align*}
\W^U(\lf,f_\lf;\Gamma_\lf)\simeq\bigcap_{\alpha\in\Pi_\lf}\Ker\left(Q_\alpha^\lf \colon \Wak^U_{\lf, \chi^\lf} \rightarrow \Wak^U_{\chi^\lf}(-\alpha)\right),
\end{align*}
where $Q_\alpha^\lf$ with $\alpha \in \Pi_\lf$ are screening operators of $\W^U(\lf,f_\lf;\Gamma_\lf)$ given in Section \ref{reductive W-alg sec} for $\para_i=\frac{2}{(\theta_i|\theta_i)}(\para+h^\vee)-h^\vee_i$ and $\chi^\lf(u) = (f_\lf|u)$. By Theorem \ref{main thm},
\begin{align*}
Q_\alpha^\lf =
\begin{cases}
\displaystyle \sum_{\beta\in(\Delta_\lf)_0^+}\int:P_{\alpha,\lf}^{\beta,R}(a^*)(z)a_{\beta}(z)\ \e^{-\frac{1}{\para_i+h_i^{\vee}}\int b^\lf_{\alpha}(z)}:dz & (\alpha\in(\Pi_\lf)_0^i),\\
\displaystyle \sum_{\beta\in(\Delta_\lf)_{\frac{1}{2}}\cap[\alpha]}\int:P_{\alpha,\lf}^{\beta,R}(a^*)(z)\Phi_\beta^\lf(z)\ \e^{-\frac{1}{\para_i+h_i^{\vee}}\int b^\lf_{\alpha}(z)}:dz & (\alpha\in(\Pi_\lf)^i_{\frac{1}{2}}),\\
\displaystyle \sum_{\beta\in(\Delta_\lf)_1\cap[\alpha]}(f_\lf|e_{\beta})\int:P_{\alpha,\lf}^{\beta,R}(a^*)(z)\ \e^{-\frac{1}{\para_i+h_i^{\vee}}\int b^\lf_{\alpha}(z)}:dz & (\alpha\in(\Pi_\lf)^i_1),
\end{cases}
\end{align*}
where $b^\lf_\alpha(z)=\frac{2}{(\theta_i|\theta_i)}b_\alpha(z)$ for all $\alpha\in(\Pi_\lf)^i$. Then, by definition of $\para_i$,
\begin{align*}
:\e^{-\frac{1}{\para_i+h_i^{\vee}}\int b^\lf_{\alpha}(z)}:\ =\ :\e^{-\frac{1}{\para+h^{\vee}}\int b_{\alpha}(z)}:,\quad
\alpha\in(\Pi_\lf)^i.
\end{align*}
Thus, we have
\begin{align*}
Q_\alpha^\lf =
\begin{cases}
\displaystyle \sum_{\beta\in\Deltazero}\int:P_{\alpha}^{\beta,R}(a^*)(z)a_{\beta}(z)\ \e^{-\frac{1}{\para+h^{\vee}}\int b_{\alpha}(z)}:dz & (\alpha\in(\Pi_\lf)_0),\\
\displaystyle \sum_{\beta\in[\alpha]}\int:P_{\alpha}^{\beta,R}(a^*)(z)\Phi_\beta(z)\ \e^{-\frac{1}{\para+h^{\vee}}\int b_{\alpha}(z)}:dz & (\alpha\in(\Pi_\lf)_{\frac{1}{2}}),\\
\displaystyle \sum_{\beta\in[\alpha]}\chi(e_{\beta})\int:P_{\alpha}^{\beta,R}(a^*)(z)\ \e^{-\frac{1}{\para+h^{\vee}}\int b_{\alpha}(z)}:dz & (\alpha\in(\Pi_\lf)_1),
\end{cases}
\end{align*}
thanks to Lemma \ref{Ps=P lemma}, Lemma \ref{ind 1/2 lemma} and Lemma \ref{[a]s=[a] lemma}. Therefore $Q_\alpha=Q_\alpha^\lf$ for all $\alpha\in\Pi_\lf$ by Theorem \ref{main thm}. Moreover, it is clear that $\Wak^U_{\g, \chi} = \Wak^U_{\lf, \chi^\lf}$. Hence the specialization of embeddings
\begin{align*}
\bigcap_{\alpha\in\Pi}\Ker Q_\alpha|_{\Wak^U_{\g, \chi}} \hookrightarrow
\bigcap_{\alpha\in\Pi_\lf}\Ker Q_\alpha|_{\Wak^U_{\g, \chi}} \hookrightarrow
\bigcap_{\alpha\in\Pi_0}\Ker Q_\alpha|_{\Wak^U_{\g, \chi}}
\end{align*}
induces vertex algebra homomorphisms
\begin{align}\label{para prf eq}
\W^k(\g,f;\Gamma)\rightarrow\W^{{\kappa}_\lf}(\lf,f_\lf;\Gamma_\lf)\rightarrow\Vtau\otimes\Fne.
\end{align}
Let us denote by
\begin{align*}
\Ind^\g_\lf\colon\W^k(\g,f;\Gamma)\rightarrow\W^{\kappa_\lf}(\lf,f_\lf;\Gamma_\lf)
\end{align*}
the first map and by $\mu_{\lf, \kappa_\lf}$ the second map in \eqref{para prf eq}. Then
\begin{align*}
\mu_k = \mu_{\lf, \kappa_\lf} \circ \Ind^\g_\lf,
\end{align*}
where $\mu_k$, $\mu_{\lf, \kappa_\lf}$ are the Miura maps for $\W^k(\g,f;\Gamma)$, $\W^{\kappa_\lf}(\lf,f_\lf;\Gamma_\lf)$ respectively by Corollary \ref{Miura cor}. Since $\mu_k$ is injective, so is $\Ind^\g_\lf$. Therefore the assertion of the theorem follows for any simple Lie algebra $\g$.

Next, consider arbitrary reductive Lie algebra $\g$. We use the decomposition \eqref{reductive dec eq}. Let $f^i_\lf$ be the image of $f_\lf$ by the projection $\lf\twoheadrightarrow\lf\cap\g^i$ and $\Gamma^i_\lf$ the good grading for $f^i_\lf$ inherited from $\Gamma_\lf$ provided that $\lf\cap\g^i\neq0$. Following the argument in the above, since $\g^i$ is a simple Lie algebra, we have an injective homomorphism
\begin{align*}
\Ind^{\g^i}_{\lf\cap\g^i}\colon\W^{k_i}(\g^i,f^i;\Gamma^i)\rightarrow\W^{\kappa_i}(\lf\cap\g^i,f^i_\lf;\Gamma^i_\lf)
\end{align*}
for all $i\in I_\lf:=\{j\in\{1,\ldots,m\}\mid\lf\cap\g^j\neq0\}$, where $k_i=\kappa(\theta^i|\theta^i)/2$ for the highest root $\theta^i$ in $\g^i$ and
\begin{align*}
\kappa_i(u|v)=k_i(u|v)+\frac{1}{2}\killing_{\g^i}-\frac{1}{2}\killing_{\lf\cap\g^i}(u|v)
\end{align*}
for all $u,v\in\lf\cap\g^i$. Since
\begin{align*}
\W^{\kappa_\lf}(\lf,f_\lf;\Gamma_\lf)=V^\kappa(\z_\g)\otimes\bigotimes_{i\in I_\lf}\W^{\kappa_i}(\lf\cap\g^i,f^i_\lf;\Gamma_\lf^i),
\end{align*}
we have an injective homomorphism
\begin{align*}
\Ind^\g_\lf=\Id_{V^\kappa(\z_\g)}\otimes\bigotimes_{i\in I_\lf}\Ind^{\g^i}_{\lf\cap\g^i}\colon\W^\kappa(\g,f;\Gamma)\rightarrow\W^{\kappa_\lf}(\lf,f_\lf;\Gamma_\lf),
\end{align*}
where $\kappa_\lf=\kappa+\frac{1}{2}\killing_\g-\frac{1}{2}\killing_\lf$. Then $\Ind^\g_\lf$ satisfies $\mu_\kappa=\mu_{\lf, \kappa_\lf} \circ \Ind^\g_\lf$ by the property of $\Ind^{\g^i}_{\lf\cap\g^i}$. The proof of the theorem is now complete except for the uniqueness of $\Ind^\g_\rf$. Let $\psi\colon\W^k(\g,f;\Gamma)\rightarrow\W^{{\kappa}_\lf}(\lf,f_\lf;\Gamma_\lf)$ be a vertex algebra homomorphism such that $\mu_\kappa = \mu_{\lf,\kappa_\lf} \circ \psi$. Since $\mu_{\lf, \kappa_\lf} \circ \Ind^\g_\lf = \mu_\kappa = \mu_{\lf, \kappa_\lf} \circ \psi$ and $\mu_{\lf, \kappa_\lf}$ is injective, we have $\Ind^\g_\lf = \psi$. This completes the proof.
\end{proof}

\begin{prop}\label{induced prop}
Let $\g$ be a reductive Lie algebra, $f$ a nilpotent element in $\g$, $\Gamma$ a good grading for $f$ on $\g$. Let $\lf$, $\lf'$ be Levi subalgebras of $\g$ such that $\lf\subset\lf'$ and $\Pi\backslash\Pi_\lf, \Pi\backslash\Pi_{\lf'}\subset\Pi_1$. Then
\begin{align*}
\Ind^\g_\lf=\Ind^{\lf'}_\lf\circ\Ind^\g_{\lf'}.
\end{align*}
\end{prop}
\begin{proof}
Since $\Pi_{\lf'}\backslash\Pi_\lf\subset\Pi\backslash\Pi_\lf\subset\Pi_1$, a map $\Ind^{\lf'}_\lf$ exists by Theorem \ref{induced thm}. By the characterization of $\Ind^\g_\rf$, $\Ind^{\lf'}_\lf$ and $\Ind^\g_{\lf'}$ given in Theorem \ref{induced thm}, $\mu_{\lf, \kappa_\lf} \circ \Ind^\g_\lf = \mu_\kappa = \mu_{\lf', \kappa_{\lf'}} \circ \Ind^\g_{\lf'}= (\mu_{\lf, \kappa_\lf} \circ \Ind^{\lf'}_\lf ) \circ \Ind^\g_{\lf'}$. Since $\mu_{\lf, \kappa_\lf}$ is injective, $\Ind^\g_\lf = \Ind^{\lf'}_\lf \circ \Ind^\g_{\lf'}$. Therefore the assertion follows.
\end{proof}

If $f$ is a principal nilpotent element in $[\g,\g]$, there exists a (unique) good grading $\Gamma$. Then $\Pi=\Pi_1=\Delta_1$. Let $\lf$ be any Levi subalgebra of $\g$. Then $\Pi_\lf=(\Pi_\lf)_1$ and the principal nilpotent element $f_\lf$ in $[\lf,\lf]$ is chosen by Lemma \ref{ind good lem}. By Theorem \ref{induced thm}, we have an injective homomorphism, which is constructed in \cite{BFN}.

\subsection{Chiralizations}\label{chiral sec}
Let $V$ be any $\frac{1}{2}\Z_{\geq0}$-graded vertex algebra. Denote by
\begin{align*}
A\circ B&=\sum_{j=0}^\infty\binom{\conf(A)}{j}A_{(j-2)}B,\\
A*B&=\sum_{j=0}^\infty\binom{\conf(A)}{j}A_{(j-1)}B
\end{align*}
for $A,B\in V$. Then a vector space $\Zhu(V)=V/(V\circ V)$ has a structure of an associative algebra by the multiplication induced by $*$, called the (twisted) Zhu algebra of $V$ \cite{Z,FZ,DK}. We call $V$ a {\em chiralization} of an associative algebra $\Zhu(V)$. Recall that $\W^\kappa(\g,f;\Gamma)=H^0(C_+,d)$, see Section \ref{W-alg sec}. By \cite{A,DK}, we have
\begin{align}\label{Zhu BRST eq}
\Zhu(H^0(C_+,d))=H^0(\Zhu(C_+),\bar{d}),
\end{align}
where $\bar{d}$ is the differential induced by $d$ such that a complex $(\Zhu(C_+^\bullet),\bar{d})$ defines the finite $\W$-algebra associated with $\g,f,\Gamma$, which we denote by $U(\g.f;\Gamma)$, i.e.
\begin{align*}
\Zhu(\W^\kappa(\g,f;\Gamma))=U(\g.f;\Gamma).
\end{align*}
See e.g. \cite{Lo2,Wan} for the definitions and properties of finite $\W$-algebras. Let $\Zhu(\Fne)=\fne$. Note that $\Zhu(V^\kappa(\g))=U(\g)$ and $\Zhu(C^0_+)=U(\g_{\leq0})\otimes\fne$. The projection $\g_{\leq0}\oplus\fne\twoheadrightarrow\g_0\oplus\fne$ induces an algebra homomorphism
\begin{align*}
\bar{\mu}\colon U(\g,f;\Gamma)\rightarrow U(\g_0\oplus\fne),
\end{align*}
called the {\em Miura map for $U(\g.f;\Gamma)$}. The following result is proved by \cite{Ly} in the case that $\Gamma$ is $\Z$-graded but may be also applied in general case. We give the sketch of the proof, following the proof of Proposition 4 in Section 2.6 in \cite{A} (with slight generalization).

\begin{lemma}[\cite{A}]\label{classical inj lem}
$\bar{\mu}$ is injective.
\end{lemma}
\begin{proof}
Recall that $\mathcal{S}_f$ is the Slodowy slice through $f$ and is isomorphic to the Marsden-Weinstein quotient of a transversal slice $f+\g_{\geq-\frac{1}{2}}$ in $\g\simeq\g^*$ by $G_{\geq\frac{1}{2}}$ (\cite{GG}). There exists a filtration on $U(\g,f;\Gamma)$, called the Kazhdan filtration, such that the induced map
\begin{align*}
\gr\bar{\mu}\colon\gr U(\g,f;\Gamma)\rightarrow\gr U(\g_0)\otimes\gr\fne.
\end{align*}
can be identified with the restriction map
\begin{align*}
\nu\colon\C[\mathcal{S}_f]=\C[f+\g_{\geq-\frac{1}{2}}]^{G_{\geq\frac{1}{2}}}\rightarrow\C[f+\g_0\oplus\g_{-\frac{1}{2}}].
\end{align*}
We will show that the restriction map $\nu$ is injective. If $P\in\C[\mathcal{S}_f]$ lies in the kernel of $\nu$, $P(g\cdot u)=0$ for all $g\in G_{\geq\frac{1}{2}}$ and $u\in f+\g_{\geq-\frac{1}{2}}$. Hence, it suffices to show that a map
\begin{align*}
\xi\colon G_{\geq\frac{1}{2}}\times(f+\g_0\oplus\g_{-\frac{1}{2}})\rightarrow f+\g_{\geq-\frac{1}{2}}
\end{align*}
defined by $\xi(g,u)=g\cdot u$ is dominant (i.e. the image of $\xi$ is Zariski dense). Let
\begin{align*}
d\xi\colon\g_{\geq\frac{1}{2}}\times\g_0\oplus\g_{-\frac{1}{2}}\rightarrow\g_{\geq-\frac{1}{2}}
\end{align*}
be the differential of $\xi$. Then the differential of $\xi$ at $(1,u)\in G_{\geq\frac{1}{2}}\times(f+\g_0\oplus\g_{-\frac{1}{2}})$ is given by
\begin{align*}
d\xi_{(1,u)}(a,b)=[a,u]+b
\end{align*}
and is an isomorphism if $u\in f+(\g_0\oplus\g_{-\frac{1}{2}})_\mathrm{reg}$, where
\begin{align*}
(\g_0\oplus\g_{-\frac{1}{2}})_\mathrm{reg}=\{v\in\g_0\oplus\g_{-\frac{1}{2}}\mid\g_{\geq\frac{1}{2}}^v=0\},\quad
\g_{\geq\frac{1}{2}}^v=\{w\in\g_{\geq\frac{1}{2}}\mid[v,w]=0\}.
\end{align*}
Hence, $\xi$ is dominant. See e.g. \cite{TY}. This completes the proof.
\end{proof}

Let $V,W$ be any $\frac{1}{2}\Z_{\geq0}$-graded vertex algebras and $\psi\colon V\rightarrow W$ any vertex algebra homomorphism. Since $\psi(V\circ V)=\psi(V)\circ \psi(V)\subset W\circ W$, the map $\psi$ induces an algebra homomorphism
\begin{align*}
\Zhu(\psi)\colon\Zhu(V)\rightarrow\Zhu(W).
\end{align*}
We shall say that $\psi$ is a chiralization of $\Zhu(\psi)$. For $\psi=\mu_\kappa$, we obtain a map
\begin{align*}
\Zhu(\mu_\kappa) \colon U(\g,f;\Gamma) \rightarrow U(\g_0)\otimes\fne.
\end{align*}
\begin{lemma}\label{classical Miura lem}
$\Zhu(\mu_\kappa) = \bar{\mu}$. In particular, $\Zhu(\mu_\kappa)$ is injective.
\end{lemma}
\begin{proof}
The formula $\Zhu(\mu_\kappa)=\bar{\mu}$ is induced by \eqref{Zhu BRST eq}. The injectivity of $\Zhu(\mu_\kappa)$ follows from Lemma \ref{classical inj lem}.
\end{proof}

For any map $\Ind^\g_\lf$ given in Theorem \ref{induced thm}, it induces an algebra homomorphism
\begin{align*}
\Zhu(\Ind^\g_\lf)\colon U(\g,f;\Gamma)\rightarrow U(\lf,f_\lf;\Gamma_\lf)
\end{align*}
such that
\begin{align}\label{mu ind eq}
\Zhu(\mu_\kappa)=\Zhu(\mu_{\lf, \kappa_\lf})\circ\Zhu(\Ind^\g_\lf)
\end{align}
by the characterization of $\Ind^\g_\lf$.
\begin{lemma}\label{Zhu ind lem}
$\Zhu(\Ind^\g_\lf)$ is a unique injective algebra homomorphism that satisfies $\bar{\mu}=\bar{\mu}_\lf\circ\Zhu(\Ind^\g_\lf)$.
\end{lemma}
\begin{proof}
The assertion of the lemma immediately follows from \eqref{mu ind eq} and Lemma \ref{classical Miura lem}.
\end{proof}

Losev constructs an injective algebra homomorphism in \cite{L}
\begin{align}\label{Losev eq}
U(\g,f;\Gamma)\rightarrow\widetilde{U}(\lf,f_\lf;\Gamma_\lf)
\end{align}
if $G\cdot f=\ind^\g_\lf L\cdot{f_\lf}$, where $\widetilde{U}(\lf,f_\lf;\Gamma_\lf)$ is a certain completion of $U(\lf,f_\lf;\Gamma_\lf)$.

\begin{conj}\label{Conj}
Under the condition that $\Pi\backslash\Pi_\lf\subset\Pi_1$, $\Zhu(\Ind^\g_\lf)$ coincides with the map \eqref{Losev eq}.
\end{conj}

The pull-back of Losev's map \eqref{Losev eq} gives a functor from $U(\lf,f_\lf;\Gamma_\lf)$-mod to $U(\g,f;\Gamma)$-mod, called a parabolic induction functor and first introduced by Premet in \cite{P}. Motivated by these results and Conjecture \ref{Conj}, we call a map $\Ind^\g_\lf$ given in Theorem \ref{induced thm} a {\em parabolic induction} for $\W$-algebras, which induces a functor from $\W^{\kappa_\lf}(\lf,f_\lf;\Gamma_\lf)$-mod to $\W^{\kappa}(\g,f;\Gamma)$-mod.

\section{Coproducts}\label{coproduct sec}

In this section, we consider parabolic inductions $\Ind^\g_\lf$ in the case that $\g=\gl_N$ and the case that $\g$ is of type $BCD$ and $f$ is a rectangular nilpotent element. In the former case, we show that $\Ind^\g_\lf$ is a chiralization of a coproduct $\bar{\cop}$ of finite $\W$-algebras of type $A$ constructed by \cite{BK2}, which we call coproducts for $\W$-algebras of type $A$. In the latter case, we give a structure of coproducts on $W$-algebras of type $BCD$ with rectangular nilpotent elements, giving rise to coproducts of twisted Yangians of level $l$.

\subsection{Pyramids}\label{pyramid sec}
To describe good gradings of $\gl_N$ combinatorially, we introduce {\em pyramids}, which are first introduced by \cite{EK} to classify all good gradings of simple (classical) Lie algebras. Following \cite{BK2}, we only consider pyramids corresponding to good $\Z$-gradings of $\gl_N$, which should be called even pyramids but shall call just pyramids.

Let $q=(q_1,\ldots,q_l)$ be a sequence of positive integers and $\Py$ a diagram defined by stacking $q_1$ boxes in the first column, $q_2$ boxes in the second column, $\cdots$, $q_l$ boxes in the right-most column. The diagram $\Py$ is called a pyramid if each row of $\Py$ consists of a single connected strip, i.e. $0\leq {}^\exists t\leq l$ such that $q_1\leq\cdots\leq q_t$ and $q_{t+1}\geq\cdots\geq q_l$. For example,
\vspace{3mm}
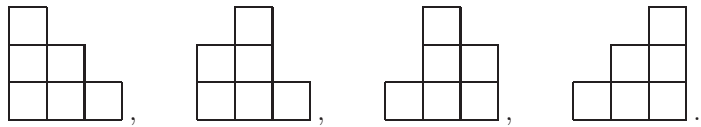
\begin{align*}
\setlength{\unitlength}{1mm}
\begin{picture}(0, 0)(20,10)
\put(-30,0){\line(0,1){15}}
\put(-30,0){\line(1,0){15}}
\put(-30,5){\line(1,0){15}}
\put(-30,10){\line(1,0){10}}
\put(-30,15){\line(1,0){5}}
\put(-25,0){\line(0,1){15}}
\put(-20,0){\line(0,1){10}}
\put(-15,0){\line(0,1){5}}
\put(-14,0){,}
\put(-5,0){\line(0,1){10}}
\put(-5,0){\line(1,0){15}}
\put(-5,5){\line(1,0){15}}
\put(-5,10){\line(1,0){10}}
\put(0,0){\line(0,1){15}}
\put(0,15){\line(1,0){5}}
\put(5,0){\line(0,1){15}}
\put(10,0){\line(0,1){5}}
\put(11,0){,}
\put(20,0){\line(0,1){5}}
\put(20,0){\line(1,0){15}}
\put(20,5){\line(1,0){15}}
\put(25,0){\line(0,1){15}}
\put(25,10){\line(1,0){10}}
\put(25,15){\line(1,0){5}}
\put(30,0){\line(0,1){15}}
\put(35,0){\line(0,1){10}}
\put(36,0){,}
\put(45,0){\line(0,1){5}}
\put(45,0){\line(1,0){15}}
\put(45,5){\line(1,0){15}}
\put(50,0){\line(0,1){10}}
\put(50,10){\line(1,0){10}}
\put(55,0){\line(0,1){15}}
\put(55,15){\line(1,0){5}}
\put(60,0){\line(0,1){15}}
\put(61,0){.}
\end{picture}\\
\quad\\
\end{align*}
Fix a pyramid $\Py$. Set the height $n=\max(q_1,\ldots,q_l)$ of $\Py$, the sequence $p=(p_1,\ldots,p_n)$ of length of rows of $\Py$ from top to bottom, and the number $N=\sum_{i=1}^l q_i=\sum_{i=1}^n p_i$ of boxes in $\Py$ (e.g. $p=(1,2,3)$ and $N=6$ in the above examples). We fix a numbering of boxes in $\Py$ by $1,\ldots,N$ from top to bottom and from left to right, and denote by $\row(i)$ the row number of the box in which $i$ appears and by $\col(i)$ the column number similarly. For example,\vspace{5mm}
\begin{align}\label{cop diag}
\ 
\setlength{\unitlength}{1mm}
\begin{picture}(0, 0)(20,10)
\put(-26,6){$\Py\ =$}
\put(-11,16.5){\footnotesize$\row(i)$}
\put(-8,11){1}
\put(-8,6){2}
\put(-8,1){3}
\put(0,0){\line(0,1){5}}
\put(0,0){\line(1,0){20}}
\put(0,5){\line(1,0){20}}
\put(2,1){1}
\put(5,0){\line(0,1){15}}
\put(5,10){\line(1,0){10}}
\put(5,15){\line(1,0){5}}
\put(7,11){2}
\put(7,6){3}
\put(7,1){4}
\put(10,0){\line(0,1){15}}
\put(12,6){5}
\put(12,1){6}
\put(15,0){\line(0,1){10}}
\put(17,1){7}
\put(20,0){\line(0,1){5}}
\put(21,0){,}
\put(2,-6){1}
\put(7,-6){2}
\put(12,-6){3}
\put(17,-6){4}
\put(23,-5.5){\footnotesize$\col(i)$}
\put(40,11){$p=(\ 1,2,4\ )$,}
\put(40,6){$q=(\ 1,3,2,1\ )$,}
\put(40,0){$N=7$.}
\end{picture}
\end{align}

\vspace{15mm}\hspace{-5mm}We have $\row(4)=3$ and $\col(4)=2$ etc. Let $\{v_i\}_{i=1}^{N}$ be the standard basis of $\C^N$ and $\{e_{i,j}\}_{i,j=1}^N$ the standard basis of $\gl_N=\End(\C^N)$ by $e_{i,j}(v_k)=\delta_{j,k}v_i$. We attach to $\Py$ a nilpotent element $f_\Py$ by
\begin{align*}
f_\Py(v_j)=
\begin{cases}
\ v_i & (\ \row(i)=\row(j)\quad \mathrm{and}\quad \col(i)=\col(j)+1\ ),\\
\ 0 & (\ \mathrm{otherwise}\ )
\end{cases}
\end{align*}
and a $\Z$-grading $\Gamma_\Py$ on $\gl_n$ by $\deg_{\Gamma_\Py}(e_{i,j})=\col(j)-\col(i)$. Then $f_\Py$ has the standard Jordan form consisting of $p_1$ Jordan block, $p_2$ Jordan block, $\cdots$, $p_n$ Jordan block with all the diagonal $0$, and $\Gamma_\Py$ is good for $f_\Py$. Set a Cartan subalgebra $\h=\bigoplus_{i=1}^N\C e_{i,i}$, the dual basis $\{\epsilon_i\}_{i=1}^N$ of $\h^*$ by $\epsilon_i(e_{j,j})=\delta_{i,j}$, and the root system
\begin{align*}
\Delta=\{\epsilon_i-\epsilon_j\in\h^*\mid1\leq i\neq j\leq N\},\quad
\Pi=\{\alpha_i:=\epsilon_i-\epsilon_{i+1}\mid i=1,\ldots,N-1\}
\end{align*}
as usual. Since $\deg_{\Gamma_\Py}(\epsilon_i-\epsilon_j)=\deg_{\Gamma_\Py}(e_{i,j})$, we have
\begin{align}\label{pi deg eq}
\Pi_{0}=\{\alpha_i\in\Pi\mid\row(i)<\row(N)\},\quad
\Pi_{1}=\{\alpha_i\in\Pi\mid\row(i)=\row(N)\}.
\end{align}
In the case of \eqref{cop diag},
\begin{align*}
\quad\\
f_\Py=e_{5,3}+e_{7,6}+e_{6,4}+e_{4,1},\hspace{60mm}
\setlength{\unitlength}{1mm}
\begin{picture}(0,0)(20,10)
\put(-29,10){\circle{2}}
\put(-30,13){\footnotesize$1$}
\put(-30,5){\footnotesize$\alpha_1$}
\put(-28,10.3){\line(1,0){8}}
\put(-19,10){\circle{2}}
\put(-20,13){\footnotesize$0$}
\put(-20,5){\footnotesize$\alpha_2$}
\put(-18,10.3){\line(1,0){8}}
\put(-9,10){\circle{2}}
\put(-10,13){\footnotesize$0$}
\put(-10,5){\footnotesize$\alpha_3$}
\put(-8,10.3){\line(1,0){8}}
\put(1,10){\circle{2}}
\put(0,13){\footnotesize$1$}
\put(0,5){\footnotesize$\alpha_4$}
\put(2,10.3){\line(1,0){8}}
\put(11,10){\circle{2}}
\put(10,13){\footnotesize$0$}
\put(10,5){\footnotesize$\alpha_5$}
\put(12,10.3){\line(1,0){8}}
\put(21,10){\circle{2}}
\put(20,13){\footnotesize$1$}
\put(20,5){\footnotesize$\alpha_6$}
\put(23,9){.}
\end{picture}\\
\hspace{10mm}
\end{align*}
We split $\Py$ into two pyramids $\Py_1$, $\Py_2$ along a column, which we denote by $\Py=\Py_1\oplus\Py_2$. For example,
\begin{align*}
\quad\\
\setlength{\unitlength}{1mm}
\begin{picture}(0, 0)(20,10)
\put(-10,0){\line(0,1){5}}
\put(-10,0){\line(1,0){20}}
\put(-10,5){\line(1,0){20}}
\put(-8,1){1}
\put(-5,0){\line(0,1){15}}
\put(-5,10){\line(1,0){10}}
\put(-5,15){\line(1,0){5}}
\put(-3,11){2}
\put(-3,6){3}
\put(-3,1){4}
\put(0,0){\line(0,1){15}}
\put(2,6){5}
\put(2,1){6}
\put(5,0){\line(0,1){10}}
\put(7,1){7}
\put(10,0){\line(0,1){5}}
\put(13.5,6){$=$}
\put(20,0){\line(0,1){5}}
\put(20,0){\line(1,0){10}}
\put(20,5){\line(1,0){10}}
\put(22,1){1}
\put(25,0){\line(0,1){15}}
\put(25,10){\line(1,0){5}}
\put(25,15){\line(1,0){5}}
\put(27,11){2}
\put(27,6){3}
\put(27,1){4}
\put(30,0){\line(0,1){15}}
\put(33.5,6){$\oplus$}
\put(40,0){\line(0,1){10}}
\put(40,0){\line(1,0){10}}
\put(40,5){\line(1,0){10}}
\put(40,10){\line(1,0){5}}
\put(42,6){5}
\put(42,1){6}
\put(45,0){\line(0,1){10}}
\put(47,1){7}
\put(50,0){\line(0,1){5}}
\put(51,0){.}
\end{picture}
\quad\\ \quad\\ \quad\\
\end{align*}
For $i=1,2$, let $N_i$ be the number of boxes in $\Py_i$ and $\lf_i=\gl_{N_i}$ the Lie subalgebra of $\gl_N$ spanned by all $e_{j,j'}$, where $j,j'$ run over numbers labeling $\Py_i$. Then $\Gamma_{\Py_i}$ is a good grading on $\lf_i$ for $f_{\Py_i}$ and is the restriction of $\Gamma_{\Py}$ by construction. Denote by $\Oc_\Py$ a nilpotent orbit in $\gl_N$ through $f_\Py$, by $\Oc_{\Py_i}$ a nilpotent orbit in $\lf_i$ through $f_{\Py_i}$ and by $\lf=\lf_1\oplus\lf_2$ a maximal Levi subalgebra of $\gl_N$. A combinatorial description of induced nilpotent orbits in $\gl_N$ given in \cite{Kr, OW} is compatible with our cases:
\begin{align*}
\Oc_\Py=\ind^{\gl_N}_\lf(\Oc_{\Py_1}+\Oc_{\Py_2}).
\end{align*}

\subsection{Coproducts for type $A$}\label{cop thm sec}
Let $\Py$ be a pyramid consisting of $N$ boxes. Set
\begin{align*}
\W^k(\gl_N,\Py)=V^{k+N}(\z_{\gl_N})\otimes\W^k(\slf_N,f_\Py;\Gamma_\Py).
\end{align*}
Set the subset $\Pi^i$ of $\Pi$ consisting of simple roots in $\gl_{N_i}$. Then
\begin{align*}
\Pi^1=\{\alpha_1,\ldots,\alpha_{N_1-1}\},\quad
\Pi^2=\{\alpha_{N_1+1},\ldots,\alpha_{N-1}\}.
\end{align*}
Therefore $\Pi\backslash(\Pi^1\sqcup\Pi^2)=\{\alpha_{N_1}\}$ and $\deg_{\Gamma_\Py}\alpha_{N_1}=1$ by $\row(N_1)=\row(N)$.

\begin{theorem}\label{coproduct thm}
Let $\Py$ be a pyramid split into $\Py_1\oplus\Py_2$. Set the numbers $N$, $N_1$, $N_2$ of boxes in $\Py$, $\Py_1$, $\Py_2$ respectively ($N=N_1+N_2$). Then there exists an injective vertex algebra homomorphism
\begin{align*}
\cop = \cop^{\Py}_{\Py_1,\Py_2}\colon\W^k(\gl_N,\Py)\rightarrow\W^{k_1}(\gl_{N_1},\Py_1)\otimes\W^{k_2}(\gl_{N_2},\Py_2)
\end{align*}
for all $k\in\C$, where $k+N=k_1+N_1=k_2+N_2$, such that
\begin{enumerate}
\item $\cop$ is a unique vertex algebra homomorphism that satisfies $\mu_k=(\mu_{1, k_1}\otimes\mu_{2, k_2})\circ\cop$, where $\mu_k$, $\mu_{1, k_1}$, $\mu_{2, k_2}$ are the Miura maps for $\W^k(\gl_N,\Py)$, $\W^{k_1}(\gl_{N_1},\Py_1)$, $\W^{k_2}(\gl_{N_2},\Py_2)$ respectively.
\item $\cop$ is coassociative, i.e.
\begin{align*}
(\Id\otimes\cop^{\Py_2\oplus\Py_3}_{\Py_2,\Py_3})\circ\cop^\Py_{\Py_1,\Py_2\oplus\Py_3}=(\cop^{\Py_1\oplus\Py_2}_{\Py_1,\Py_2}\otimes\Id)\circ\cop^\Py_{\Py_1\oplus\Py_2, \Py_3}
\end{align*}
for $\Py=\Py_1\oplus\Py_2\oplus\Py_3$.
\end{enumerate}
\end{theorem}
\begin{proof}
We use the notations in Section \ref{pyramid sec}. Let $\Pi_\lf$ be the set of simple roots in $\lf=\lf_1\oplus\lf_2$. Since $\{\alpha_{N_1}\}=\Pi\backslash\Pi_\lf\subset\Pi_1$, we may apply Theorem \ref{induced thm} for a nilpotent element $f_\Py$ of $\gl_N$. Hence, we have an injective vertex algebra homomorphism
\begin{align*}
\cop=\cop^\Py_{\Py_1,\Py_2}=\Ind^{\gl_N}_{\lf}\colon\W^k(\gl_N,\Py)\rightarrow\W^{\kappa_\lf}(\lf,f_\lf;\Gamma_\lf)=\W^{k_1}(\gl_{N_1},\Py_1)\otimes\W^{k_2}(\gl_{N_2},\Py_2)
\end{align*}
for all $k\in\C$, where $k+N=k_1+N_1=k_2+N_2$, which satisfies the desired properties by the characterization of $\Ind^{\gl_N}_\lf$ and Proposition \ref{induced prop}. This completes the proof.
\end{proof}

We will call $\cop$ a {\em coproduct} for $\W$-algebras of type $A$.
\smallskip

Let $U(\gl_N,\Py)=U(\gl_N,f_\Py;\Gamma_\Py)$ be the finite $\W$-algebra associated with $\gl_N,f_\Py,\Gamma_\Py$ for a pyramid $\Py$ consisting of $N$ boxes. It is known that $U(\gl_N,\Py)$ is isomorphic to a truncation of a shifted Yangian by \cite{BK2}. Following \cite{BK2}, for any pyramid $\Py$ split into $\Py_1\oplus\Py_2$, we have an injective algebra homomorphism
\begin{align*}
\bar{\cop}=\bar{\cop}^\Py_{\Py_1, \Py_2}\colon U(\gl_N,\Py)\rightarrow U(\gl_{N_1},\Py_1)\otimes U(\gl_{N_2},\Py_2),
\end{align*}
called a coproduct for finite $\W$-algebras of type $A$, where $N_i$ is the number of boxes in $\Py_i$. 

\begin{prop}\label{chiral prop}
$\Zhu(\cop)=\bar{\cop}$. Therefore $\cop$ is a chiralization of $\bar{\cop}$.
\end{prop}
\begin{proof}
Let $\Py$ be a pyramid split into $\Py_1\oplus\Py_2$, $\cop$ the corresponding coproduct for $\W$-algebras, $\bar{\cop}$ the corresponding coproduct for finite $\W$-algebras, $N$ the number of boxes in $\Py$, $N_i$ the number of boxes in $\Py_i$, $l$ the column length of $\Py$ and $l_i$ the column length of $\Py_i$ ($l=l_1+l_2$). We split $\Py_i$ into individual columns, i.e.
\begin{align}
\label{chiral eq}\Py_i&=\Py_i^1\oplus\cdots\oplus\Py_i^{l_i},\\
\label{chiral eq2}\Py&=\Py_1^1\oplus\cdots\oplus\Py_1^{l_1}\oplus\Py_2^1\oplus\cdots\oplus\Py_2^{l_2}
\end{align}
such that $\Py_i^j$ has only one column for all $i=1,2$ and $j=1,\ldots,l_i$. By \cite{BK2}, the coproducts of finite $\W$-algebras corresponding to \eqref{chiral eq}, \eqref{chiral eq2} are the Miura maps $\bar{\mu}_i$, $\bar{\mu}$ for $U(\gl_{N_i},\Py_i)$, $U(\gl_N,\Py)$ respectively. By coassociativity of $\bar{\cop}$, it satisfies that
\begin{align*}
\bar{\mu}=(\bar{\mu}_1\otimes\bar{\mu}_2)\circ\bar{\cop},
\end{align*}
which implies that $\bar{\cop}=\Zhu(\cop)$ by Lemma \ref{Zhu ind lem}. This completes the proof. 
\end{proof}

\subsection{Coproducts for type $BCD$}\label{cop BCD sec}
Let $N$ be a positive integer and $\g_N=\so_N$ or $\spf_N$. If $\g_N=\spf_N$, we assume that $N$ is even. Recall that all nilpotent orbits in $\g_N$ are classified by orthogonal partitions of $N$ if $\g_N=\so_N$ and by symplectic partitions of $N$ if $\g_N=\spf_N$. See e.g. \cite{CM}. In case of $\so_{2M}$, we mean nilpotent orbits under the group $O_{2M}$ not $SO_{2M}$ here. Let $f$ be a rectangular nilpotent element in $\g_N$, corresponding to a partition $p=(l^n)$ of $N$. A rectangular pyramid with the height $n$ and the width $l$ represents a good grading for $f$ on $\g_N$ in the classification of good gradings of $\g_N$ in \cite{EK}, and we denote by $\Py^+$ if $\g_N=\so_N$ and by $\Py^-$ if $\g_N=\spf_N$. We fix a numbering  of boxes in $\Py^\epsilon$ ($\epsilon=\pm$) by $1,\ldots,M$ from top to bottom and from left to right, by $-1,\cdots,-M$ in central symmetry and by $0$ in the central box if the central box exists, where $M=\lfloor\frac{N}{2}\rfloor$. For example,\\

\begin{align*}
\setlength{\unitlength}{1mm}
\begin{picture}(0, 0)(20,10)
\put(-28,0){\line(0,1){15}}
\put(-28,0){\line(1,0){25}}
\put(-28,5){\line(1,0){25}}
\put(-28,10){\line(1,0){25}}
\put(-28,15){\line(1,0){25}}
\put(-26,11){1}
\put(-26,6){2}
\put(-26,1){3}
\put(-23,0){\line(0,1){15}}
\put(-21,11){4}
\put(-21,6){5}
\put(-21,1){6}
\put(-18,0){\line(0,1){15}}
\put(-16,11){7}
\put(-16,6){0}
\put(-17,1){-7}
\put(-13,0){\line(0,1){15}}
\put(-12,11){-6}
\put(-12,6){-5}
\put(-12,1){-4}
\put(-8,0){\line(0,1){15}}
\put(-7,11){-3}
\put(-7,6){-2}
\put(-7,1){-1}
\put(-3,0){\line(0,1){15}}
\put(-2,0){,}
\put(7,0){\line(0,1){10}}
\put(7,0){\line(1,0){25}}
\put(7,5){\line(1,0){25}}
\put(7,10){\line(1,0){25}}
\put(9,6){1}
\put(9,1){2}
\put(12,0){\line(0,1){10}}
\put(14,6){3}
\put(14,1){4}
\put(17,0){\line(0,1){10}}
\put(19,6){5}
\put(18,1){-5}
\put(22,0){\line(0,1){10}}
\put(23,6){-4}
\put(23,1){-3}
\put(27,0){\line(0,1){10}}
\put(28,6){-2}
\put(28,1){-1}
\put(32,0){\line(0,1){10}}
\put(33,0){,}
\put(42,0){\line(0,1){20}}
\put(42,0){\line(1,0){20}}
\put(42,5){\line(1,0){20}}
\put(42,10){\line(1,0){20}}
\put(42,15){\line(1,0){20}}
\put(42,20){\line(1,0){20}}
\put(44,16){1}
\put(44,11){2}
\put(44,6){3}
\put(44,1){4}
\put(47,0){\line(0,1){20}}
\put(49,16){5}
\put(49,11){6}
\put(49,6){7}
\put(49,1){8}
\put(52,0){\line(0,1){20}}
\put(53,16){-8}
\put(53,11){-7}
\put(53,6){-6}
\put(53,1){-5}
\put(57,0){\line(0,1){20}}
\put(58,16){-4}
\put(58,11){-3}
\put(58,6){-2}
\put(58,1){-1}
\put(62,0){\line(0,1){20}}
\put(63,0){.}
\end{picture}
\quad\\ \quad\\
\end{align*}
Denote a basis of $\C^N$ by $\{v_i,v_{-i}\}_{i=1}^{M}$ if $N=2M$, and by $\{v_i\}_{i=-M}^M$ if $N=2M+1$. Then $\End(\C^N)$ has a basis consisting of all $e_{i,j}$ by $e_{i,j}v_k=\delta_{j,k}v_i$. To describe $f$ from $\Py^\epsilon$ explicitly, we fix a basis of $\g_N$ in $\End(\C^N)$ as follows:
\begin{align*}
\so_{2M+1}:\ &e_{i,j}-e_{-j,-i},\ e_{s,-t}-e_{t,-s},\ e_{-s,t}-e_{-t,s},\ e_{i,0}-e_{0,-i},\ e_{0,-i}-e_{-i,0}\\
\so_{2M}:\ &e_{i,j}-e_{-j,-i},\ e_{s,-t}-e_{t,-s},\ e_{-s,t}-e_{-t,s}\\
\spf_{2M}:\ &e_{i,j}-e_{-j,-i},\ e_{s,-t}+e_{t,-s},\ e_{-s,t}+e_{-t,s},\ e_{i,-i},\ e_{-i,i}
\end{align*}
with $1\leq i,j\leq M$ and $1\leq s<t\leq M$. We attach to $\Py^\epsilon$ a nilpotent element $f_{\Py^\epsilon}$ by
\begin{align}\label{BCD nil eq}
f_{\Py^\epsilon}(v_j)=
\begin{cases}
\ \pm v_i & (\ \row(i)=\row(j)\quad \mathrm{and}\quad \col(i)=\col(j)+1\ ),\\
\ 0 & (\ \mathrm{otherwise}\ ),
\end{cases}
\end{align}
where the sign $\pm$ is chosen such that $f_{\Py^\epsilon}\in\g_N$. We split $\Py^\epsilon$ into three rectangular pyramids along two columns in line symmetry, which we denote by $\Py^\epsilon=\Py_1^\epsilon\oplus\Py_2^\epsilon\oplus\Py_1^\epsilon$, such that $\Py^\epsilon_2$ represents a symmetric partition of $N_2$ if $\g_{N_2}=\so_{N_2}$, or an orthogonal partition of $N_2$ if $\g_{N_2}=\spf_{N_2}$. For example,\\
\begin{align*}
\setlength{\unitlength}{1mm}
\begin{picture}(0, 0)(20,10)
\put(-25,0){\line(0,1){15}}
\put(-25,0){\line(1,0){35}}
\put(-25,5){\line(1,0){35}}
\put(-25,10){\line(1,0){35}}
\put(-25,15){\line(1,0){35}}
\put(-23,11){1}
\put(-23,6){2}
\put(-23,1){3}
\put(-20,0){\line(0,1){15}}
\put(-18,11){4}
\put(-18,6){5}
\put(-18,1){6}
\put(-15,0){\line(0,1){15}}
\put(-13,11){7}
\put(-13,6){8}
\put(-13,1){9}
\put(-10,0){\line(0,1){15}}
\put(-9,11){10}
\put(-8,6){0}
\put(-9.8,1){-10}
\put(-5,0){\line(0,1){15}}
\put(-4,11){-9}
\put(-4,6){-8}
\put(-4,1){-7}
\put(0,0){\line(0,1){15}}
\put(1,11){-6}
\put(1,6){-5}
\put(1,1){-4}
\put(5,0){\line(0,1){15}}
\put(6,11){-3}
\put(6,6){-2}
\put(6,1){-1}
\put(10,0){\line(0,1){15}}
\put(13.5,6){$=$}
\put(20,0){\line(0,1){15}}
\put(20,0){\line(1,0){10}}
\put(20,5){\line(1,0){10}}
\put(20,10){\line(1,0){10}}
\put(20,15){\line(1,0){10}}
\put(22,11){1}
\put(22,6){2}
\put(22,1){3}
\put(25,0){\line(0,1){15}}
\put(27,11){4}
\put(27,6){5}
\put(27,1){6}
\put(30,0){\line(0,1){15}}
\put(33.5,6){$\oplus$}
\put(40,0){\line(0,1){15}}
\put(40,0){\line(1,0){15}}
\put(40,5){\line(1,0){15}}
\put(40,10){\line(1,0){15}}
\put(40,15){\line(1,0){15}}
\put(42,11){7}
\put(42,6){8}
\put(42,1){9}
\put(45,0){\line(0,1){15}}
\put(46,11){10}
\put(47,6){0}
\put(45.2,1){-10}
\put(50,0){\line(0,1){15}}
\put(51,11){-9}
\put(51,6){-8}
\put(51,1){-7}
\put(55,0){\line(0,1){15}}
\put(58.5,6){$\oplus$}
\put(65,0){\line(0,1){15}}
\put(65,0){\line(1,0){10}}
\put(65,5){\line(1,0){10}}
\put(65,10){\line(1,0){10}}
\put(65,15){\line(1,0){10}}
\put(66,11){-6}
\put(66,6){-5}
\put(66,1){-4}
\put(70,0){\line(0,1){15}}
\put(71,11){-3}
\put(71,6){-2}
\put(71,1){-1}
\put(75,0){\line(0,1){15}}
\put(76,0){.}
\end{picture}
\quad\\ \quad\\
\end{align*}
Let $N_a$ be the number of boxes in $\Py_a^\epsilon$ ($N=2N_1+N_2$), and $\h=\bigoplus_{i=1}^M\C h_i$ a Cartan subalgebra of $\g_N$, where $h_i=e_{i,i}-e_{-i,-i}$. Set the dual basis $\{\epsilon_i\}_{i=1}^M$ of $\h^*$ by $\epsilon_i(h_j)=\delta_{i,j}$, and a set $\Pi=\{\alpha_i\}_{i=1}^M$ of simple roots by $\alpha_i=\epsilon_i-\epsilon_{i+1}$ for $i=1,\ldots,M-1$ and 
\begin{align*}
\alpha_M=
\begin{cases}
\epsilon_M & (\g_N=\so_{2M+1}),\\
2\epsilon_M & (\g_N=\spf_{2M}),\\
\epsilon_{M-1}+\epsilon_{M} & (\g_N=\so_{2M}).
\end{cases}
\end{align*}
Let $\lf=\lf_1\oplus\lf_2$ be a maximal Levi subalgebra of $\g_N$ such that $\{\alpha_i\}_{i=1}^{N_1-1}$ is a set of simple roots in $\lf_1=\gl_{N_1}$, and $\{\alpha_i\}_{i=N_1+1}^{M}$ is a set of simple roots in $\lf_2=\g_{N_2}$. We attach to $\Py_a^\epsilon$ a nilpotent element $f_{\Py_a^\epsilon}$ in $\lf_a$ by the same formula in \eqref{BCD nil eq}, where $i,j$ run over the set of numbers labeling $\Py_a^\epsilon$. By \cite{Ke}, we have
\begin{align*}
\Oc_{\Py^\epsilon}=\ind^{\g_N}_{\lf}(\Oc_{\Py_1^\epsilon}+\Oc_{\Py_2^\epsilon}),
\end{align*}
where $\Oc_{\Py^\epsilon}$ denotes a nilpotent orbit in $\g_N$ through $f_{\Py^\epsilon}$, and $\Oc_{\Py_a^\epsilon}$ denotes a nilpotent orbit in $\lf_a$ through $f_{\Py_a^\epsilon}$. Define a good $\Z$-grading $\Gamma_{\Py^\epsilon}$ on $\g_N$ for $f_{\Py^\epsilon}$ by $\deg_{\Gamma_{\Py^\epsilon}}(e_{i,j})=\col(j)-\col(i)$ and a good $\Z$-grading $\Gamma_{\Py_a^\epsilon}$ on $\lf_a$ for $f_{\Py_a^\epsilon}$ similarly. Then $\Gamma_{\Py_a^\epsilon}$ is the restriction of $\Gamma_{\Py^\epsilon}$ on $\lf_a$, and satisfies that $\{\alpha_{N_1}\}=\Pi\backslash\Pi_\lf\subset\Pi_1$. Let
\begin{align*}
\W^k(\g_N,\Py^\epsilon)&=\W^k(\g_N,f_{\Py^\epsilon};\Gamma_{\Py^\epsilon}),\\
\W^k(\gl_{N_1},\Py_1^\epsilon)&=V^{k+N_1}(\z_{\gl_{N_1}})\otimes\W^k(\slf_{N_1},f_{\Py_1^\epsilon};\Gamma_{\Py_1^\epsilon})
\end{align*}
and
\begin{align*}
\gamma(\epsilon)=
\begin{cases}
1 & (\epsilon=+), \\
2 & (\epsilon=-).
\end{cases}
\end{align*}
Recall that the dual Coxter number of $\g_N$ is $N-2$, $M+1$ if $\g_N=\so_{N}, \spf_{2M}$ respectively. The same proof as we use to prove Theorem \ref{coproduct thm} is applicable to the following.

\begin{theorem}\label{cop BCD thm}
In the above settings, there exists an injective vertex algebra homomorphism
\begin{align*}
\cop^\epsilon\colon\W^k(\g_N,\Py^\epsilon)\rightarrow\W^{k_1}(\gl_{N_1},\Py_1^\epsilon)\otimes\W^{k_2}(\g_{N_2},\Py_2^\epsilon),
\end{align*}
where $k+h^\vee=\gamma(\epsilon)(k_1+N_1)=k_2+h^\vee_2$ and $h^\vee,h^\vee_2$ are the dual Coxter numbers of $\g_N$, $\g_{N_2}$ respectively. Moreover, $\cop^\epsilon$ is a unique vertex algebra homomorphism that satisfies that $\mu_k=(\mu_{1, k_1} \otimes \mu_{2, k_2})\circ\cop^\epsilon$, where $\mu_k$, $\mu_{1, k_1}$, $\mu_{2, k_2}$ are the Miura maps for $\W^k(\g_N,\Py^\epsilon)$, $W^{k_1}(\gl_{N_1},\Py_1^\epsilon)$, $\W^{k_2}(\g_{N_2},\Py_2^\epsilon)$ respectively.
\end{theorem}

Suppose that the height $n$ of $\Py^\epsilon$ is even if $\g_N=\spf_N$. Let $l_a$ be the width of $\Py_a^\epsilon$ ($2l_1+l_2=l=N/n$). According to \cite{BK2} and \cite{Br}, it follows that $U(\gl_{N_1}, f_{\Py_1^\epsilon};\Gamma_{\Py_1^\epsilon})$ is isomorphic to the Yangian $Y_{l_1}(\gl_n)$ of level $l_1$, and $U(\g_N,f_{\Py^\epsilon};\Gamma_{\Py^\epsilon})$ is isomorphic to the twisted Yangian $Y_l^\epsilon(\g_n)$ of level $l$.

\begin{cor}\label{cop BCD cor}
Suppose that the height $n$ of $\Py^\epsilon$ is even if $\g_N=\spf_N$. Then there exists an injective algebra homomorphism
\begin{align*}
\bar{\cop}^\epsilon\colon Y_l^\epsilon(\g_n)\rightarrow Y_{l_1}(\gl_{n})\otimes Y_{l_2}^\epsilon(\g_{n}).
\end{align*}
\end{cor}
\begin{proof}
The assertion of the corollary immediately follows from Theorem \ref{cop BCD thm} and Lemma \ref{Zhu ind lem}.
\end{proof}

\section{Examples}\label{examples sec}
We describe $\cop$ in Theorem \ref{coproduct thm} explicitly in some examples.

\subsection{Principal nilpotent}\label{principal sec}
Let $\Py_\prin$ be a pyramid that represents a principal nilpotent element in $\gl_N$, i.e. a pyramid consisting of one row of $N$ boxes. Set a basis $\{h_i=e_{i,i}\}_{i=1}^N$ of a Cartan subalgebra $\h$ of $\gl_N$ and the associated Heisenberg vertex algebra $\Hi=\Hi^{k+N}(\h)$, in which
\begin{align*}
h_i(z)h_j(w)\sim\frac{(k+N)\delta_{i,j}}{(z-w)^2}
\end{align*}
holds for all $i,j=1,\ldots,N$. Consider fields $W_i(z)$ on $\Hi$ defined by the following formal products
\begin{align*}
:(\hat{\der}+h_1(z))\cdot(\hat{\der}+h_2(z))\cdots(\hat{\der}+h_N(z)):\ =\sum_{i=0}^N W_i(z)\hat{\der}^{N-i},
\end{align*}
where $\hat{\der}$ is defined by $[\hat{\der},h_i(z)]=(k+N-1)\der_z h_i(z)$ for all $i$. Then, a vertex subalgebra of $\Hi$ generated by $W_i(z)$ for all $i=1,\ldots,N$ is isomorphic to the $\W$-algebra $\W^k_N=\W^k(\gl_N,\Py_\prin)$ by \cite{FL} (and \cite{FF4}), which coincides with the image of the Miura map for $\W^k_N$. Let $N_1$, $N_2$ be positive integers such that $N_1+N_2=N$ and $W_i^1(z)$, $W_j^2(z)$ fields on $\Hi$ defined by
\begin{align*}
\sum_{i=0}^{N_1} W_i^1(z)\hat{\der}^{N_1-i}&=\ :(\hat{\der}+h_1(z))\cdot(\hat{\der}+h_2(z))\cdots(\hat{\der}+h_{N_1}(z)):,\\
\sum_{i=0}^{N_2} W_i^2(z)\hat{\der}^{N_2-i}&=\ :(\hat{\der}+h_{N_1+1}(z))\cdot(\hat{\der}+h_{N_1+2}(z))\cdots(\hat{\der}+h_{N}(z)):.
\end{align*}
For $j=1,2$, a vertex subalgebra of $\Hi$ generated by $W_i^j(z)$ for all $i=1,\ldots,N_j$ is isomorphic to $\W^{k_j}_{N_j}$, where $k+N=k_1+N_1=k_2+N_2$. By construction,
\begin{align}\label{principal eq}
\sum_{i=0}^N W_i(z)\hat{\der}^{N-i}=\ :\left(\sum_{i=0}^{N_1} W_i^1(z)\hat{\der}^{N_1-i}\right)\left(\sum_{i=0}^{N_2} W_i^2(z)\hat{\der}^{N_2-i}\right):,
\end{align}
which induces an injective vertex algebra homomorphism
\begin{align*}
\cop\colon\W^k_N\rightarrow\W^{k_1}_{N_1}\otimes\W^{k_2}_{N_2}
\end{align*}
for all $k\in\C$. This map $\cop$ is a coproduct for $\W^k_N$ corresponding to a splitting of a pyramid:
\vspace{3mm}
\begin{align*}
\setlength{\unitlength}{1mm}
\begin{picture}(0, 0)(20,5)
\put(-20,3.5){\line(0,1){5}}
\put(-20,3.5){\line(1,0){32}}
\put(-20,8.5){\line(1,0){32}}
\put(-18.3,4.8){$1$}
\put(-15,3.5){\line(0,1){5}}
\put(-13.3,4.8){$2$}
\put(-10,3.5){\line(0,1){5}}
\put(-7.3,5){$\cdots$}
\put(0,3.5){\line(0,1){5}}
\put(1,5){\fontsize{4pt}{0pt}\selectfont $N$$-$$1$}
\put(7,3.5){\line(0,1){5}}
\put(8,4.8){\fontsize{7pt}{0pt}\selectfont $N$}
\put(12,3.5){\line(0,1){5}}
\put(15.5,5){$=$}
\put(22,3.5){\line(0,1){5}}
\put(22,3.5){\line(1,0){20}}
\put(22,8.5){\line(1,0){20}}
\put(23.7,4.8){$1$}
\put(27,3.5){\line(0,1){5}}
\put(29.7,5){$\cdots$}
\put(37,3.5){\line(0,1){5}}
\put(38,4.8){\fontsize{7pt}{0pt}\selectfont $N_1$}
\put(42,3.5){\line(0,1){5}}
\put(45.5,5){$\oplus$}
\put(52,3.5){\line(0,1){5}}
\put(52,3.5){\line(1,0){22}}
\put(52,8.5){\line(1,0){22}}
\put(52.5,5){\fontsize{4pt}{0pt}\selectfont $N_1$$+$$1$}
\put(59,3.5){\line(0,1){5}}
\put(61.7,5){$\cdots$}
\put(69,3.5){\line(0,1){5}}
\put(70,4.8){\fontsize{7pt}{0pt}\selectfont $N$}
\put(74,3.5){\line(0,1){5}}
\put(75,3.5){.}
\end{picture}\\
\end{align*}

\subsection{Rectangular cases}\label{rectangular sec}
We generalize the above construction to the case that $f$ is a rectangular nilpotent element in $\gl_N$. We follow the framework in \cite{AM}. Let $\Py$ be a rectangular pyramid of the width $l$ and the height $n$ and $N=n l$. The target space of the Miura map for $\W^k(\gl_{N},\Py)$ is a tensor vertex algebra $V^\kappa(\gl_n)^{\otimes l}$, where $\kappa$ is defined by
\begin{align*}
\kappa(u|v)=
\begin{cases}
(k+n l)\mathrm{tr}(u v) & (u,v\in\z_{\gl_n}) \\
(k+n(l-1))\mathrm{tr}(u v) & (u,v\in\slf_n).
\end{cases}
\end{align*}
Denote by $u^{(t)}(z)$ a field $u(z)$ on the $t$-th component in $V^\kappa(\gl_n)^{\otimes l}$ for all $u\in\gl_n$ and $t=1,\ldots,l$. Set a fields-valued matrix
\begin{align*}
A_t(z)=\left(e_{j,i}^{(t)}(z)\right)_{i,j=1}^n
\end{align*}
for each $t=1,\ldots,l$. Let $W_{i,j,t}(z)$ be a field on $V^\kappa(\gl_n)^{\otimes l}$ defined by the formal product
\begin{align*}
:(\hat{\der}+A_1(z))\cdot(\hat{\der}+A_2(z))\cdots(\hat{\der}+A_l(z)):\ =\sum_{t=0}^l W_t(z)\hat{\der}^{(l-t)}
\end{align*}
and $W_t(z)=\left(W_{i,j,t}(z)\right)_{i,j=1}^n$, where a product (e.g. $A_i(z)\cdot A_j(z)$ etc) is computed by the usual product of matrices and $\hat{\der}$ is defined by
\begin{align*}
[\hat{\der},A_t(z)]=(k+n(l-1))\der_z A_t(z),\quad
\der_z A_t(z)=\left(\der_z e_{j,i}^{(t)}(z)\right)_{i,j=1}^n.
\end{align*}
Then a vertex subalgebra of $V^\kappa(\gl_n)^{\otimes l}$ generated by $W_{i,j,t}(z)$ for all $i,j=1,\ldots,n$ and $t=1,\ldots,l$ is isomorphic to a $\W$-algebra $\W^k(\gl_{N},\Py)$ by \cite{AM}, and coincides with the image of the Miura map for $\W^k(\gl_{N},\Py)$. For a splitting $\Py=\Py_1\oplus\Py_2$ of $\Py$, let $l_i$ be the width of $\Py_i$ and $N_i=n l_i$. Then $\Py_i$ is a $n\times l_i$ rectangular pyramid. Define fields $W_{i,j,t}^{1}(z)$, $W_{i,j,t}^{2}(z)$ on $V^\kappa(\gl_n)^{\otimes l}$ by
\begin{align*}
\sum_{i=0}^{l_1} W_t^1(z)\hat{\der}^{(l_1-t)}&=\ :(\hat{\der}+A_1(z))\cdot(\hat{\der}+A_2(z))\cdots(\hat{\der}+A_{l_1}(z)):,\\
\sum_{i=0}^{l_2} W_t^2(z)\hat{\der}^{(l_2-t)}&=\ :(\hat{\der}+A_{l_1+1}(z))\cdot(\hat{\der}+A_{l_1+2}(z))\cdots(\hat{\der}+A_{l}(z)):,
\end{align*}
where $W_t^1(z)=\left(W_{i,j,t}^{1}(z)\right)_{i,j=1}^n$ and $W_t^2(z)=\left(W_{i,j,t}^{2}(z)\right)_{i,j=1}^n$. By construction,
\begin{align}
\sum_{t=0}^l W_t(z)\hat{\der}^{(l-t)}=\ :\left(\sum_{i=0}^{l_1} W_t^1(z)\hat{\der}^{(l_1-t)}\right)\left(\sum_{i=0}^{l_2} W_t^2(z)\hat{\der}^{(l_2-t)}\right):,
\end{align}
which induces an injective vertex algebra homomorphism
\begin{align*}
\cop\colon\W^k(\gl_{N},\Py)\rightarrow\W^{k_1}(\gl_{N_1},\Py_1)\otimes\W^{k_2}(\gl_{N_2},\Py_2)
\end{align*}
for all $k\in\C$. This map $\cop$ is a coproduct corresponding to $\Py=\Py_1\oplus\Py_2$.

\subsection{Subregular nilpotent}\label{subregular sec}
Let $\Py$ be the pyramid with the sequence of column lengths $(2,1^{N-2})$. Then the nilpotent element $f_\Py=\sum_{i=2}^{N-1}e_{i+1,i}$ is a subregular nilpotent element in $\gl_N$. We have $(\gl_N)_0=\z_{(\gl_N)_0}\oplus\slf_2$ and $\z_{(\gl_N)_0}=\bigoplus_{i=3}^N\C h_i$, where $h_i=e_{i,i}$. The corresponding $\W$-algebra $\W^k(\gl_N,\Py)$ is then isomorphic to the tensor of the Feigin-Semikatov algebra $\W^{(2)}_{N}$ (\cite{FS}) and the Heisenberg vertex algebra $V^{k+N}(\z_{\gl_N})$ if $k+N\neq0$ (\cite{G}). From now on, we assume that $k+N\neq0$. Let $H(z),Z(z),E(z),F(z)$ be fields on $V^{\tau_k}((\gl_N)_0)=V^{k+N}(\z_{(\gl_N)_0})\otimes V^{k+2}(\slf_2)$ defined by
\begin{align*}
H(z)&=h_1(z)-\frac{1}{N}\sum_{i=1}^N h_i(z),\quad
Z(z)=\sum_{i=1}^N h_i(z),\quad
E(z)=e_{1,2}(z),\\
F(z)&=\ :(\hat{\der}+(h_1-h_N)(z))(\hat{\der}+(h_1-h_{N-1})(z))\cdots(\hat{\der}+(h_1-h_3)(z))e_{2,1}(z):,
\end{align*}
where $\hat{\der}=(k+N-1)\der_z$, which generate a vertex subalgebra of $V^{\tau_k}((\gl_N)_0)$ isomorphic to $\W^k(\gl_N,\Py)$ by \cite{FS}. We split $\Py$ as
\vspace{3mm}
\begin{align}\label{subregular eq}
\setlength{\unitlength}{1mm}
\begin{picture}(0, 0)(20,5)
\put(-20,0.5){\line(0,1){10}}
\put(-20,0.5){\line(1,0){32}}
\put(-20,5.5){\line(1,0){32}}
\put(-20,10.5){\line(1,0){5}}
\put(-18.3,6.8){$1$}
\put(-18.3,1.8){$2$}
\put(-15,0.5){\line(0,1){10}}
\put(-13.3,1.8){$3$}
\put(-10,0.5){\line(0,1){5}}
\put(-7.3,2){$\cdots$}
\put(0,0.5){\line(0,1){5}}
\put(1,2){\fontsize{4pt}{0pt}\selectfont $N$$-$$1$}
\put(7,0.5){\line(0,1){5}}
\put(8,1.8){\fontsize{7pt}{0pt}\selectfont $N$}
\put(12,0.5){\line(0,1){5}}
\put(15.5,2){$=$}
\put(22,0.5){\line(0,1){10}}
\put(22,0.5){\line(1,0){20}}
\put(22,5.5){\line(1,0){20}}
\put(22,10.5){\line(1,0){5}}
\put(23.7,6.8){$1$}
\put(23.7,1.8){$2$}
\put(27,0.5){\line(0,1){10}}
\put(29.7,2){$\cdots$}
\put(37,0.5){\line(0,1){5}}
\put(38,1.8){\fontsize{7pt}{0pt}\selectfont $N_1$}
\put(42,0.5){\line(0,1){5}}
\put(45.5,2){$\oplus$}
\put(52,0.5){\line(0,1){5}}
\put(52,0.5){\line(1,0){22}}
\put(52,5.5){\line(1,0){22}}
\put(52.5,2){\fontsize{4pt}{0pt}\selectfont $N_1$$+$$1$}
\put(59,0.5){\line(0,1){5}}
\put(61.7,2){$\cdots$}
\put(69,0.5){\line(0,1){5}}
\put(70,1.8){\fontsize{7pt}{0pt}\selectfont $N$}
\put(74,0.5){\line(0,1){5}}
\put(75,0.5){,}
\end{picture}\ 
\end{align}
\vspace{3mm}

\hspace{-5mm}which we denote by $\Py=\Py_1\oplus\Py_2$. Let $\Zc=V^{k+N}(\z_{(\gl_N)_0})$ and $\Zc_1$ (resp. $\Zc_2$) a vertex subalgebra of $\Zc$ generated by $h_i(z)$ with $i=3,\ldots,N_1$ (resp. $i=N_1+1,\ldots,N$). Set $N_2=N-N_1$. We have $\Zc=\Zc_1\otimes\Zc_2$. Let $H_1(z), Z_1(z), E_1(z), F_1(z)$ be fields on $\Zc_1\otimes V^{k+2}(\slf_2)$ defined by
\begin{align*}
H_1(z)&=h_1(z)-\frac{1}{N_1}\sum_{i=1}^{N_1} h_i(z),\quad
Z_1(z)=\sum_{i=1}^{N_1} h_i(z),\quad
E_1(z)=e_{1,2}(z),\\
F_{1}(z)&=\ :(\hat{\der}+(h_1-h_{N_1})(z))(\hat{\der}+(h_1-h_{N_1-1})(z))\cdots(\hat{\der}+(h_1-h_3)(z))e_{2,1}(z):,
\end{align*}
which generate a vertex subalgebra of $\Zc_1\otimes V^{k+2}(\slf_2)$ isomorphic to $\W^{k_1}(\gl_{N_1},\Py_1)$ by construction, where $k+N=k_1+N_1$. For $i=0,\ldots,N_2$, let $W_i(z)$ be fields on $\Zc_2$ defined by
\begin{align*}
:(\hat{\der}-h_N(z))\cdots(\hat{\der}-h_{N_1+1}(z)):\ =\sum_{i=0}^{N_2}W_i(z)\hat{\der}^{N_2-i}.
\end{align*}
Since an automorphism $\tau$ on $\Zc_2$ defined by $h_i(z)\mapsto -h_{N_1+1+N-i}(z)$ ($i=N_1+1,\ldots,N$) implies the formula
\begin{align*}
\sum_{i=0}^{N_2}\tau(W_i(z))
=&\ :(\hat{\der}+h_{N_1+1}(z))\cdots(\hat{\der}+h_{N}(z)):,
\end{align*}
which is a formal product defining generating fields of the $\W^{k_2}_{N_2}$ introduced in Section \ref{principal sec}, the fields $W_i(z)$ with $i=1,\ldots,N_2$ generate a vertex subalgebra of $\Zc_2$ isomorphic to $\W^{k_2}(\gl_{N_2},\Py_2)$, where $k+N=k_2+N_2$. We have
\begin{align}
\label{subreg eq1}H(z)=&H_1(z)+\frac{N_2}{N N_1}Z_1(z)-\frac{1}{N}W_1(z),\\
\label{subreg eq2}Z(z)=&Z_1(z)+W_1(z),\\
\label{subreg eq3}E(z)=&E_1(z),\\
\label{subreg eq4}F(z)
=&\sum_{i=0}^{N_2}\sum_{j=0}^{N_2-i}\binom{N_2-j}{i}:\left(W_j(z)\hat{\der}^{N_2-j-i}\right)P_i(z)F_1(z):,
\end{align}
where
\begin{align*}
P_0(z)=1,\quad
P_i(z)=:(\hat{\der}-h_1(z))^{i-1}h_1(z):
\end{align*}
for $i=1,\ldots,N_2$. Here, we use the following lemma:

\begin{lemma}
\begin{align*}
\sum_{j=0}^{N_2-i}\binom{N_2-j}{i}W_j(z)\hat{\der}^{N_2-j-i}
=\sum_{j_1<\cdots<j_i}:(\hat{\der}-h_N(z))\overset{j_1}{\check{\cdots}}\ \overset{\cdots}{\cdots}\ \overset{j_i}{\check{\cdots}}(\hat{\der}-h_{N_1+1}(z)):
\end{align*}
for all $i=0,\ldots,N_2$.
\end{lemma}
\begin{proof}
For $1\leq n\leq N_2$, $1\leq j\leq n$ and $1\leq t_1<\cdots<t_{n}\leq N_2$, we define fields $W_j^{n}(u_{t_1},\cdots,u_{t_n})$ on $\Zc_2$ by the following formula:
\begin{align*}
:(\hat{\der}-u_{t_1}(z))\cdots(\hat{\der}-u_{t_n}(z)):
\ =\sum_{j=0}^{n}W_j^{n}(u_{t_1},\cdots,u_{t_n})\hat{\der}^{n-j},
\end{align*}
where $u_i(z)=h_{N-i+1}(z)$. Set $W_0^0(\phi)=0$. The assertion of the lemma is equivalent to the formula
\begin{align}\label{subreg lem eq}
\binom{n-j}{i}W_j^n(u_1,\ldots,u_n)=\sum_{1\leq j_1<\cdots<j_i\leq n}W_j^{n-i}(u_1,\overset{j_1}{\check{\cdots}}\ \overset{\cdots}{\cdots}\ \overset{j_i}{\check{\cdots}},u_n)
\end{align}
for $n=N_2$, where $(i,j)$ run over $\{(i,j)\in\Z^2\mid0\leq i,j,i+j\leq n\}$. We will show the formula \eqref{subreg lem eq} by induction on $n$ and $i+j$. If $n=1$ or $i+j=n$, it is easy to check that the formula \eqref{subreg lem eq} follows. If $n>1$ and $i+j<n$, we have
\begin{align*}
&\binom{n-j}{i}W_j^n(u_1,\ldots,u_n)
=\binom{n-j}{i}\binom{n-j}{i+1}^{-1}\binom{n-j}{i+1}W_j^n(u_1,\ldots,u_n)\\
=&\frac{i+1}{n-i-j}\sum_{1\leq j_1<\cdots<j_{i+1}\leq n}W_j^{n-i-1}(u_1,\overset{j_1}{\check{\cdots}}\ \overset{\cdots}{\cdots}\ \overset{j_{i+1}}{\check{\cdots}},u_n)\\
=&\frac{1}{n-i-j}\sum_{1\leq j_1<\cdots<j_{i}\leq n}\ \sum_{t\neq j_1,\ldots,j_{i}}W_j^{n-i-1}(u_1,\overset{j_1}{\check{\cdots}}\ \overset{\cdots}{\cdots}\ \overset{t}{\check{\cdots}}\ \overset{\cdots}{\cdots}\ \overset{j_{i}}{\check{\cdots}},u_n)\\
=&\sum_{1\leq j_1<\cdots<j_i\leq n}W_j^{n-i}(u_1,\overset{j_1}{\check{\cdots}}\ \overset{\cdots}{\cdots}\ \overset{j_i}{\check{\cdots}},u_n)
\end{align*}
by using our inductive assumptions. This completes the proof.
\end{proof}

Since $h_1(z)=H_1(z)+\frac{1}{N_1}Z_1(z)$ is a field on $\W^{k_1}(\gl_{N_1},\Py_1)$, $:P_i(z)F_{1}(z):$ are fields on $\W^{k_1}(\gl_{N_1},\Py_1)$ for all $i=0,\cdots,N_2$. Hence, the formula \eqref{subreg eq1}--\eqref{subreg eq4} induces an injective vertex algebra homomorphism
\begin{align*}
\cop\colon\W^k(\gl_N,\Py)\rightarrow\W^{k_1}(\gl_{N_1},\Py_1)\otimes\W^{k_2}(\gl_{N_2},\Py_2),
\end{align*}
which is the coproduct corresponding to $\Py=\Py_1\oplus\Py_2$.
\vspace{20mm}

\appendix
\begin{center}
\vspace*{12pt}

\textbf{\bfseries\large APPENDIX. Miura maps and Parabolic Wakimoto resolutions}
\vspace{12pt} \\
\begin{tabular}{rl}
SHIGENORI NAKATSUKA
\end{tabular}\footnote{\texttt{nakatuka@ms.u-tokyo.ac.jp} Graduate School of Mathematical Sciences, The University of Tokyo, 3-8-1 Komaba, Tokyo, Japan 153-8914}
\vspace{12pt} \\
\end{center}

\section{}\label{appendix}
We introduce here the parabolic Wakimoto resolution $E^F_{I,\bullet}$ for a $\W$-algebra $\W^F(\g,f;\Gamma)$. The resolution is constructed via the Drinfeld-Sokolov reduction from a parabolic Wakimoto resolution of $V^F(\g)$. The latter resolution (of $V^F(\g)$) is obtained from the dual $C^\vee_{I,\bullet}$ of the parabolic BGG resolution of the trivial $\g$-module $\C$ with respect to the parabolic subalgebra $\g_{\geq0}$ via the Fiebig's equivalence. We prove that the Miura map $\mu_F$ for $\W^F(\g,f;\Gamma)$ and the screening operators $\widetilde{Q}_\alpha$ given in \cite{G} are reconstructed as differential in $E^F_{I,\bullet}$, see Proposition \ref{Miura screening from differential} and \ref{coincidence with Miura map}.
Similarly, we reconstruct the Wakimoto resolution $E^F_{\bullet}$ of $\W^F(\g,f;\Gamma)$ from the dual $C^\vee_{\bullet}$ of the BGG resolution of $\C$.
Since $C^\vee_{I,\bullet}$ is embedded into $C^\vee_{\bullet}$, so is $E^F_{I,\bullet}$ into $E^F_{\bullet}$, see \eqref{parabolic BGG complex} and \eqref{parabolic dual BGG complex for W-algebra}. As a consequence, we obtain Lemma \ref{Miura lemma}.
The resolutions $E^F_{I,\bullet}$ and $E^F_{\bullet}$ are the same for the principal $\W$-algebras and coincide with the one in \cite{ACL, FF6}.

\subsection{Parabolic BGG resolution}\label{Para BGG}
Our main reference here is  \cite[Chapter 9]{Ku}, but our notation here is different from those in loc.\ cit listed in \cite[Chapter 1]{Ku}.
Let $\g$ be a finite-dimensional simple Lie algebra. We follows the notation in Section \ref{W-alg sec}-\ref{W-alg Wak sec}: in particular, $\g=\nil_+\oplus \h\oplus \nil_- = \bigoplus_{j\in\frac{1}{2}\Z}\g_j$, $\bo_\pm = \h\oplus\nil_\pm$, $\Pi = \Pi_0 \sqcup \Pi_{\frac{1}{2}} \sqcup \Pi_1$ the set of simple roots of $\g$, $W$ the Weyl group of $\g$, $\ell(w)$ the length of $w\in W$, $w_\circ$ the longest element in $W$, $w \circ \lambda$ the dot action of $w\in W$ for $\lambda \in \h^*$.

Let $M(\lambda)$ (resp. $M_0(\lambda)$) be the Verma module of $\g$ (resp. $\g_0$) with the highest weight $\lambda \in \h^*$ and $L(\lambda)$ (resp. $L_0(\lambda)$) be the simple quotient of $M(\lambda)$ (resp. $M_0(\lambda)$). Let $\widetilde{I}$ be an index set of simple roots of $\g$, i.e. $\Pi = \{\alpha_i\}_{i\in \tilde{I}}$ and $I$ be the subset of $\widetilde{I}$ such that $\Pi_0 = \{\alpha_i\}_{i\in I}$. The {\it generalized Verma module }$M_I(\lambda)$ of $\g$ with the highest weight $\lambda\in\h^*$ for a parabolic subalgebra $\g_{\geq0}$ is defined by
\begin{align*}
M_I(\lambda) = U(\g)\otimes_{U(\g_{\geq0})}L_{0}(\lambda),
\end{align*}
where $L_{0}(\lambda)$ extends to a $\g_{\geq0}$-module by $\g_{>0} \rightarrow 0$. Using the universal property of induced modules, we have projections $M(\lambda) \twoheadrightarrow M_I(\lambda) \twoheadrightarrow L(\lambda)$ as $\g$-modules, which we call canonical projections. Let $W_0$ be the Weyl group of $\g_0$, $\overline{\Oc}$ be the BGG category of $\g$-modules and $(?)^\vee$ be the duality functor in $\overline{\Oc}$. Then we have $L(\lambda)^\vee \simeq L(\lambda)$ and the canonical inclusions $L(\lambda) \hookrightarrow M_I(\lambda)^\vee \hookrightarrow M(\lambda)^\vee$ induced from the canonical projections. By Theorem 9.2.18 in \cite{Ku}, we have the following commutative diagram:
\begin{equation}\label{parabolic BGG complex}
\SelectTips{lu}{}
\vcenter{\xymatrix@W=15pt@H=20pt@R=12pt@C=10pt{
0 \ar[r]
&\C\ar[r]\ar@{=}[d]
&C_0^\vee\ar[r]
&C_1^\vee\ar[r]
&\cdots\ar[r]
&C_{\ell(w_\circ')}^\vee\ar[r]
&\cdots\ar[r]
&C_{\ell(w_\circ)}^\vee\ar[r]
&0\\
0 \ar[r]
&\C\ar[r]
&C_{I,0}^\vee\ar[r]\ar@{^{(}-_>}[u]
&C_{I,1}^\vee\ar[r]\ar@{^{(}-_>}[u]
&\cdots\ar[r]
&C_{I,\ell(w_\circ')}^\vee\ar[r]\ar@{^{(}-_>}[u]
&0,
}}
\end{equation}
where
\begin{align*}
&C_i^\vee = \bigoplus_{\begin{subarray}{c}w\in W\\ \ell(w)=i\end{subarray}}M(w^{-1}\circ 0)^\vee,\quad
C_{I,i}^\vee=\bigoplus_{\begin{subarray}{c}w\in W_0'\\ \ell(w)=i\end{subarray}}M_I(w^{-1}\circ 0)^\vee,\\
&W_0'=\{w\in W\mid \forall u\in W_0,\ \ell(wu)\geq \ell(w)\}
\end{align*}
with the longest element $w_\circ'$ in $W'_0$. Here the first row is the dual of the BGG resolution of the trivial $\g$-module $\C$, the second row is the dual of the parabolic BGG resolution of $\C$ associated with a Levi subalgebra $\g_0$ and the vertical morphisms are sum of canonical inclusions $M_I(w^{-1}\circ0)^\vee \hookrightarrow M(w^{-1}\circ0)^\vee$ for $w \in W'_0$.

\subsection{Fiebig's Equivalence}\label{Feibig}
Let $\g^F=\g\otimes_\C F$ be a Lie algebra over $F$ and $\hat{\g}^F = \g^F \otimes \C[t,t^{-1}]\oplus F K$ be the affine Lie algebra of $\g^F$ defined by
\begin{align}\label{commutator}
[x t^m, y t^n]=[x,y] t^{m+n}+m\delta_{m, n}(x|y)K,\quad
x, y \in \g^F,
\end{align}
and $K$ is a central element. Let $\overline{\Oc}_{F}$ be the category of $\g_F$-modules whose object are $\g^F$-modules $\overline{M}$ equipped with a $\C$-form $\overline{M}_\C$ that is a $\g$-module in $\overline{\Oc}$ and satisfy $\overline{M} = \overline{M}_\C \otimes F$, and whose morphisms are all the $\g^F$-homomorphisms preserving the $\C$-forms. If the $\C$-form is obvious, we omit mentioning it. It is clear that $\overline{\Oc}$ is equivalent to $\overline{\Oc}_F$ by $\overline{N} \mapsto \overline{N} \otimes F$. Recall that the affine vertex algebra $V^F(\g)$ over $F$ is defined as the induced $\hat{\g}^F$-module from the trivial $\g^F$-module $F$:
\begin{align*}
V^F(\g) = U(\hat{\g}^F)\underset{U(\g^F\otimes \C[t] \oplus F K)}{\otimes}F,
\end{align*}
where $F$ extends a $\g^F \otimes \C[t] \oplus F K$-module by $\g^F[t]t \mapsto 0$ and $K \mapsto \para$. Since $\para+h^\vee$ is invertible in $F$, $V^F(\g)$ is conformal by the Sugawara construction. The field corresponding to the conformal vector is given by
\begin{align*}
L(z) = \sum_{n\in\Z}L_n z^{-n-2}
=\frac{1}{2(\para+h^\vee)}\sum_{i=1}^{\dim\g}:x_i(z)x^i(z):,
\end{align*}
where $\{x_i\}_{i=1}^{\mathrm{dim}\g}$ is a basis of $\g$ and $\{x^i\}_{i=1}^{\mathrm{dim}\g}$ is its dual basis with respect to the normalized invariant form $\inv$. Let $\mathcal{O}_F$ be the category of $V^F(\g)$-modules whose objects are $V^F(\g)$-modules $M$ equipped with a $\C$-form $M_\C$ that is a $\g$-submodule over $\C$, satisfying the following properties: (1) $\h$ and $L_0$ acts on $M$ semisimply, (2) $\dim_F U(\hat{\nil}_+)m<\infty$ for all $m\in M$ where $\hat{\nil}_+=\nil_+\oplus \g^F[t]t$. The morphisms are $V^F(\g)$-homomorphisms preserving the $\C$-forms. For $M \in \mathcal{O}_F$, let $M = \oplus_{\Delta\in F}M_\Delta$ be the $L_0$-grading and $M^{\mathrm{top}}$ be the $\g_F$-submodule of $M$ defined by the direct sum of $M_r$ such that $M_{r+n}=0$ for all $n\in \Z_{<0}$. Then by Fiebig's equivalence \cite{Fie} it follows that the functors
\begin{align*}
\begin{array}{ll}
\ind_{\g}^{\hat{\g}^F}(?):& \overline{\Oc} \rightarrow \mathcal{O}_F,\quad
\overline{N} \mapsto U(\hat{\g}^F) \underset{U(\g^F\otimes \C[t] \oplus F K)}{\otimes} (\overline{N} \otimes F);\\
(?)^{\mathrm{top}}_\C:&\mathcal{O}_F \rightarrow \overline{\Oc},\quad
M \mapsto M^{\mathrm{top}} \cap M_\C,
\end{array}
\end{align*}
give rise to an equivalence of categories and are inverses to each other, where $\overline{N} \otimes F$ extends to a $\g^F\otimes \C[t] \oplus F K$-module by $\g^F[t]t \mapsto 0$ and $K \mapsto \para$ in the definition of $\ind_{\g}^{\hat{\g}^F}(?)$. For $M \in \mathcal{O}_F$, let $M = \oplus_{\lambda \in \h^*}M^\lambda$ be the $\h$-weight decomposition, $M^\lambda_\Delta = M_\Delta \cap M^\lambda$ and $M^\vee = \bigoplus_{\lambda, \Delta}\Hom_F\left( M^\lambda_\Delta, F \right)$ be the contragredient dual of $M$. Since $((\ind_{\g}^{\hat{\g}^F}(\overline{N}))^\vee)^\mathrm{top}_\C = \overline{N}^\vee$, we have $(\ind_{\g}^{\hat{\g}^F}(\overline{N}))^\vee = \ind_{\g}^{\hat{\g}^F}(\overline{N}^\vee)$ for $\overline{N} \in \overline{\Oc}$. Let $\mathbb{M}^F(\lambda)$ be the Verma module of $\hat{\g}^F$ with the highest weight $\lambda \in \h^*$ at level $\para$, and $\mathbb{M}^F_I(\lambda) = \ind_{\g}^{\hat{\g}^F}(M_I(\lambda))$ be the induced $\hat{\g}^F$-module from $M_I(\lambda)$. Then both of $\mathbb{M}^F(\lambda)$ and $\mathbb{M}^F_I(\lambda)$ are objects in $\Oc_F$. By \cite{FF2}, it follows that
\begin{align*}
\mathbb{M}^F(\lambda)^\vee \simeq \Wak^F(\lambda),\quad
\lambda \in \h^*.
\end{align*}
See also Proposition 3.3 in \cite{ACL}.

\subsection{Generalized Wakimoto representations of $V^F(\g)$}\label{G Wakimoto}
Suppose that the coordinate $c(\nil_+)$ on $N_+$ is compatible with $N_+ = G_{>0} \times G_0^+$. Recall that we introduce Wakimoto free fields realizations $\hat{\rho}_{F}$, $\hat{\rho}_{\g_{0}, F}$ of $V^F(\g)$, $V^F(\g_0)$ in Section \ref{sec:Wakimoto subsec}, \ref{Explicit forms of screenings} respectively, where $V^F(\g_0)$ is the affine vertex algebra over $F$ associated with $\g_0$ and its invariant form $\tau_\para$.
\begin{lemma}[\cite{F}]\label{lemma:para Wakimoto}
There exists an injective vertex algebra homomorphism $\hat{\rho}_{\g_0,F}^{\g} \colon V^F(\g) \rightarrow \Am^F\otimes V^F(\g_0)$ over $F$ such that $\hat{\rho}_{F} = (\Id \otimes \hat{\rho}_{\g_{0}, F}) \circ \hat{\rho}_{\g_0,F}^{\g}$ and
\begin{align}\label{parabolic Wakimoto}
\hat{\rho}_{\g_0,F}^{\g}(u)(z)=u(z)-\sum_{\alpha,\beta\in\Delta_{>0}}c_{u,\beta}^\alpha :a_\beta^*(z)a_\alpha(z):,\quad u\in\g_0.
\end{align}
\end{lemma}
Then any $\Am^F\otimes V^F(\g_0)$-module is a $V^F(\g)$-module through $\hat{\rho}_{\g_0,F}^{\g}$, called a {\it generalized Wakimoto representation of $V^F(\g)$}. Let $P_{0,+} = \{ \lambda \in \h^* \mid (\lambda|\alpha_i^\vee)\in\Z_{\geq0}\ \mathrm{for}\ \mathrm{all}\ i \in I\}$ be the set of dominant integral weights of $\g_0$. Then $L_0(\lambda)$ is a finite-dimensional simple $\g_0$-module. For $\lambda \in P_{0,+}$, let $\mathbb{V}_0^F(\lambda)$ be the Weyl module of $\hat{\g}_0^F$ with highest weight $\lambda$ defined by
\begin{align*}
\Weyl^F_0(\lambda) = U(\hat{\g}^F_0)\underset{U(\g_0^F \otimes \C[t]\oplus F K)}{\otimes}(L_{0}(\lambda)\otimes F).
\end{align*}
Then $\Weyl^F_0(\lambda)$ is a $V^F(\g_0)$-module and thus $\Am^T\otimes\Weyl^F_0(\lambda)$ is a $V^F(\g)$-module in $\mathcal{\Oc}_F$. Denote by
\begin{align*}
\Wak^F_I(\lambda) = \Am^T\otimes\Weyl^F_0(\lambda),\quad
\lambda \in P_{0,+};\quad
\Wak^F_{I, \g} = \Wak^F_I(0).
\end{align*}
\begin{lemma}
$\mathbb{M}^F_I(\lambda)^\vee \simeq \Wak^F_I(\lambda)$ for $\lambda \in P_{+,0}$.
\end{lemma}
\proof
The assertion follows from the Fiebig's equivalence and
\begin{align*}
\left(\Wak^F_I(\lambda)\right)^{\mathrm{top}}_\C \simeq M_I(\lambda)^\vee.
\end{align*}
\endproof
Let $\mathbb{M}_0^F(\lambda)$ be the Verma module of $\hat{\g}_0^F$ with the highest weight $\lambda\in \h^*$ and $\overline{\Oc}_0$ be the BGG category of $\g_0$-modules. As an analogue of $\Oc_F$, the category $\Oc_{0, F}$ of $V^F(\g_0)$-modules satisfying the same properties as in $\Oc_F$ can be defined. Then $\overline{\Oc}_0$ is equivalent to $\Oc_{0, F}$ through the induction functor by Fiebig's equivalence and the inclusion $L_0(\lambda) \hookrightarrow M_0(\lambda)^\vee$ implies that $\Weyl^F_0(\lambda) \hookrightarrow \mathbb{M}_0^F(\lambda)^\vee$ for $\lambda \in P_{0,+}$, where $(?)^\vee$ denotes the duality functor in $\overline{\Oc}_0$ and the contragredient dual in $\Oc_{0, F}$. Using an injective homomorphism $\hat{\rho}_{\g_0, F} \colon V^F(\g_0) \rightarrow \mathcal{A}_{\Deltazero}^F\otimes \mathcal{H}^F$ of vertex algebras over $F$, we have a Wakimoto representation $\mathcal{A}_{\Delta_0^+}^F\otimes \mathcal{H}^F_\lambda$ of $V^F(\g_0)$ with the highest weight $\lambda\in\h^*$, which we denote by $\Wak^F_0(\lambda)$. Then we also have \cite{FF2, ACL} isomorphisms
\begin{align*}
\mathbb{M}^F_0(\lambda) \simeq \Wak^F_0(\lambda),\quad
\lambda \in \h^*.
\end{align*}
Hence there exist embeddings
\begin{align*}
\hat{\rho}_{\g_0,\lambda,F}\colon \mathbb{V}_0^F\left(\lambda\right)\hookrightarrow \mathbb{M}_0^F(\lambda)^\vee \xrightarrow{\sim} \Wak^F_0(\lambda),\quad
\lambda \in P_{0,+}.
\end{align*}

\subsection{Parabolic Wakimoto resolution}\label{Para Wakimoto}
Now, we apply the induction functor $\mathrm{Ind}_{\g}^{\hat{\g}_F}$ to \eqref{parabolic BGG complex}. Then we obtain the following commutative diagram:
\begin{equation}\label{affine parabolic dual BGG complex}
\SelectTips{lu}{}
\vcenter{\xymatrix@W=15pt@H=20pt@R=12pt@C=10pt{
0 \ar[r]
&V^F(\g)\ar[r]^-{\epsilon}\ar@{=}[d]
&D_0^F\ar[r]^-{d_0}
&D_1^F\ar[r]
&\cdots\ar[r]
&D^F_{\ell(w_\circ')}\ar[r]
&\cdots\ar[r]
&D^F_{\ell(w_\circ)}\ar[r]
&0\\
0 \ar[r]
&V^F(\g)\ar[r]^-{\epsilon_I}
&D^F_{I,0}\ar[r]^-{d_{I,0}}\ar@{^{(}-_>}[u]_-{\iota_0}
&D^F_{I,1}\ar[r]\ar@{^{(}-_>}[u]_-{\iota_1}
&\cdots\ar[r]
&D^F_{I,\ell(w_0')}\ar[r]\ar@{^{(}-_>}[u]
&0,
}}
\end{equation}
where
\begin{align*}
D^F_i=\bigoplus_{\begin{subarray}{c}w\in W\\ \ell(w)=i\end{subarray}}W^F(w^{-1}\circ0),\quad 
D^F_{I,i}=\bigoplus_{\begin{subarray}{c}w\in W_0'\\ \ell(w)=i\end{subarray}}\Wak^F_I\left(w^{-1}\circ0\right).
\end{align*}
Here the first and second rows are exact and
\begin{align*}
D_0^F = \Wak^F_{\g},\quad
D_{I, 0}^F = \Wak^F_{I, \g}.
\end{align*}
By the Fiebig's equivalence,
$$\dim\Hom_{\mathcal{O}_F}(V^F(\g), D_0^F)=\dim\Hom_{\overline{\Oc}}(\C, C_0^\vee)=1.$$
Focusing on the top spaces, we see that $\epsilon = \hat{\rho}_F$ and $\epsilon_I = \hat{\rho}^\g_{\g_0, F}$. Similarly, we have $d_0 = \bigoplus_{\alpha\in \Pi}S_\alpha$, see \eqref{scaffine eq} for the definition of $S_\alpha$. Since the higher differentials $d_i \colon D^F_i \rightarrow D^F_{i+1}$ with $i \geq 1$ are unique up to scalar by the Fiebig's equivalence and Lemma 9.2.16 in \cite{Ku}, we conclude that the exact sequence in the first row in \eqref{affine parabolic dual BGG complex} coincides with \eqref{affine Wakimoto resolution} up to scalar for the choices of higher differentials. Next, consider the column arrows $\iota_i \colon D^F_{I, i} \rightarrow D^F_i$ with $i\geq0$. Again by the Fiebig's equivalence,
\begin{align*}
\dim\left(\Wak^F_I(\lambda),W^F(\lambda)\right)
=\dim\left(M_I(\lambda)^\vee,M(\lambda)^\vee\right)=1,\quad
\lambda \in P_{0,+}.
\end{align*}
Thus, we have $\iota_0=\Id\otimes \hat{\rho}_{\g_0,F}$, $\iota_1=\bigoplus_{\alpha\in \Pi}\Id\otimes \hat{\rho}_{\g_0,-\alpha,F}$ and
\begin{align*}
\iota_i=\bigoplus_{\begin{subarray}{c}w\in W_0'\\ \ell(w)=i\end{subarray}}c_w(\Id\otimes \hat{\rho}_{\g_0,w^{-1}\circ0,F}),\quad
i \geq 2
\end{align*}
for some invertibles $c_w\in F$.
Finally, consider the second row. We already know $\epsilon_I = \hat{\rho}_{\g_0,F}^{\g}$. Since $d_{I, 0}$ is the restriction of $\oplus_{\alpha\in \Pi_{>0}}S_\alpha$ to $\Wak^F_{I, \g}$ through $\iota_0$, the image of $\iota_0\left(\Wak^F_{I,\g}\right)$ by $S_\alpha$ for $\alpha \in \Pi_{>0}$ should be included in $\iota_1\left(\Wak^F_I(-\alpha)\right)$. Thus,
\begin{align}\label{eq:Sa-restriction}
S_\alpha \colon \iota_0\left(\Wak^F_{I,\g}\right) \rightarrow \iota_1\left(\Wak^F_I(-\alpha)\right),\quad
\alpha \in \Pi_{>0}.
\end{align}
\smallskip

From now on, we fix $\alpha\in\Pi_{>0}$. Recall that $S_\alpha$ is the 0-th mode of the intertwining operator $S_\alpha(z)$ corresponding to the vector $\hat{\rho}^R(e_\alpha)_{(-1)}|-\alpha\rangle \in \Wak^F(-\alpha)$ and
\begin{align*}
\hat{\rho}^R(e_\alpha)_{(-1)}|-\alpha\rangle = \sum_{\beta\in \Delta_{+}}P^{\beta,R}_{\alpha}(a^*)_{(-1)}a_{\beta (-1)}|-\alpha\rangle.
\end{align*}
See \eqref{scaffine eq}. Since $W^F(\lambda)^\mathrm{top}_\C = \{ P(a^*)_{(-1)}|\lambda\rangle \mid P(x) \in \C[x_\beta]_{\beta\in\Delta_+}\}$ for $\lambda \in \h^*$, we have an isomorphism
\begin{align*}
I_\lambda \colon W^F(\lambda)^\mathrm{top}_\C \xrightarrow{\sim} \C[x_\beta]_{\beta\in\Delta_+}\e^\lambda,\quad
a^*_{\beta (-1)} \mapsto x_\beta,\ 
|\lambda\rangle \mapsto \e^\lambda
\end{align*}
of vector spaces. Then
\begin{align}\label{eq:Sa-I}
I_{-\alpha} \circ S_\alpha \circ I_0^{-1} = \sum_{\beta \in \Delta_+}I_{-\alpha}(\bar{v}_\beta)\der_\beta,\quad
\bar{v}_\beta = P_\alpha^{\beta, R}(a^*)_{(-1)}|-\alpha\rangle.
\end{align}
Similarly, we have an isomorphsim $\Wak^F_I(\lambda)^\mathrm{top}_\C \simeq \C[x_\beta]_{\beta \in \Delta_{>0}}\otimes L_0(\lambda)$ of vector spaces for $\lambda \in P_{0,+}$. Thus, the restrictions of $I_0$, $I_{-\alpha}$ to $\iota_0\left(\Wak^F_{I,\g}\right)^\mathrm{top}_\C$, $\iota_1\left(\Wak^F_I(-\alpha)\right)^\mathrm{top}_\C$ respectively, yield the following isomorphisms of vector spaces:
\begin{align*}
I_0 \colon \iota_0\left(\Wak^F_{I,\g}\right)^\mathrm{top}_\C \xrightarrow{\sim}
\overline{W}_{I, \g},\quad
I_{-\alpha} \colon \iota_1\left(\Wak^F_I(-\alpha)\right)^\mathrm{top}_\C \xrightarrow{\sim}
\overline{W}_{I}(-\alpha),
\end{align*}
where
\begin{align*}
\overline{W}_{I, \g} = \C[x_\beta]_{\beta \in \Delta_{>0}},\quad
\overline{W}_{I}(-\alpha) = \C[x_\beta]_{\beta \in \Delta_{>0}}\otimes \hat{\rho}_{\g_0,-\alpha,F}\left( L_0(-\alpha) \right).
\end{align*}
Using the equations \eqref{eq:Sa-restriction} and \eqref{eq:Sa-I}, it follows that
\begin{align*}
I_{-\alpha}(\bar{v}_\beta) = (I_{-\alpha} \circ S_\alpha \circ I_0^{-1})(x_\beta) \in \overline{W}_{I}(-\alpha),\quad
\beta \in \Delta_{>0}.
\end{align*}
If $P_\alpha^{\beta, R}(x) \neq 0$ and $P_\alpha^{\beta, R}(x) \in \C[x_\beta]_{\beta \in \Deltazero}$, then $\beta \in [\alpha]$ by Lemma \ref{lem:PR zero}. Recall that $\{v_\beta\}_{\beta \in [\alpha]}$ is a basis of $L_0(-\alpha)$ satisfying \eqref{eq:L0(-alpha)}.
\begin{lemma}
For $\beta\in [\alpha]$, we have $\bar{v}_\beta\in  \hat{\rho}_{\g_0,-\alpha,F}\left( L_0(-\alpha) \right)$ and $\hat{\rho}_{\g_0,-\alpha,F}(v_\beta) = \bar{v}_\beta$.
\end{lemma}
\proof
The first assertion is immediate from the weight consideration. We will show the second assertion. Let $\widetilde{v}_\beta = \hat{\rho}_{\g_0,-\alpha,F}(v_\beta)$ for $\beta \in [\alpha]$. Using the fact that the weight space of $L_0(-\alpha)$ with the weight $-\beta$ is one-dimensional and spanned by $v_\beta$, it follows that $\bar{v}_\beta = c_{\alpha, \beta}\widetilde{v}_\beta$ with some $c_{\alpha, \beta}\in\C$. By $P_\alpha^{\alpha, R}(x) =1$, we have $\widetilde{v}_\beta = |-\alpha\rangle = \bar{v}_\alpha$. Thus $c_{\alpha, \alpha} = 1$. Since $\hat{\rho}_{F}(u) \in \Ker S_\alpha$ for $u\in\g$, we have $S_\alpha \circ \hat{\rho}_{F}(u)_{(0)} = \hat{\rho}_{F}(u)_{(0)} \circ S_\alpha$. Now, by Lemma \ref{lemma:para Wakimoto} and \eqref{eq:Sa-I},
\begin{align*}
(I_{-\alpha} \circ S_\alpha \circ \hat{\rho}_{F}(u)_{(0)} \circ I_0^{-1})(x_\beta) = \sum_{\gamma\in\Delta_{>0}}c_{\gamma, u}^\beta I_{-\alpha}(\bar{v}_\gamma),\quad
u\in\g_0,\ 
\beta \in \Delta_{>0}.
\end{align*}
On the other hand, again by Lemma \ref{lemma:para Wakimoto},
\begin{align*}
(I_{-\alpha} \circ \hat{\rho}_{F}(u)_{(0)} \circ S_\alpha \circ I_0^{-1})(x_\beta) = I_{-\alpha}(u \cdot\bar{v}_\beta),\quad
u\in\g_0,\ 
\beta \in \Delta_{>0}.
\end{align*}
Therefore $u \cdot \bar{v}_\beta = \sum_{\gamma\in\Delta_{>0}}c_{\gamma, u}^\beta \bar{v}_\gamma$ for $u\in\g_0$ and $\beta \in [\alpha]$. By the simplicity of $L_0(\alpha)$, we have $c_{\alpha, \beta} = 1$ for all $\beta \in [\alpha]$. This completes the proof.
\endproof
We see that $\bar{v}_\beta \in \iota_1\left(\Wak^F_I(-\alpha)\right)^\mathrm{top}_\C$ for all $\beta \in \Delta_{>0}$. For $\beta \in \Delta_{>0}\backslash[\alpha]$, we also denote by $v_\beta$ the vector in $\Wak^F_I(-\alpha)$ such that $\iota_1(v_\beta) = \bar{v}_\beta$. Set
\begin{align*}
&\bar{S}_\alpha^{(1)} = \sum_{\beta\in[\alpha]}\int Y(v_\beta,z)\ dz \colon \Wak^F_{I,\g} \rightarrow \Wak^F_I(-\alpha),\\
&\bar{S}_\alpha^{(2)} = \sum_{\beta\in \Delta_{>0}\backslash[\alpha]}\int Y(v_\beta,z)\ dz\colon \Wak^F_{I,\g} \rightarrow \Wak^F_I(-\alpha),
\end{align*}
where $Y(v_\beta,z)$ denotes the intertwining operator corresponding to $v_\beta$.
\begin{cor}\label{lem: identification}
$d_{I, 0} = \bigoplus_{\alpha \in \Pi_{>0}}\bar{S}_\alpha$, where $\bar{S}_\alpha = \bar{S}_\alpha^{(1)} + \bar{S}_\alpha^{(2)}$.
\end{cor}
As a consequence, we obtain the following commutative diagram from \eqref{affine parabolic dual BGG complex}:
\begin{equation}\label{affine parabolic dual BGG complex: differentials}
\SelectTips{lu}{}
\vcenter{\xymatrix@W=20pt@H=20pt@R=12pt@C=15pt{
0 \ar[r]
&V^F(\g)\ar[r]^-{\hat{\rho}_F}\ar@{=}[d]
&D_0^F\ar[r]^-{\oplus S_\alpha}
&D_1^F\ar[r]
&\cdots\ar[r]
&D^F_{\ell(w_\circ')}\ar[r]
&\cdots\\
0 \ar[r]
&V^F(\g)\ar[r]^-{\hat{\rho}_{\g_0,F}^{\g}}
&D^F_{I,0}\ar[r]^-{\oplus\bar{S}_\alpha}\ar@{^{(}-_>}[u]_-{\hat{\rho}_{\g_0,F}}
&D^F_{I,1}\ar[r]\ar@{^{(}-_>}[u]_-{\oplus \hat{\rho}_{\g_0,-\alpha,F}}
&\cdots\ar[r]
&D^F_{I,\ell(w_0')}\ar[r]\ar@{^{(}-_>}[u]
&0,
}}
\end{equation}
Here we omit $\Id\otimes\text{-}$ before $\hat{\rho}_{\g_0,F}$ and $\hat{\rho}_{\g_0,-\alpha,F}$ for simplicity. The first exact sequence is called the Wakimoto resolution of $V^F(\g)$ and the second one is called the parabolic Wakimoto resolution of $V^F(\g)$ associated with a Levi subalgebra $\g_0$.

\subsection{Resolutions for $\mathcal{W}$-algebras}\label{Resol W-alg}
By applying the cohomology functor $H_\chi^0(?)$ to \eqref{affine parabolic dual BGG complex}, we obtain the commutative diagram
\begin{equation}\label{parabolic dual BGG complex for W-algebra}
\SelectTips{lu}{}
\vcenter{\xymatrix@W=20pt@H=20pt@R=12pt@C=15pt{
0 \ar[r]
&\mathcal{W}^F(\g,f;\Gamma)\ar[r]^-{\omega_F}\ar@{=}[d]
&E^F_0\ar[r]^-{\oplus Q_\alpha}
&E^F_1\ar[r]
&\cdots\ar[r]
&E^F_{\ell(w_\circ')}\ar[r]
&\cdots\\
0 \ar[r]
&\mathcal{W}^F(\g,f;\Gamma)\ar[r]^-{[\hat{\rho}_{\g_0,F}^{\g}]}
&E^F_{0,0}\ar[r]^-{\oplus[\bar{S}_\alpha]}\ar@{^{(}-_>}[u]_-{\hat{\rho}_{\g_0,F}}
&E^F_{0,1}\ar[r]\ar@{^{(}-_>}[u]_-{\oplus \hat{\rho}_{\g_0,-\alpha,F}}
&\cdots\ar[r]
&E^F_{0,\ell(w_\circ')}\ar[r]\ar@{^{(}-_>}[u]
&0,}}
\end{equation}
where
\begin{align*}
E^F_{0,i}=\bigoplus_{\begin{subarray}{c}w\in W_0'\\ \ell(w)=i\end{subarray}}\mathbb{V}^F_0(w^{-1}\circ0)\otimes \FneF.
\end{align*}
See \eqref{W-alg Wakimoto resolution} for the definitions of $\omega_F$, $Q_\alpha$ and $E_i^F$ in the first row. The horizontal complexes are exact since the modules appearing in \eqref{affine parabolic dual BGG complex} are $H_\chi^0(?)$-acyclic by \eqref{eq:W-alg vanish} and \eqref{eq:W-Wak vanish} for the first row and by \eqref{cohomology vanishing for semi-regular bimodule} for the second row.  We call the resolution of the $\W$-algebra in the second row the parabolic Wakimoto resolution. Consider the differential $\oplus [\bar{S}_\alpha]\colon E^F_{0,0}\rightarrow E^F_{0,1}$. We have
$$E^F_{0,0}=V^F(\g_0)\otimes \Phi^F(\g_{\frac{1}{2}}),\quad
E^F_{0,1}=\bigoplus_{\alpha\in \Pi_{>0}}\mathbb{V}^F_0(-\alpha)\otimes \Phi^F(\g_{\frac{1}{2}})$$
and thus
$$[\bar{S}_\alpha]: V^F(\g_0)\otimes \Phi^F(\g_{\frac{1}{2}})\rightarrow \mathbb{V}^F_0(-\alpha)\otimes \Phi^F(\g_{\frac{1}{2}}).$$
\begin{prop}\label{Miura screening from differential}
For $\alpha\in \Pi_{>0}$, $[\bar{S}_\alpha]=\widetilde{Q}_\alpha$.
\end{prop}
\proof
By weight consideration, we have $[\bar{S}_{\alpha}]=[\bar{S}_{\alpha}^{(1)}]$ for $\alpha\in \Pi_{>0}$. Then Theorem \ref{main thm} implies the assertion.
\endproof
By Corollary \ref{Miura const lem}, we have $[\hat{\rho}_{\g_0,F}^{\g}]=\widetilde{\mu}_F$.
Thus the proof of Lemma \ref{Miura lemma} reduces to the following proposition.
\begin{prop}\label{coincidence with Miura map}
The map $[\hat{\rho}_{\g_0,F}^{\g}]$ coincides with the Miura map $\mu_F$.
\end{prop}
Before going to prove Proposition \ref{coincidence with Miura map}, we prepare some filtrations on the complexes for the functor $H_\chi^0$. Recall the settings in Section \ref{W-alg sec}: by definition $\mathcal{W}^F(\g,f;\Gamma)=H^0(C_F,d)$ and the complex $C_F=V^F(\g)\otimes F_{\mathrm{ch}}^F(\g_{>0})$ is quasi-isomorphic to the subcomplex $C_{+,F}$ defined as a vertex superalgebra generated by the fields $J^{(u)}(z)=u(z)+\sum_{\alpha,\beta\in \Delta_{>0}}c_{u,\beta}^\alpha:\varphi_\alpha(z)\varphi^\beta(z):$, $(u\in \g_{\leq0})$, $\varphi^\alpha(z)$, $(\alpha\in \Delta_{>0})$ and $\Phi_{\beta}(z)$, $(\beta\in \Delta_{\frac{1}{2}})$, see \cite{KW1}. Set $C_F^W =V^F(\g_0)\otimes \mathcal{A}_{\Delta_{>0}}^F\otimes F_{\mathrm{ch}}^F(\g_{>0})\otimes \Phi^F(\g_{\frac{1}{2}})$. Define weights on the complexes $C_F$ and $C^W_F$ by 
\begin{align*}
&\mathrm{wt}(u)=-2\mathrm{deg}(u),\quad
\mathrm{wt}(\Phi_\alpha)=0,\quad
\mathrm{wt}(\varphi^\alpha)=2\mathrm{deg}(\alpha)=-\mathrm{wt}(\varphi_\alpha),\\
&\mathrm{wt}(a_\alpha^*)=2\mathrm{deg}(\alpha)=-\mathrm{wt}(a_\alpha),\quad
\mathrm{wt}(ab)=\mathrm{wt}(a)+\mathrm{wt}(b),\quad \mathrm{wt}(\partial a)=\mathrm{wt}(a)
\end{align*}
with arbitrary fields $a, b$. Then we have the weight decompositions $C_F=\oplus_{n \in\Z}C_{F,n}$,  $C_F^W=\oplus_{n \in\Z}C_{F, n}^W$. Note that the map $\hat{\rho}_{\g_0,F}^{\g}\colon C_F\rightarrow C_F^W$ preserves these grading and the subcomplex $C_{+,F}$ is also graded as
$C_{+,F}=\oplus_{n\in\Z_{\geq0}}C_{+, F, n}$ and that the differential $d$ has the following weights:
$$\mathrm{wt}(d_{\mathrm{st}})=0,\quad \mathrm{wt}(d_{\mathrm{ne}})=1,\quad \mathrm{wt}(d_{\chi})=2.$$ 
Thus $d$ preserves the decreasing filtrations $\{F_n M=\oplus_{j\geq n}M_j\}_{n\in \Z}$ for $M=C_{+,F}, C_F^W$.
The filtration $F_\bullet C_{+,F}$ is used in \cite{G} to construct the Miura map $\mu_F$ as follows. The associated spectral sequence $\{E_q^\bullet\}_{q=1}^\infty$ converges, and we have
\begin{align}
E_1^\bullet=H^\bullet(C_{+,F},d_{\mathrm{st}})\simeq V^F(\g_0)\otimes \Phi^F(\g_{\frac{1}{2}})\otimes H^\bullet(\g_{>0},\C)
\end{align}
as complexes, where $H^\bullet(\g_{>0},\C)$ is the Lie algebra cohomology of $\g_{>0}$ with coefficients in the trivial module $\C$. Moreover, 
$$E_1^0\simeq V^F(\g_0)\otimes \Phi^F(\g_{\frac{1}{2}}),\quad E_1^1\simeq \bigoplus_{\alpha\in \Pi_{>0}}\mathbb{V}_0^F(-\alpha)\otimes \Phi^F(\g_{\frac{1}{2}})$$
as $V^F(\g_0)\otimes \Phi^F(\g_{\frac{1}{2}})$-modules and the the map $\widetilde{d}=d_{\mathrm{ne}}+d_{\chi}$ induces a differential 
$$[\widetilde{d}]: E_1^0\rightarrow E_1^1,$$
whose kernel is isomorphic to the $\W$-algebra
\begin{align}\label{Miura as spectral sequence}
\W^F(\g,f;\Gamma)\xrightarrow{\sim} \mathrm{Ker}\left([\widetilde{d}]: E_1^0\rightarrow E_1^1 \right).
\end{align}
This map coincides with the Miura map $\mu_F$, see \cite{G} for details. 

Similarly, let $\{E^{W,\bullet}_q\}_{q=1}^\infty$ be the spectral sequence associated with the filtration $F_\bullet C_F^{W}$. Define a conformal grading $G(?)$ on $C_F^W$ by 
\begin{align*}
&G(u)=1,\quad
u\in\g_0;\quad
G(\Phi_\alpha)=\frac{1}{2};\\
&G(a_\alpha)=G(\varphi_\alpha)=1-\mathrm{deg}(\alpha);\quad
G(a_\alpha^*)=G(\varphi^\alpha)=\mathrm{deg}(\alpha).
\end{align*}
Then $d$ preserves the decomposition $C_F^W=\oplus_{n\in \frac{1}{2}\Z}C_F^W[n]$ by the conformal grading. The following assertion is proved straightforwardly.
\begin{lemma}
$F_pC_F^W\cap C_F^W[n]=0$ for $p>2n$.
\end{lemma}
Then the associated spectral sequence $\{E_q^{W,\bullet}\}_{q=1}^\infty$ converges.
Note that $E_1^{W,\bullet}=H^\bullet(C_F^W,d_{\mathrm{st}})=V^F(\g_0)\otimes \Phi^F(\g_{\frac{1}{2}})\otimes H^\bullet(\mathcal{A}_{\Delta_{>0}}^F,d_{\mathrm{st}})\simeq V^F(\g_0)\otimes \Phi^F(\g_{\frac{1}{2}})$ by \cite{Fei}. Thus the spectral sequence collapses $E^{W,\bullet}_\infty=E^{W,0}_1\simeq V^F(\g_0)\otimes \Phi^F(\g_{\frac{1}{2}})$.
\begin{proof}[Proof of Proposition \ref{coincidence with Miura map}]
The map $\hat{\rho}_{\g_0,F}^{\g}\colon C_{+,F}\rightarrow C_F^W$ induces a map
$[\hat{\rho}_{\g_0,F}^{\g}]_1\colon E_1^\bullet \rightarrow E_1^{W,\bullet}$, whose zero-th part is 
$$[\hat{\rho}_{\g_0,F}^{\g}]_1\colon V^F(\g_0)\otimes \Phi^F(\g_{\frac{1}{2}}) \rightarrow V^F(\g_0)\otimes \Phi^F(\g_{\frac{1}{2}}).$$
By \eqref{Miura as spectral sequence} and the collapsing of the spectral sequences, it suffices to show that $[\hat{\rho}_{\g_0,F}^{\g}]_1$ is the identity.
Now by \eqref{parabolic Wakimoto}, we have
\begin{align*}
[\hat{\rho}_{\g_0,F}^{\g}]_1([u(z)])
&=\left[u(z)-\sum_{\alpha,\beta\in \Delta_{>0}}c_{u,\beta}^\alpha(:a_\beta^*(z)a_\alpha(z):-:\varphi_\alpha(z)\varphi^\beta(z):)\right]\\
&=[u(z)]-\left[\sum_{\alpha,\beta\in \Delta_{>0}}c_{u,\beta}^\alpha(:a_\beta^*(z)a_\alpha(z):-:\varphi_\alpha(z)\varphi^\beta(z):)\right]
\end{align*}
for $u\in\g_0$. Since the conformal grading $G$ of the second term is 1, we conclude that the second term is equal to 0 by $ H^0(\mathcal{A}_{\Delta_{>0}}^F,d_{\mathrm{st}})=\C$. Thus $[\hat{\rho}_{\g_0,F}^{\g}]_1$ is the identity map on $V^F(\g_0)$. It is obvious from the definition of $\hat{\rho}_{\g_0,F}^{\g}$ that $[\hat{\rho}_{\g_0,F}^{\g}]_1$ is also the identity map on $\Phi^F(\g_{\frac{1}{2}})$. This completes the proof.
\end{proof}

\end{document}